\documentclass[a4paper,11pt]{amsart}
\usepackage{amstext}
\usepackage{amsthm}
\usepackage{amssymb}
\usepackage{color}
\usepackage{esint}
\usepackage{dsfont}
\usepackage{upgreek}
\usepackage{amsthm}
\usepackage{enumerate}
\usepackage{amsfonts}
\usepackage{amssymb}
\usepackage{mathtools}
\usepackage{braket}
\usepackage{bm}
\usepackage{amsmath}
\usepackage{graphicx}
\usepackage{caption}
\usepackage{subcaption}
\usepackage{fullpage}
\usepackage{longtable}
\usepackage{mathrsfs}
\usepackage{mathtools}
\usepackage{xcolor}
\usepackage{tikz}

\numberwithin{equation}{section}
\graphicspath{ {./images/} }

\providecommand{\U}[1]{\protect\rule{.1in}{.1in}}
\newtheorem{theorem}{Theorem}[section]

\newtheorem{corollary}[theorem]{Corollary}

\newtheorem{definition}[theorem]{Definition}
\newtheorem{example}[theorem]{Example}

\newtheorem{lemma}[theorem]{Lemma}

\newtheorem{proposition}[theorem]{Proposition}
\newtheorem{remark}[theorem]{Remark}

\newtheorem{assumptions}[theorem]{Assumption}

\DeclareMathOperator*{\argmin}{arg\,min}
\DeclareMathOperator*{\minn}{min}
\newcommand{\divv}{\mathrm{div}}

\newcommand{\red}{\color{black}}

\newcommand{\T}{\mathcal{T}}

\definecolor{mygreen}{rgb}{0.1,0.75,0.2}

\newcommand{\nc}{\normalcolor}

\newcommand{\EE}{\mathcal{E}}
\newcommand{\ee}{\vartheta}

\newcommand{\R}{\mathbb{R}}
\newcommand{\G}{\mathcal{G}}
\renewcommand{\P}{\mathcal{P}} 
\newcommand{\veps}{\varepsilon}

\newcommand{\M}{\mathcal{M}}

\title{ Semi-discrete optimization through semi-discrete optimal transport: a framework for neural architecture search }

\author{Nicol\'as Garc\'ia Trillos and Javier Morales}
\address{Nicol\'as Garc\'ia Trillos, Department of Statistics, University of Wisconsin-Madison. 1300 University Avenue, Madison, WI, USA 53706  \\}
\email{garciatrillo@wisc.edu}
\address{Javier Morales, Center for Scientific Computation and Mathematical Modeling (CSCAMM), University of Maryland, College Park MD 20742  \\}
\email{javierm1@cscamm.umd.edu}


\begin{document}

\thanks{{\bf Acknowledgements:} N. Garc\'ia Trillos was supported by NSF-DMS 2005797. The work of J. Morales was supported by NSF grants DMS16-13911, RNMS11-07444 (KI-Net) and ONR grant N00014-1812465. Support for this research was provided by the Office of the Vice Chancellor for Research and Graduate Education at the University of Wisconsin-Madison with funding from the Wisconsin Alumni Research Foundation.}

\maketitle

	\begin{abstract}
		In this paper we introduce a theoretical framework for semi-discrete optimization using ideas from optimal transport. Our primary motivation is in the field of deep learning, and specifically in the task of neural architecture search. With this aim in mind, we discuss the geometric and theoretical motivation for new techniques for neural architecture search (in the companion work \cite{practical} we show that algorithms inspired by our framework are competitive with contemporaneous methods). We introduce a Riemannian-like metric on the space of probability measures over a semi-discrete space $\R^d \times \G$ where $\G$ is a finite weighted graph. With such Riemmanian structure in hand, we derive formal expressions for the gradient flow of a relative entropy functional, as well as second order dynamics for the optimization of said energy. Then, with the aim of providing a rigorous motivation for the gradient flow equations derived formally, we also consider an iterative procedure known as minimizing movement scheme (i.e., Implicit Euler scheme, or JKO scheme) and apply it to the relative entropy with respect to a suitable cost function. For some specific choices of metric and cost, we rigorously show that the minimizing movement scheme of the relative entropy functional converges to the gradient flow process provided by the formal Riemannian structure. This flow coincides with a system of reaction-diffusion equations on $\R^d$.
		
	\end{abstract}

	\tableofcontents

	\section{Introduction}

	Let $(\mathcal{G},K)$ be a weighted graph over the finite set $\G$ and consider the semi-discrete space $\R^d \times \G$; the function $K:\G \times \G \rightarrow [0,\infty)$ is assumed to be symmetric. In this paper we study, from geometric and variational perspectives, the system of reaction diffusion PDEs:

 \begin{align}\begin{aligned}\label{eqn:GradFlowE}\partial_{t}&f_t(x,g)=  \Delta_{x}f_{t}(x,g)+\text{div}_{x}(f_{t}(x,g)\nabla_{x}V(x,g))\\
+ & \sum_{g^{\prime}\in\mathcal{G}}\big[\log f_{t}(x,g)+V(x,g)-(\log f_t(x,g^{\prime})+V(x,g^{\prime}))\big]K(g,g^{\prime})\theta_{x,g,g'}(f_{t}(x,g),f_{t}(x,g')),
\end{aligned}
\end{align}
for $g \in \G$. In the above, $V:\R^d \times \G \rightarrow \R$ is a potential function defined on the semi-discrete space $\R^d \times \G$. The function $f_t$, i.e. the solution to the system of PDEs, is a function from $\R^d \times \G$ into $\R$ (alternatively, $f_t$ can be thought of as a collection of real valued functions on $\R^d$ indexed by $\G$), and can be interpreted as the density of a probability distribution on $\R^d \times \G $. Finally, the \textit{mobility} function $\theta_{x,g,g'}: [0,\infty) \times [0,\infty) \rightarrow [0,\infty)$ serves as ``interpolator" for the masses at the points $(x,g)$ and $(x,g')$ and in general dictates the rate at which mass can be exchanged between nodes in $\G$. 

In the first part of the paper, we provide a geometric interpretation of system \eqref{eqn:GradFlowE} by casting it as a formal gradient flow of a relative entropy functional defined on the space $\mathcal{P}(\R^d \times \G)$ of probability measures on $\R^d \times \G$ with respect to an appropriate semi-discrete optimal transport metric, this optimal transport metric is reminiscent to the Wasserstein metric in Euclidean space in its dynamic form. While the geometric interpretation that we study here is largely formal, the framework that we introduce is quite rich and allows us to give formal definitions of geodesic equations and second order dynamics in the space $\mathcal{P}(\R^d \times \G)$. 

The second perspective that we take has a variational flavor. We introduce a static optimal transport problem that serves as cost function in a minimizing movement scheme (a.k.a. JKO scheme) for the relative entropy functional $\EE$. Then, we rigorously show that for a mobility that is independent of the masses to be interpolated (i.e. $\theta_{x,g,g'}$ does not depend on $f_t(x,g)$ and $f_t(x,g')$), system \eqref{eqn:GradFlowE} can be recovered as the limit of the minimizing movement scheme as the time discretization converges to zero; see Theorem \ref{main-result} for a precise statement.

Regardless of the perspective taken, the main conceptual insight stemming from our work is that the system of equations \eqref{eqn:GradFlowE} can be interpreted as a gradient flow of relative entropy in the space of probability measures $\mathcal{P}(\R^d \times \G)$. What interests us from this interpretation is that it allows us to motivate new schemes for the optimization of an objective function of the form $V : \R^d\times \G \rightarrow \R$, with applications in machine learning such as \textit{neural architecture search} in mind (see the discussion in section \ref{sec:OpenProblems}). The discussion in the next section in the familiar Euclidean setting will help us motivate the prospects of using semi-discrete optimal transport for semi-discrete optimization; we also motivate the theoretical results that we seek in this paper by providing a brief historical background on gradient flows in the space of probability measures. Our companion paper \cite{practical} discusses more concretely how part of the theoretical framework presented in this work can be used to define scalable neural architecture search algorithms.

	\nc

	\nc

	\red

	\nc

	\subsection{Motivation from Euclidean space: Otto Calculus in $\mathcal{P}(\R^d)$} 
	\label{sec:OTEuclidean}

	Consider an optimization problem on $\R^d$ of the form
	\[\min_{x \in \R^d} V(x), \]
	where for the sake of exposition $V$ is assumed to be a nice enough differentiable function. Let us consider the following dynamics on the state space $\R^d$:
	\begin{equation}\label{rd_grad_flow}  \begin{cases} dx(t)= - \nabla_x V(x(t))dt &, \quad t >0
	\\ x(0)= x_0,
	\end{cases}\end{equation}
	\begin{equation}\label{rd_grad_flow_noise} \begin{cases} dx(t) = - \nabla_x V(x(t))dt + \frac{\sqrt{\eta}}{2}dB_t,
	&, \quad t >0
	\\ x(0)= x_0,
	\end{cases}
	\end{equation}
	\begin{equation}\label{rd_grad_flow_noiseKalman}
	\begin{cases}
	dx^j(t) = - C_t \nabla_x V( x^j(t))dt + \sqrt{2 C_t}dB_t^j, \quad t > 0 \quad j=1, \dots, J \\
	C_t:=  \frac{1}{J} \sum_{j=1}^J (x^j(t) - \overline{x}(t) ) \otimes (x^j(t) - \overline{x}(t) ).
	\end{cases}
	\end{equation}
	All of the above dynamics can be interpreted as gradient-based continuous time algorithms for the optimization of the function $V$. \eqref{rd_grad_flow} is gradient descent. \eqref{rd_grad_flow_noise} is gradient descent with Brownian noise; in principle useful to help gradient descent scape local minima. \eqref{rd_grad_flow_noiseKalman} is a preconditioned gradient descent with noise. In \eqref{rd_grad_flow_noiseKalman} multiple interacting particles are used to define the preconditioning matrix $C_t$ (in this case the running covariance matrix associated to the particles). Besides being used for the optimization of the objective $V$ defined on $\R^d$, equations \eqref{rd_grad_flow}, \eqref{rd_grad_flow_noise}, and \eqref{rd_grad_flow_noiseKalman} share a common underlying structure: they can be associated to certain gradient flows in the space of probability measures $\mathcal{P}(\R^d)$ when endowed with an appropriate optimal transport cost. In what follows we revisit this connection for \eqref{rd_grad_flow_noise} (notice that while degenerate, \eqref{rd_grad_flow} can be seen as a special case of \eqref{rd_grad_flow_noise}) and refer the interested reader to \cite{KalmanWasserstein} for details on how to interpret \eqref{rd_grad_flow_noiseKalman}. 
	
	It is well known that the law of the process $x(t)$ in \eqref{rd_grad_flow_noise} denoted $\mu_t$ solves a Fokker Planck equation of the form:
	\begin{equation}
	\dot{\mu}_t- \divv_x(\mu_t \nabla_x V )-\eta \Delta_x(\mu_t) =0 , \quad t >0,
	\label{eqn:FokkerPlanckRd}
	\end{equation}
	with initial datum $\mu_0$, where in the above $\divv_x$ is the divergence operator in $\R^d$, $\nabla_x$ the gradient operator, and $\Delta_x$ the \textit{Laplacian} operator $\Delta_x:=\divv_x\circ \nabla_x$. In general, equation \eqref{eqn:FokkerPlanckRd} must be interpreted in weak form.
	
	Mathematicians and physicists have studied Fokker Planck equations for decades, and more recently, the seminal work of \cite{J-K-O} has provided a gradient flow interpretation for these equations. This interpretation uses the setting of gradient flows in the space of probability measures endowed with the Wasserstein distance. To be more precise let us first recall the definition of the Wasserstein distance with quadratic cost for a pair of probability measures $\mu, \nu \in \mathcal{P}_2(\R^d)$ (i.e. probability measures with finite second moments): 
	\begin{equation}\label{def:Wass}
	W_{2}(\mu,\nu)^2 := \min_{\pi \in \Gamma(\mu, \nu)}\int_{\R^d \times \R^d} |x-y|^2 d\pi(x,y), 
	\end{equation}
	where $\Gamma(\mu, \nu)$ is the set of couplings between $\mu$ and $\nu$. The above definition can be thought of as describing a \textit{static} optimal transport problem, where one seeks for an optimal assignment of sources and targets of mass without specifying how said transport is actually realized dynamically in time. An alternative \textit{dynamic} reformulation due to Benamou and Brenier \cite{B-B} states that
	\[ W_{2}(\mu,\nu)^{2}=\inf _{t \in [0,1] \mapsto (\mu_t, \nabla_x \varphi_t) }\int_{0}^{1}\int_{\R^{d}}|\nabla_x \varphi_t|^{2}\hspace{1mm}d\mu_{t}dt, \]
	where the minimum is taken over all solutions $(\mu_t, \nabla_x \varphi_t)$ to the continuity equation
	\begin{equation}\label{Conteqn}
	\dot{\mu}_t + \text{div}( \mu_t \nabla_x \varphi_t) =0,
	\end{equation}
	with $\mu_0=\mu$ and $\mu_1=\nu$. The Benamou-Brenier reformulation highlights the otherwise unclear dynamic nature of the optimal transport problem \eqref{def:Wass} and it reveals a deeper geometric structure that we now discuss. First, solutions to the continuity equation $t \in [0,1] \mapsto (\mu_t, \nabla_x \varphi_t)$ which represent the different ways in which one can dynamically transport mass from $\mu_0$ to $\mu_1$ can be mathematically interpreted as curves in the space of probability measures. Here, $\mu_t$ specifies the location of a particle at time $t$ while the potential $\varphi_t: \R^d \rightarrow \R$ is interpreted as ``tangent vector" characterizing an allowed infinitesimal change to the location $\mu_t$. Second, the objective function in the Benamou-Brenier problem can be interpreted as the ``length" of a given curve (in this case a kinetic energy). A formal Riemannian metric tensor $\langle \cdot  ,\cdot \rangle_{\mu}$ can be defined according to: 
	\[\langle \varphi,\varphi^{\prime}\rangle_{\mu}:=\int_{\R^{d}}\nabla_x\varphi\cdot\nabla_x\varphi^{\prime} d\mu\] for any two potentials $\varphi, \varphi^{\prime}: \R^d \rightarrow \R$ (i.e. any two tangent vectors at $\mu$). From the above discussion one can now see that the Wasserstein distance corresponds to the \textit{geodesic distance} associated to the above formal metric tensor, and reveals that the metric space $(\mathcal{P}_2(\R^d), W_2)$ can be treated (at least formally) as a Riemannian manifold.

	%
	%

	Now, seeing $(\mathcal{P}_2(\R^d), W_2)$ as a formal Riemannian manifold allows one to give a heuristic definition for the gradient flow of a functional $\mathcal{E}$ defined on $\mathcal{P}_2(\R^d)$:   
	\begin{equation} \label{eqn:GradFlowRd}
	\begin{cases} \dot \mu(t) = - \nabla_{W_2} \mathcal{E}(\mu(t))\\ \mu(0)= \mu_0. \end{cases}  \end{equation}
	With the Fokker Planck equation in \eqref{eqn:FokkerPlanckRd} in mind let us consider the functional
	\[  \mathcal{E}(\mu)=\int_{\R^{d}}V\hspace{1mm}d\mu+\eta H(\mu),    \]
	where $H$ is the negative Shannon entropy
	\[  H(\mu)=\begin{cases}
	\int_{\R^{d}} f \log f dx & \text{ if } d\mu=f(x)dx,\\
	+\infty & \text{ othwerwise }.
	\end{cases}  \]
	In the Riemannian formalism $\nabla_{W_2}\EE(\mu)$ must be interpreted as a tangent vector to $\mu$ (i.e. a potential)  which serves as Riesz representer to the map of directional derivatives of the energy $\EE$. Namely, for an arbitrary curve $t\mapsto \mu_t \in \mathcal{P}_2(\R^d)$ which at time $t=0$ passes through $\mu$ with tangent vector $\varphi$ one must have
	\[ \frac{d}{dt}\EE(\mu_t)|_{t=0}= \langle \nabla_{W_2} \EE(\mu) ,  \varphi \rangle_{\mu}.\]
	The set of heuristic computations used to determine the gradient $\nabla_{W_2} E(\mu)$ from the above formula is nowadays widely known as Otto Calculus (see chapter 15 in \cite{Villani2009}), and in the case of the relative entropy it gives the formula:
	\[  -\nabla_{W_2}\EE(\mu)= -   V - \eta \log f  ,    \]
	for every $d\mu=f(x) dx;$ a similar computation will be presented in more detail in section \ref{sec:FormalGradient} for the semi-discrete setting explored here. Plugging the above potential back in the continuity equation, we recover the Fokker Planck equation \eqref{eqn:FokkerPlanckRd}. In other words, through heuristic arguments from Riemannian geometry that rely on the geometric structure of the optimal transport distance $W_2$, the dynamics \eqref{rd_grad_flow_noise} used for optimization of $V$ can be lifted to the space $\mathcal{P}_2(\R^d)$ where one can give a gradient flow interpretation.

	There is a second way of motivating an interpretation of \eqref{eqn:GradFlowRd} which coincides with the one coming from the Riemannian formalism. To discuss this alternative let us first consider a more general setting and let us assume that $\M$ is an arbitrary topological space, ${E}: \M \rightarrow (-\infty, \infty]$ is an objective function to optimize, $C: \M \times \M \rightarrow [0,\infty)$ is a driving cost function, and $\tau>0$ is a time step. One can then consider the \textit{minimizing movement scheme} (also known as JKO scheme)
	\begin{equation}
	\label{eqn:JKO}
	\mu_{k+1}   \in  \nc \argmin_{\mu \in \M}  {E}(\mu) + \frac{1}{2\tau } C(\mu_k,\mu)^2,
	\end{equation}
	as a discrete time scheme for optimization. 
	Under suitable conditions, in the limit $\tau\rightarrow0$ iterates \eqref{eqn:JKO} define a function in time describing what one can refer to as a ``gradient flow of ${E}$" with respect to the cost function $C$. Notice that when $\M=\R^d$ and $C$ is the Euclidean metric, the above scheme is essentially the variational formulation of implicit Euler iterates (i.e., the computation of a proximal operator for the function $E$). 
	
	When $\M = \mathcal{P}_2(\R^d),$ $C$ is the Wasserstein distance $W_2$, and $E=\EE$ is the relative entropy, the iterates $\mu_0, \mu_{1}, \dots, \mu_{k},\dots$ (where $\mu_0$ is assumed to satisfy $\EE(\mu_0)<
	\infty$) defined recursively by the JKO scheme, i.e.
	\begin{equation}
	\mu_{k+1}  \in \nc \argmin_{\mu \in \mathcal{P}_2(\R^d)} \mathcal{E}(\mu) + \frac{1}{2\tau} W_2^2( \mu_k,\mu),
	\label{def:JKO2} 
	\end{equation}
	can be shown to converge as $\tau \rightarrow 0,$ to a solution of the Fokker Planck equation \eqref{eqn:FokkerPlanckRd} (see \cite{J-K-O}). Historically, the JKO scheme \eqref{def:JKO2} was the first approach used to give a ``gradient flow" interpretation to the Fokker Planck equation \eqref{eqn:FokkerPlanckRd}. In more generality, evolution equations of the form
	\[\dot{\mu}_t={\rm {div}_x}\bigg(\nabla_x \mu_t+\mu_t\nabla_x V+\mu_t(\nabla_x U\ast\mu_t\big)\bigg),\ 
	\]
	are limits of the JKO scheme \eqref{eqn:JKO} for appropriate functionals defined on $\mathcal{P}_{2}(\mathbb{R}^{d})$ using the Wasserstein distance as cost function. The gradient flow interpretation via the minimizing movement scheme allows one to prove entropy estimates and functional inequalities (see \cite{Villani2009} for more details on this area, which is still very active and in constant evolution). \\

The minimization problem \label{eqn:JKO} can be stated in a Lagrangian form as the problem of finding

	\begin{equation}
	\mu_{k+1}  \in  \argmin_{\mu \in \mathcal{P}_2(\R^d)} \mathcal{E}(\mu) +  \mathcal{A}^{\tau}( \mu_k,\mu),
	\label{def:JKO2:lagrangian} 
	\end{equation}
where $\mathcal{A}^{\tau}( \mu_k,\mu)$ denotes the action of the curve in the tangent bundle of $(\mathcal{P}_2(\R^d), W_2)$ with minimal kinetic energy connecting $\mu_k$ and $\mu$ in $\tau$ units of time.\\

\color{black}

In summary, the gradient based dynamics \eqref{rd_grad_flow_noise} used for optimization of an objective $V$ defined on the state space $\R^d$ are closely linked to a gradient flow on the space of probability measures $\mathcal{P}(\R^d)$. This gradient flow can be motivated using either the formal Riemannian structure that the dynamic formulation of optimal transport has, or the minimizing movement scheme with driving cost taken to be the Wasserstein distance (given that the two interpretations coincide).

	\nc

\subsection{Outline}  We organize the rest of the paper as follows. In section \ref{sec:Semi-discreteOT} we introduce the main objects studied in the paper and state our main results precisely. We start in section \ref{subsec:DiffOperatorsGraphs} introducing the basic analytical objects on graphs used throughout the paper. In section \ref{sec:Riemman} we introduce a family of distances on the space of probability measures over $\R^d \times \G$ based on a dynamic formulation of optimal transport. We highlight the formal Riemannian structure of the metric introduced and explore the connections between our definition and the literature on discrete optimal transport. In section \ref{sec:FormalGradient} we use the Riemannian formalism from section \ref{sec:Riemman} in order to motivate a definition for the gradient flow of a relative entropy energy closely related to the objective function in the semi-discrete optimization problem of interest. In section \ref{sec:Hamiltonian} we use the Riemannian formalism once again and motivate a method for optimization of the relative entropy. In section \ref{sec:main-results} we provide concrete theoretical support for the formal definitions and computations presented in the earlier sections. In particular, we state our main theoretical result, which establishes a connection between the formal definitions from section \ref{sec:FormalGradient} and the minimizing movement scheme discussed in the introduction. To realize the JKO scheme we introduce a new cost that can be interpreted as a \textit{static} semi-discrete optimal transport cost.

Section \ref{ap:heuristic} explores metric and geometric properties of the transport distances introduced in section \ref{sec:Riemman} (i.e. the dynamic semi-discrete transport problems). More specifically, in section \ref{ap:TheoremDistance} we prove that these ``distances" are indeed metrics. Section \ref{sec:tangentplanes} aims at providing concrete and rigorous support for the heuristic discussion in section \ref{sec:Riemman}. The discussion in this section motivates more concretely (and rigorously) the characterization of tangent planes of the space of probability measures over $\R^d \times \G$. Section \ref{sec:geodesic} presents some heuristic computations justifying the definition of the accelerated method for optimization presented in section \ref{sec:Hamiltonian}.

Section \ref{sec:optimal_maps} studies the static semi-discrete transport problem introduced in section \ref{sec:main-results}. This section is used later on in the paper, but is also of independent interest. We establish a characterization for solutions to the static semi-discrete optimal transportation problem that is analogous to the celebrated result by Brenier characterizing solutions to the quadratic (Euclidean) optimal transport problem.   

Section \ref{sec:Prelim} studies properties of the variational problem used to define the JKO scheme relative to the static semi-discrete cost. We provide a full characterization of solutions to this variational problem. We also establish a maximum principle that is characteristic of Fokker Plank equations.

In section \ref{JKO_proof_sec} we put together the results proved in sections \ref{sec:optimal_maps} and \ref{sec:Prelim} and prove our main theoretical result Theorem \ref{main-result}, i.e. we show the convergence of the JKO scheme proposed in section \ref{sec:main-results}.

We wrap up the paper in section \ref{sec:OpenProblems} where we provide some conclusions, perspective on future research directions, and discussion on some of the applications in machine learning that have motivated this work.

\textbf{Note:} Throughout the paper some computations will be carried out at a formal level. One of our aims is to stress the importance of the intuition emanating from the formal Riemannian structure that the dynamic formulation of optimal transport has. After all, it is this Riemannian formalism that motivates the algorithms that are implemented in our companion paper \cite{practical} for the purposes of neural architecture search (including accelerated methods). The formal computations (or heuristic arguments) that we present here are, for the most part, accompanied by rigorous counterparts.

\section{Semi-discrete optimal transport and gradient flows}
\label{sec:Semi-discreteOT}

\subsection{Some differential operators on graphs}
\label{subsec:DiffOperatorsGraphs}

In this section we introduce the discrete differential operators that will later be used to introduce a semi-discrete optimal transport problem on $\R^d \times \G$.

Throughout the paper we assume that $(\G , K)$ is connected, meaning that for every $g,g' \in \G$ there exists a path $g_0, \dots, g_m \in \G$ with $g_0=g$, $g_m = g'$ and $K(g_l, g_{l+1})>0$ for every $l=0,\dots, m-1$. 
\nc

Given a function $\phi: \G \rightarrow \R$ we define its \textit{discrete gradient} as the function $\nabla_g \phi: \G \times \G \rightarrow \R$ 
\[ \nabla_g \phi(g,g') := \phi(g') - \phi(g). \]
We use the subscript $g$ in $\nabla_g$ to distinguish the discrete gradient from the gradient of a function defined on $\R^d$ (where we use the notation $\nabla _x$). This distinction will become important later on when we consider functions $\phi: \R^d \times \G \rightarrow \R$ for which we can compute its gradient $\nabla_x$ as well as its discrete gradient $\nabla_g$. 


Given a function $h: \G \times \G \rightarrow \R$ (i.e. a discrete vector field) we define its \textit{discrete divergence} as the function $\divv_g h: \G \rightarrow \R$ defined by
\[ \divv _g \hspace{1mm} h (g) := \sum_{g'} (h(g,g')- h(g',g))K(g,g').  \]

Discrete gradients and discrete divergences are related to each other via a discrete integration by parts formula.  Namely, a straightforward computation shows that for every $h: \G \times \G \rightarrow \R$ and $\phi: \G \rightarrow \R$ it holds
\begin{equation}
 \sum_{g} \divv_g(h)(g) \phi(g)  = - \sum_{g,g'}h(g,g')\nabla_g \phi(g,g')K(g,g'). 
 \label{eqn:DiscreteIntegrationParts}
\end{equation}
In particular if $h$ is of the form $h = \nabla_g \psi \cdot S $ (where $\cdot$ is interpreted as a coordinatewise product) for some $S: \G \times \G \rightarrow \R$, then 
\begin{equation}
 \sum_{g} \divv_g( \nabla_g \psi \cdot S)(g) \phi(g)  =- \sum_{g,g'} \nabla_g \phi \cdot \nabla_g \psi S(g,g')K(g,g').
 \label{eqn:DiscreteIntegrationParts2}
\end{equation}

\red  
In the remainder we use the following result establishing existence and uniqueness of solutions to elliptic graph PDEs.
\begin{proposition}
\label{lemm:GraphPDE}
Suppose that the graph $(\G, K)$ is connected. Let $\phi: \mathcal{G} \rightarrow \R$ be such that
\[ \sum_{g}\phi(g) =0,\]
and let $S : \G \times \G \rightarrow [0,\infty)$ be a symmetric function which is strictly positive whenever $K(g,g')>0$. Then, there exists a unique solution $\eta: \G \rightarrow \R$ to the graph PDE
\begin{equation}
\divv_g(\nabla_g \eta \cdot S  )=\phi
\end{equation}
satisfying
\[ \sum_{g} \eta(g)=0. \]
Moreover, 
\[  \sum_{g,g'}|\nabla_g \eta(g,g')|^2S(g,g')K(g,g') \leq \frac{1}{\lambda_S} \sum_{g} |\phi(g)|^2,     \]
where $\lambda_{S}$ represents the first non-zero eigenvalue of the graph Laplacian matrix $L_S$ with entries:
\[ L_{S}(g,g'):= \mathds{1}_{g=g'}\sum_{g''}2S(g,g'')K(g,g'') - 2S(g,g')K(g,g'). \]
\end{proposition}
\begin{proof}
The graph PDE can be written in matrix form as
\[   L_{S} \eta = -\phi, \]
where $\phi$ and $\eta$ are interpreted as vectors whose coordinates are indexed by the elements in $\G$, and where the matrix $L_S$ is the (unnormalized) graph Laplacian for a weighted graph $(\G, \omega)$ with weights $\omega_{g,g'}:= 2S(g,g') K(g,g')$--see \cite{Chung1996} for the definition of graph Laplacians. The assumptions on $S$ guarantee that the graph $(\G, \omega)$ is connected and thus its graph Laplacian $L_S$ is a positive semi-definite matrix with zero eigenvalue of multiplicity one. The assumption on $\phi$ guarantees that it belongs to the orthogonal complement of the null space of $L_S$, and thus is an element of the range of $L_{S}$. We conclude that the graph PDE indeed has a unique solution $\eta$ with average zero.

Finally, according to \eqref{eqn:DiscreteIntegrationParts2}, 
\[    \sum_{g,g'}|\nabla_g \eta(g,g')|^2S(g,g')K(g,g') =  \sum_{g} -\divv_{g}(\nabla_g \eta S)\eta(g) =  -\sum_{g} \phi(g)\eta(g) = \sum_{g} L_S\eta(g) \eta(g),       \]
and thus from Cauchy-Schwartz inequality it follows that
\[ \sum_{g,g'}|\nabla_g \eta(g,g')|^2S(g,g')K(g,g') \leq  \left(\sum_{g} |\phi(g)|^2 \right)^{1/2} \left(\sum_{g}|\eta(g)|^2 \right)^{1/2}.\]
From the fact that the graph $(\G,\omega)$ is connected it follows that
\[  \sum_{g} |\eta(g)|^2   \leq \frac{1}{\lambda_S} \sum_{g} L_S\eta(g) \eta(g),      \] 
where $\lambda_S$ is the first non-zero eigenvalue of $L_S$. Combining the above two inequalities we obtain the desired result. 

\end{proof}

\nc

\subsection{A Riemannian structure for semi-discrete OT}
\label{sec:Riemman}

Let us denote by $\mathcal{P}_2(\R^d \times \G)$ the space of Borel probability measures on $\R^d \times \G$ with finite second moments. In this section we introduce a metric $W_2$ on $\mathcal{P}_2(\R^d \times \G)$ which can be formally interpreted as the geodesic distance associated to a formal Riemannian structure on $\mathcal{P}_2(\R^d \times \G)$. Viewing $\mathcal{P}_2(\R^d \times \G)$ as a Riemannian manifold, in section \ref{sec:FormalGradient} we will be able to give a concrete heuristic interpretation for the gradient descent equation:
\begin{equation} 
\label{eqn:GradFlow}
\begin{cases} \dot \mu(t) = - \nabla_{W_2} \mathcal{E}(\mu(t))\\ \mu(0)= \mu_0, \end{cases}  \end{equation}
for a conveniently chosen function $\mathcal{E}$ on $  \mathcal{P}_2(\R^d \times \G) $ that depends on the objective function $V$ in \eqref{eq:semidiscreteopti}. Here $t \mapsto \mu_t$ describes a path in the space $\mathcal{P}_2(\R^d \times \G)$.


\subsubsection{A dynamic optimal transport problem in $\mathcal{P}_2(\R^d \times \G)$.}
Motivated by the (Euclidean) Otto Calculus discussed in section \ref{sec:OTEuclidean}, in order to define an optimal transport problem in the semi-discrete setting, we first introduce an appropriate notion of continuity equation. As in the Euclidean case, semi-discrete continuity equations are used to describe paths in the space $\mathcal{P}_2(\R^d \times \G)$. 

The definition of a semi-discrete continuity equation depends on the choice of a \textit{mobility function} $\theta$ which in full generality is a function of the form 
\[ \theta : \R^d \times \G \times \G \times \R_+\times \R_+ \longrightarrow \R_+.\]
In the remainder we will often write $\theta_{x,g,g'}(s,t)$ and drop the subscripts when no confusion may arise from doing so. The mobility function is used to quantify how easy it is to move mass from a point $(x,g)$ to a point $(x,g')$ when the amount of mass at each of these points is $s$ and $t$ respectively.  Mobilities as described above are motivated by the literature on discrete optimal transport. See \cite{Zhou,Maas,Mielke_2011,Mielke2013} where discrete optimal transport was first introduced and \cite{Erbar2012,Fathi,Esposito2019NonlocalinteractionEO} for other references where the topic has been developed further. A rigorous passage to the limit from discrete OT to OT in $\R^d$, at least for certain classes of geometric graphs, has been explored in \cite{gigmaas,GarcaTrillos2020,HomogeneizationOT,ScalinOT}.

Throughout the paper we will make the following assumptions on $\theta$. 
These assumptions are closely related to those in \cite{Erbar2012,Maas} for discrete OT.

\begin{assumptions}
\label{assump:theta}
The mobility function $\theta$ satisfies either:
\begin{enumerate}
    \item[(A0)] $\theta$ is non-zero, does not depend on $s,t$ and satisfies the symmetry condition: $\theta_{x,g,g'}$ is equal to $\theta_{x,g',g}$ for all $x\in \R^d$, $g,g' \in \G$. In addition, $\theta_{x,g,g}$ is uniformly bounded away from zero on compact sets of $\R^d \times \G \times \G$.  
\end{enumerate}

\textbf{or} all of the following

\begin{enumerate} 
		\item[(A1)] Symmetry: $\theta_{x,g,g'}(s,t) = \theta_{x,g,g'}(t,s) $ for all $s,t$.
		\item[(A2)] Differentiability: The function $\theta_{x,g,g'}(\cdot, \cdot)$ is differentiable.
		\item[(A3)] Monotonicity: $\theta_{x,g,g'}(r,t) \leq \theta_{x,g,g'}(s,t)$ for all $r \leq s$ and all $t$.
		\item[(A4)] Positive homogeneity: $\theta_{x,g,g'}(\lambda s , \lambda t) = \lambda \theta_{x,g,g'}(s, t)$ for all $\lambda \geq 0$ and all $s,t$.
		\item[(A5)] The quantity
		\[ C_{x,g,g'}:= \int_{0}^1 \frac{1}{\sqrt{\theta_{x,g,g'}(1-t, t)}} dt ,\]
		is uniformly bounded above on compact subsets of $\R^d \times \G \times \G$, and the quantity $\theta_{x,g,g'}(1,1)$ is uniformly bounded away from zero on compact subsets of $\R^ d \times \G \times \G$.
	\end{enumerate}
	\end{assumptions}

\begin{definition}\label{def:SemiDiscreteContEq}
In what follows, we consider $v_t: \R^d \times \G \rightarrow \R^d$, $h_t : \R^d \times \G \times \G \rightarrow \R$ and $\mu_t \in \mathcal{P}_2(\R^d \times \G)$. We say that $t\in[0,T]\mapsto(\mu_{t},v_{t},h_{t})$ satisfies the semi-discrete continuity equation and write
\begin{equation}
\label{def:SemiDiscreteCENota}
\dot{\mu}_{t}+\divv_{x}(v_{t}\mu_{t})+\divv_{g}(h_{t}\mu_{t})=0,
\end{equation}
if for all smooth test functions $\zeta \in C_c^\infty(\R^d \times \G)$  (i.e. $\zeta(\cdot, g)$ is $C^\infty_c(\R^d)$ for all $g \in \G$) \nc we have
\begin{align}\begin{aligned}\label{def:SemiDiscreteContEq_{0}}\frac{d}{dt}\int_{\R^d}\sum_{g} \zeta(x,g)d\mu_{t} & =\int_{\mathbb{R}^{d}}\sum_{g}\nabla_{x}\zeta(x,g)\cdot v_{t}(x,g)d\mu_{t}\\
 & +\int_{\mathbb{R}^{d}}\sum_{g,g^{\prime}} \nabla_g \zeta(x,g,g') h_{t}(x,g,g^{\prime}) d\hat{\mu}_{t}(x,g,g^{\prime}).
\end{aligned}
\end{align}

In the above expression, for a given $\mu \in \mathcal{P}_2(\R^d \times \G),$ we use $\hat{\mu}$ to denote the measure on $\R^d \times \G \times \G$ given by
\[ d \hat{\mu}(x,g,g') = \theta_{x,g,g'} dxdgdg'  \]
when $\theta$ satisfies $(A0)$ in Assumption \ref{assump:theta} and
\[ d \hat{\mu}(x,g,g') =  \theta(\mu_{g|x}(g), \mu_{g|x}(g')) d \mu_x(x)dg dg'.\]
when $\theta$ satisfies $(A1)-(A5)$ instead. Here $\mu_{g|x}$ denotes the conditional distribution of $g$ given $x$.  Also, here and in the remainder $dg$ represents the measure on $\G$ that gives mass one to every element of $\G$.\nc
\end{definition}

\begin{remark}
We notice that when $\mu$ has a density with respect to $dxdg$, i.e.,
\[ d\mu(x,g)= f(x,g) dxdg, \]
then
\[ d\hat{\mu}(x,g,g')= \theta(f(x,g), f(x,g'))dx dg dg'. \]
Indeed, this is immediate if $\theta $ satisfies $(A0)$ and otherwise follows from the homogeneity of the mobility $\theta$, i.e. condition $(A4)$.
\end{remark}

\begin{remark}
Let $t\in[0,T]\mapsto(\mu_{t},v_{t},h_{t})$ be a solution to the semi-discrete continuity equation and suppose that for every $t$,  $\mu_{t}$ is absolutely continuous with respect to $dxdg$ and has density  $f_t(x,g).$ Additionally, suppose that the mappings $(t,x,g)\mapsto f(t,x,g)$, $(t,x,g) \mapsto v_t(x,g,g')$ and $(t,x,g) \mapsto h_t(x,g,g')$  are all smooth. In that case we can see that for every test function $\zeta\in C_{c}^{\infty}(\R^{d}\times\G)$ we have 
\begin{align*}\begin{aligned}
&\int_{\R^{d}}  \sum_{g}\zeta(x,g)\frac{\partial}{\partial t}f_{t}(x,g)dx=\frac{d}{dt}\int_{\R^{d}}\sum_{g}\zeta(x,g)d\mu_{t}(x,g)\\
 & = \int_{\R^{d}}\sum_g\nabla_x\zeta(x,g)\cdot v_{t}(x,g)d\mu_{t} +\int_{\R^{d}}\sum_{g,g'} \nabla_g \zeta(x,g,g')h_t(x,g,g')K(g,g')\theta(f_t(x,g), f_t(x,g') )dx
 \\ & = -\int_{\R^{d}} \sum_g\zeta \divv_x(  v_{t} f_t )dx -\int_{\R^{d}}\sum_{g} \zeta \divv_g(h_t \cdot \hat{f}_t) dx,
\end{aligned}
\end{align*}
where $\hat{f}_t(x,g,g'):= \theta(f_t(x,g), f_t(x,g'))$. The last equality follows using integration by parts in $x$ for the first term and in $g$ for the second term (i.e. identity \eqref{eqn:DiscreteIntegrationParts}). We conclude that 
	\[\frac{\partial}{\partial t}f_t+\text{\normalfont{div}}_{x}\hspace{1mm}( v_{t} f_t) + \text{\normalfont{div}}_{g}\hspace{1mm}(h_{t} \hat{f})(x,g)=0,\quad\forall t,x,g.
\]
which justifies the notation \eqref{def:SemiDiscreteCENota} used in Definition \ref{def:SemiDiscreteContEq}.
	\end{remark}

With the above notion of continuity equation in hand, we are now able to introduce the following dynamic optimal transport problem.

 \begin{definition} \label{def:SemiDiscreteOTdist}
Let $\mu_{0}$ and $\mu_{1}$ be two elements in $\P_2(\mathbb{R}^{d} \times \mathcal{G})$.
We define 
\begin{align}\begin{aligned}\label{kinetic_E} W_2(\mu_{0},\mu_{1})^{2}  & :=\inf_{t \in[0,1] \mapsto (\mu_{t},\nabla_x \phi_t,\nabla_g \psi_t)}  \int_{0}^{1} \biggr(  \int_{\R^d} \sum_{g\in\mathcal{G}}  | \nabla_x \phi_t (x,g)|^{2}d\mu_{t}(x,g)
\\ & +\int_{\R^{d}}\sum_{g,g^{\prime}} (\nabla_g \psi_t(x,g,g^{\prime})){}^{2}K(g,g^{\prime})d\hat{\mu}_{t}(x,g,g^{\prime}) \biggr) dt,\nc
\end{aligned}
\end{align}
\nc where the infimum is taken among all solutions to the semi-discrete continuity equation of the form $t \in [0,1] \mapsto (\mu_t, \nabla_x \phi_t,  \nabla_g \psi_t)$, where $\phi_t: \R^d \times \G \rightarrow \R$ and $\psi_t: \R^d \times \G \rightarrow \R$.
\end{definition}

In words, $W(\mu_0,\mu_1)^2$ is obtained by minimizing the \textit{total kinetic energy} associated to paths connecting $\mu_0$ and $\mu_1$. In section \ref{ap:heuristic} we rigorously show that $W_2$ as defined above, is indeed a metric on the space $\mathcal{P}_2(\R^d \times \G)$. The precise statement is the following. 
\begin{theorem}
\label{thm:Dist}
Let $(\G, K)$ be a connected weighted graph, where $K$ is a symmetric weight matrix with non-negative entries. Suppose that the mobility function $\theta:   \R^d \times \G \times \G \times \R \times \R \rightarrow [0,\infty)$ satisfies Assumptions \ref{assump:theta}. Then,
$W_2$ as introduced in Definition \ref{def:SemiDiscreteOTdist} is a metric on the space $\mathcal{P}_2(\R^d \times \G)$.
\end{theorem}

\begin{remark}
In the above definition we have introduced the semi-discrete Wasserstein distance as an optimization problem over a specific class of solutions to the continuity equation, namely, solutions whose driving vector fields are gradients of potentials. It is actually possible to show that removing the restriction to this smaller class of vector fields does not change the definition given. We have introduced $W_2$ in this way for convenience.  

Later on we will show that the class of vector fields can actually be restricted even further (at least for regular enough measures). In particular the potentials $\phi$ and $\psi$ may be taken to be the same. This observation will be useful when interpreting $\mathcal{P}_2(\R^d \times \G)$ as a formal Riemannian manifold with geodesic distance that coincides with $W_2$.     
\end{remark}

\begin{remark}
The definition given in \eqref{def:SemiDiscreteOTdist} is a particular case of the formal definition given in \cite{Mielke_2011}. A closely related construction is also explored in \cite{ChenVectorValued} under the name of ``vector-valued optimal transport", as well as in \cite{ChenVectorValuedComputational}, which introduces an algorithm used to solve applied problems in color image processing and multi-modality imaging. Here we present some heuristic computations providing a characterization of tangent planes (see the informal Theorem \ref{single_potantial} below and its rigorous counterpart in section \ref{sec:tangentplanes}), and a formal computation of the acceleration of curves which in turn motivates: 1) geodesic equations, and 2) accelerated methods for optimization (see sections \ref{sec:Hamiltonian} and \ref{sec:geodesic}).
\end{remark}

\nc

\subsubsection{A formal Riemannian structure for $(\mathcal{P}_2(\R^d \times \G), W_2)$}



In differential geometry, when working in the setting of a smooth manifold $\M$, a tangent vector at a given point $q$ is interpreted as the velocity of a curve in $\M$ when passing through $q$. The collection of tangent vectors at $q$, i.e. $q$'s tangent plane, is typically denoted by $\T_q \M$. When $\M$ is endowed with a Riemannian structure, one can compute inner products $\langle p , \tilde p \rangle_q$ between elements $p, \tilde p \in T_q\M$ and introduce a notion of distance between points $q, \tilde q \in \M$ according to
\[ d(q, \tilde q)^2 := \inf_{t \in [0,1] \mapsto q(t) } \int_{0}^1 \langle \dot{q}(t) ,\dot{q}(t) \rangle_{q(t)} dt,\]
where the infimum ranges over all paths connecting $q$ to $\tilde q$.

We now provide some heuristics that motivate how the space $(\mathcal{P}_2(\R^d \times \G), W_2)$ can actually be interpreted in light of this Riemannian formalism. The first step is an informal statement that will justify some of the subsequent discussion. A precise (and rigorous) version will be presented in section \ref{sec:tangentplanes}.

\begin{theorem}\textbf{Characterization of potentials (informal)}\label{single_potantial} Let $t\rightarrow\mu_{t},$ be an arbitrary curve in $\mathcal{P}_2(\R^d \times \G)$ with velocity fields generated by the potentials $(\phi_{t},\psi_{t})$. Then, we can replace the potentials with a pair of the form $(\varphi_{t},\varphi_{t})$ such that it acts as a velocity field for the same curve $t\rightarrow\mu_{t},$ and has minimal total kinetic energy.
\label{TangentPlanes}
\end{theorem}

The above suggests that there is some redundancy when considering different potentials $\phi, \psi$ and actually one may take both potentials to be the same. Indeed, such a characterization allows us to formally identify the tangent plane at a measure $\mu$ in $\mathcal{P}_{2}(\mathbb{R}^{d}\times\mathcal{G})$ as:
\begin{align}\begin{aligned}\label{t_plane} T_{\mu}\mathcal{P}_{2}(\mathbb{R}^{d}\times\mathcal{G}) & :=\bigg\{\varphi\hspace{1mm}:\hspace{1mm}\int_{\mathbb{R}^{d}\times\mathcal{G}}|\nabla_{x}\varphi|^{2}d\mu(x,g)\\
 & +\int_{\mathbb{R}^{d}\times\mathcal{G}\times\mathcal{G}}[\varphi(x,g^{\prime})-\varphi(x,g)]^{2}K(g,g^{\prime})d\hat{\mu}(x,g,g^{\prime})<\infty\bigg\}.
\end{aligned}
\end{align}
endowed with the inner product:
\begin{align}
\label{def:InnerProduct}
\begin{split}
\langle  \varphi ,  \widetilde{\varphi}  \rangle_{\mu}:= &\int_{\R^d } \sum_g\nabla_{x}\varphi(x,g)\cdot \nabla_{x}\widetilde{\varphi}(x,g)d\mu(x,g)
\\ & +\int_{\mathbb{R}^{d}} \sum_{g,g'}\nabla_g \varphi \cdot \nabla_g  \tilde {\varphi} K(g,g') d\hat{\mu}(x,g,g^{\prime}).
\end{split}
 \end{align}
A rigorous definition of the tangent plane is out of the scope of this work. Putting aside all technicalities, we can observe formally that the semi-discrete Wasserstein distance $W_2$ from Definition \eqref{def:SemiDiscreteOTdist} can be rewritten as 
\[ W_{2}^{2}(\mu_{0},\mu_{1})=\inf\int_{0}^{1}\langle \varphi_{t},\varphi_{t}\rangle_{\mu_{t}}dt, \]
where the inf ranges over solutions to the continuity equation $t \in [0,1]\mapsto (\mu_t, \nabla_x \varphi_t, \nabla_g \varphi_t) $ connecting $\mu_0$ with $\mu_1$ (i.e., over paths in $\mathcal{P}_2(\R^d \times \G)$, according to the informal Theorem \ref{TangentPlanes}): this formula and its interpretation reveal the Riemannian structure of the metric $W_2$.  In the next subsection we use this Riemannian formalism to motivate a concrete interpretation for \eqref{eqn:GradFlow}.

\subsection{Computation of gradient flows using the Riemannian formalism}
\label{sec:FormalGradient}

In this section we use the Riemannian formalism for $(\mathcal{P}_2(\R^d \times \G), W_2)$ discussed in the previous section to motivate a definition for the gradient of a given energy function $\EE:\mathcal{P}_{2}(\R^{d}\times\G)\rightarrow\R\cup\{\infty\}$, and ultimately give a concrete meaning to the gradient flow ODE \eqref{eqn:GradFlow}. Looking forward to our applications, here we will focus on energies of the form
\begin{equation}\label{def:RelativeEntropy}\mathcal{E}(\mu):=\begin{cases}
 & \int_{\mathbb{\mathbb{R}}^{d}}\sum_g \ee(f(x,g),x,g)\hspace{1mm}dx\hspace{1em}\textnormal{if}\hspace{1em}d\mu(x,g)=f(x,g)\hspace{1mm}dxdg\\
 & +\infty\hspace{1em}\textnormal{otherwise},
\end{cases}\ 
\end{equation}
where $\ee:[0,\infty)\times{\mathbb{R}^{d}} \times \G \rightarrow\mathbb{\mathbb{R}}$ is given by
\[\ee(r,x,g):=r\log r+V(x,g)r.\ \ 
\]
We think of the function $V:\R^{d}\times\G\rightarrow\R$ as the objective of the semi-discrete optimization problem \eqref{eq:semidiscreteopti}. Here we assume for simplicity that $V$ is differentiable in the $x$ coordinate. Notice that $\EE$ is a relative entropy and can be written as the sum of the two terms
\[\mathcal{E}(\mu)=H(\mu)+\int_{\R^{d}\times\G}V(x,g)d\mu(x,g),\ 
\]
where $H$ denotes the (negative) entropy of $\mu$ when the base measure on $\R^{d}\times\G$ is the product measure $dx dg$. The entropy term $H$ may be multiplied by a positive factor for generality without that entailing any meaningful changes in the computations below. This choice of energy is motivated by the discussion presented in section \ref{sec:OTEuclidean}.

Let us recall that in Riemannian geometry, the gradient of a differentiable function $E: \M \rightarrow \R$ at a point $q $ is defined as a tangent vector $\nabla_\M E(q)$ at $q$ characterized by: for every smooth curve $t\in(-\veps, \veps) \mapsto q(t) \in \M$ with $q(0)=q$,
\[   \langle  \nabla_\M E(q) ,  \dot{q}(0)  \rangle_q = \frac{d}{dt} E(q(t)) \biggr \rvert_{t=0}.  \]
In words, the above means that the gradient of a given function $E$ at a given point $q$ on the Riemannian manifold $\M$ serves as Riesz representer (with respect to the inner product at that point) for the map of directional derivatives of the function $E$ at the point $q$.

Using the above discussion as motivation we notice that for arbitrary $\mu\in\mathcal{P}_{2}(\R^{d}\times\G)$ such that $\EE(\mu)<\infty$, the gradient of $\EE$  (with respect to $W_2$) at the point $\mu$ must be interpreted as a potential $\varphi_\mu$. Our goal is to identify $\varphi_\mu$. In order to achieve this, we consider $t\in(-\veps,\veps)\mapsto({\mu}_{t}, \nabla_x \psi_t, \nabla_g \psi_t)$ an arbitrary curve in $\mathcal{P}_{2}(\R^{d}\times\G)$ which at time $t=0$  passes through the point $\mu$  (i.e. ${\mu}_{0}=\mu$ ). We assume $d \mu_t= f_t dxdg$ and write $f=f_0$. We want $\varphi_\mu$ to satisfy
\begin{equation}\label{def:RiemannianGradientSemiDiscrete}\langle\varphi_\mu,\psi_0\rangle_{\mu}=\frac{d}{dt}\mathcal{E}(\mu_{t}) \biggr \rvert_{t=0}.
\end{equation}
A formal computation shows that
\begin{align*}\begin{split}\frac{d}{dt}\mathcal{E}( \mu_{t}) \biggr \rvert_{t=0} 
 & =\frac{d}{dt}\bigg|_{t=0}\int_{\R^{d}} \sum_{g}\bigg(\log f_{t}+V\bigg)f_{t}dx\\
 & =\int_{\R^{d}} \sum_g\bigg(\log f_0+1+V\bigg)\partial_t {f}_{0}(x,g)dx.\\
\end{split}
\end{align*}
Using the semi-discrete continuity equation the last line can be rewritten as
\begin{align*}\begin{aligned}
&\int_{\R^d } \sum_g \nabla_{x}(\log f  +V)\cdot\nabla_{x}\psi_0 d\mu(x,g) +\int_{\mathbb{R}^{d}} \sum_{g,g^{\prime}} \nabla_g( \log f + V  ) \cdot \nabla_g \psi_0 K(g,g') d\hat{\mu}(x,g,g^{\prime}),\\
\end{aligned}
\end{align*}
which in turn can be rewritten as $\langle \log f + V , \psi_0 \rangle_\mu$. It follows that $\varphi_{\mu}$ can be taken to be
\begin{equation}\label{formal_gradient} \nabla_{W_2} \EE(\mu) := \varphi_{\mu}=\log f +V .\ 
\end{equation}

Having found the gradient of $\mathcal{E}$  through the above heuristic computations we can now give a concrete interpretation to \eqref{eqn:GradFlow} by plugging in the potential $-(\log f + V)$ in the semi-discrete continuity equation. In particular, $t\in[0,\infty)\rightarrow\mu_{t}$  in \eqref{eqn:GradFlow} is interpreted as 
\[d\mu_{t}(x,g)=f_{t}(x,g)dxdg, 
\]
where $f_t$ follows \eqref{eqn:GradFlowE}. Equation \eqref{eqn:GradFlowE} can be described as a coupled system of \textit{reaction-diffusion} equations indexed by $g\in\G$. The presence of the last term in \eqref{eqn:GradFlowE} is responsible for the coupling of the dynamics. From the transport point of view this coupling term induces mass to be exchanged between different nodes (and thus the total mass at a single $g\in\G$  changes in time). From the optimization point of view, a coupled system implies that information on the optimization over parameters $x$  for a given node $g$ is  used for the optimization of parameters $x$ for nearby nodes $g'$ and vice versa.


We finish this section with two examples of mobility functions $\theta$ and their corresponding gradient flows.

\begin{example}\label{exam:MobilityMassIndep}
Let $W:\mathbb{R}^{d}\rightarrow\mathbb{R}$ be a function in the Sobolev space $W^{1,2}(\mathbb{R}^{d})$ satisfying     
\begin{equation}\label{mob_2moments}\int|x|^{2}e^{-W}dx<\infty.
\end{equation}
We define a mass independent mobility $\theta$  according to
\[\theta_{x,g,g'}(s,t):=e^{-W(x)}.\ 
\]
This mobility function satisfies $(A0)$ in Assumptions \ref{assump:theta}. We notice that in the corresponding optimal transport problem from definition \eqref{def:SemiDiscreteOTdist} the transfer of mass between points $(x,g)$ and $(x,g')$ is cheap precisely when $W(x)$ is large. We also notice that the cost of transporting mass along the graph $\G$  does not depend on the actual amount of mass that is initially located at the nodes of $\G$, a situation that contrasts with the one presented in the next example.

Finally, for this choice of mobility the system of equations \eqref{eqn:GradFlowE} becomes the system of non-linear reaction diffusion equations:
\begin{align}\begin{aligned}\label{reacd_{E}}\partial_{t}f(x,g)=\Delta_{x}f_t(x,g) & +\divv_{x}(f_t(x,g)\nabla_{x}V(x,g))+\sum_{g^{\prime}\in\mathcal{G}}\big[\log f(x,g)+V(x,g)\\
 & \hspace{1em}\hspace{1em}\hspace{1em}\hspace{1em}\hspace{1em}\hspace{1em}\hspace{1em}\hspace{1em}-(\log f(x,g^{\prime})+V(x,g^{\prime})\big]K(g,g^{\prime})e^{-W(x)}.
\end{aligned}
\end{align}
\end{example}
\begin{example}
Suppose that the mobility $\theta$  takes the form
\[\theta_{x,g,g'}(s,t)=\theta_{log}(s\exp(V(x,g)),t\exp(V(x,g'))\ 
\]
where $\theta_{log}$  is the logarithmic interpolation function:
\[\theta_{log}(a,b):=\frac{a-b}{\log(a)-\log(b)}=\int_{0}^{1}a^{r}b^{1-r}dr.\ 
\]
For this choice of mobility, the dynamic cost of transporting mass from $(x,g)$  into $(x,g')$  depends on the value of the potential $V$  at these points, as well as on the value of the mass that is currently located at them. In particular, it is more expensive to move mass between these points when the amount of mass at one of them is close to zero. This mobility function satisfies  $(A1)-(A5)$ in Assumptions \eqref{assump:theta}. In this case, equations \eqref{eqn:GradFlowE} take the form
\begin{align*}\begin{aligned}\partial_{t}f_t(x,g)=\Delta_{x}f_{t}(x,g) & +\divv_{x}(f_{t}(x,g)\nabla_{x}V(x,g))\\
 & +\sum_{g^{\prime}\in\mathcal{G}}\big[f_{t}(x,g)\exp(V(x,g))-f_t(x,g^{\prime})\exp(V(x,g'))\big]K(g,g^{\prime}).
\end{aligned}
\end{align*}
which is a \textbf{linear} system of reaction diffusion equations. 
\end{example}


\subsection{Hamiltonian dynamics: formal computation of geodesic equations and accelerated methods for optimization}
\label{sec:Hamiltonian}

In this section we discuss how a formal Riemannian structure can be used to introduce accelerated methods for optimization of energies on $\mathcal{P}_2(\R^d \times\G)$. We first provide a characterization of the geodesic equations in the space $\mathcal{P}_2(\R^d \times \G)$, and then introduce a system of \textit{accelerated} dynamics for the minimization of the energy $\EE$ in \eqref{def:RelativeEntropy}. These two sets of equations are related to certain Hamiltonian systems in $\mathcal{P}_2(\R^d \times \G)$ which can be formally defined using a notion of acceleration of curves. Throughout this section we continue to work at a formal level.

\subsubsection{Geodesics}
To motivate the characterization of geodesics in $\mathcal{P}_2(\R^d \times \G)$, let us recall that when working on a smooth Riemannian manifold $\M$, the local equation satisfied by a geodesic $t \mapsto q(t) \in \M$ can be written as
\[  \begin{cases}  \dot{q}(t) = p(t) \\ \dot{p}(t) =0,   \end{cases}  \]
where $t \mapsto p(t)$ is understood as a vector field along the curve $t \mapsto q(t)$, and its derivative as the covariant derivative of $p$ along the curve $q$ (using the Levy-Civita connection) written $\nabla_{\dot{q}} p$. The second equation states that geodesics have zero \textit{acceleration}, i.e. $\nabla_{\dot{q}} \dot{q}=0$. This system can be understood as a Hamiltonian system on the tangent bundle $\T\M$ with Hamiltonian $\mathcal{H}(q,p):= \frac{1}{2}|p|^2_q$. 

Following the above intuition, in section \ref{sec:geodesic} we will formally derive for the formal Riemannian manifold $(\mathcal{P}_2(\R^d \times \G),W_2)$ the system of equations:
\begin{align}\label{geo_equation}\begin{cases}
\dot{\mu}_{t}+\text{div}_x(\nabla_x\varphi_t \mu_t)+\text{div}_{g}(\nabla_{g}\varphi_t\hat{\mu}_t)=0\\
\partial_{t}\varphi_t+\frac{1}{2}|\nabla_x\varphi_t|^{2}+\sum_{g^{\prime}}\big(\nabla_g \varphi_t \big)^{2}K(g,g')\partial_{1}\theta_{x,g,g'}(f_t(x,g),f_t(x,g'))=0,
\end{cases}\end{align}
characterizing geodesics in the space $(\mathcal{P}_2(\R^d \times \G), W_2) $; in the above $d\mu(x,g)=f(x,g) dxdg $, and we interpret $\partial_{1} \theta_{x,g,g'}(s,t)$ as the derivative in $s$ of the mobility function.  The first of the two equations, i.e. the continuity equation, simply states that the curve $t \mapsto \mu_t$ moves with velocity $(\nabla_x \varphi_t,\nabla_g \varphi_t)$. On the other hand, the left hand side of the second equation can be understood as the derivative of the velocity along the curve (i.e. the acceleration), and so by setting it to zero one matches the intuition coming from Riemannian geometry that was discussed earlier.

\subsubsection{ Second order dynamics \nc}

In order to introduce a system of second order dynamics for the optimization of an energy $\mathcal{E}$ like that in \eqref{def:RelativeEntropy}, we once again return to the setting of a smooth Riemmanian manifold $\M$ and consider the optimization of an objective function $q \in \M \mapsto E(q)$. The system
\[  \begin{cases}  \dot{q}(t) = p(t) \\ \dot{p}(t) = -\gamma p(t) - \nabla_\M E(q(t)),   \end{cases}  \]
can be interpreted as a continuous time accelerated method for the optimization of the objective $E$.  Here we abuse the use of the term \textit{accelerated} method slightly given the motivation coming from the Euclidean setting. \nc Indeed, in the case $\M= \R^d$ and when the parameter $\gamma$ is allowed to depend on time according to $\gamma=\gamma_t =3/t$, the above dynamics correspond to the continuous time analogue of the celebrated Nesterov accelerated method for optimization \cite{Nesterov}. For general $\M$, the above system may be interpreted again as a dynamical system on the tangent bundle $\T\M$, and can be understood as the flow map induced by a vector field  that is the addition of a Hamiltonian vector field on $\T \M$ with Hamiltonian $\mathcal{H}(q,p)= \frac{1}{2}|p|_q^2 + E(q)$ and a dissipative term that corresponds to the gradient of an energy $(q,p) \mapsto \frac{\gamma}{2}|p|_q^2 $ for a positive parameter $\gamma>0$.

Following the above intuition, we can introduce an accelerated method for the optimization of an objective on $\mathcal{P}_2(\R^d \times \G)$ such as the relative entropy $\mathcal{E}$. For this purpose we use the formal computation of the gradient of the relative entropy \eqref{formal_gradient} from subsection \ref{sec:FormalGradient} as well as the expression for the acceleration of curves in the formal Riemannian structure (which actually was already used when introducing the geodesic equation \eqref{geo_equation} and will be formally computed in section \ref{sec:geodesic}). We obtain the system:
\begin{align}\label{second_order}\begin{cases}
\dot{\mu}_{t}+\text{div}_x(\nabla_x\varphi_t \mu_t)+\text{div}_{g}(\nabla_{g}\varphi_t\hat{\mu}_t)=0\\
\partial_{t}\varphi_t+\frac{1}{2}|\nabla_x\varphi_t|^{2}+\sum_{g^{\prime}}\big(\nabla_g \varphi_t \big)^{2}K(g,g')\partial_{1}\theta_{x,g,g'}(f_t(x,g),f_t(x,g'))\\
\hspace{20em}= -[ \gamma \varphi_t(x,g) +\log f_t(x,g)+V(x,g)];
\end{cases}\end{align}
in the above, we interpret $d\mu(x,g)=f(x,g) dx dg $.

\red 

\begin{remark}

Notice that when the interpolation map $\theta $ is like the one in Example \ref{exam:MobilityMassIndep} the expression for the acceleration of a curve with velocity induced by the potentials $\varphi_t$ reads
\[ \partial_t \varphi_t + \frac{1}{2}|\nabla_x \varphi_t|^2.\]
\end{remark}

\nc

\subsection{Main theoretical result}\label{sec:main-results}

In the previous sections we have taken a formal Riemannian approach to make sense of the gradient descent ODE \eqref{eqn:GradFlow} when the energy $\mathcal{E}$ is the relative entropy defined in \eqref{def:RelativeEntropy}.  In this section we provide a more solid theoretical ground motivating equations \eqref{eqn:GradFlowE}. For that purpose we will define the gradient flow of $\mathcal{E}$  using the \textit{minimizing movement scheme} approach that we mentioned at the end of section \ref{sec:OTEuclidean}. To achieve this, we first introduce a family of \textit{static} transport costs that are used to define the iterations \eqref{eqn:JKO} (thinking of $\M=\mathcal{P}_{2}(\R^{d}\times\mathcal{G})$). Our main theoretical result, Theorem \ref{main-result} below, states that for a suitable static cost (see \eqref{sta_cost_def} below), and for a suitable choice of mobility $\theta$  (the one in Example \ref{exam:MobilityMassIndep}), the resulting minimizing movement scheme converges, as the time discretization parameter $\tau$  goes to zero, toward a solution of the equation formally derived in \eqref{reacd_{E}}. 


It is worth highlighting that the minimizing movement scheme that we consider here has the advantage of being defined in terms of a (static) transport cost that is closer to the Kantorovich formulation of the classical optimal transport problem (i.e. \eqref{def:Wass}), rather that in terms of the dynamic problem \eqref{def:SemiDiscreteOTdist}. First, the static formulation is computationally cheaper (e.g. using the entropic regularization methods from \cite{SinkhornAlgorithm} which can be used in our context). Additionally, for the static formulation we will be able to use techniques similar to those developed in \cite{F-G} to show 
that the resulting minimizing movement scheme satisfies a type of maximum principle characteristic of Fokker Planck equations.\\

To define our static transportation costs, we first introduce some notation. Given a measure $\mu \in \P_2(\R^d \times \G)$ we will consider the unique collection $\{\mu_{g}\}_{g\in\mathcal{G}}$   of positive measures over $\mathbb{R}^{d},$  such that 
\begin{equation}
\mu=\sum_{g\in\mathcal{G}}\mu_{g}\otimes\delta_{g}.
\label{eqn:Decomp}
\end{equation}
In the remainder we will often deal with absolutely continuous measures $d\mu(x,g)=f(x,g)\hspace{1mm}dxdg$ in $\P_2(\mathbb{R}^{d}\times\G)$, and by abuse of notation, in that case we will simply use the density $f$ to denote the measure $\mu.$  For example in the above decomposition, we will use the functions $f_{g}:\text{\ensuremath{\mathbb{R}^{d}\rightarrow\mathbb{R},}}$ (i.e. $f_g(x)=f(x,g)$) to denote the measures $\mu_{g}.$ 
We now introduce our static transportation problem which we remark is of interest in its own right.\\

  {\textbf{Static semi-discrete transportation problem.}} Let $\tau>0$ be a positive time step and let $W$  be as in Example \ref{exam:MobilityMassIndep}. For arbitrary measures $\mu,\sigma$  in $\mathcal{P}_{2}(\R^{d}\times\G)$  we define $ADM(\mu,\sigma)$  to be the set of pairs $(\gamma,h)$  (the admissible pairs) that satisfy:
\begin{enumerate}
\item $\gamma=\{\gamma_{g}\}_{g\in\G}$  where each $\gamma_{g}$  is a Borel positive measure on $\R^{d}\times\R^{d}$  and whose first marginal $\pi_{1\sharp}\gamma_{g}$  is equal to $\mu_{g}$ .
\item $h:\R^{d}\times\G\times\G\rightarrow\R$ is antisymmetric in $\G \times\G$ (i.e. for all $g,g' \in \G$, $x \in \R^d$ we have $h(x,g,g') =-h(x,g',g)$), and it belongs to 
\begin{equation}\label{L2Wgg}L_{W,K}^{2}(\mathbb{R}^{d}\times\mathcal{G}\times\mathcal{G}):=\bigg\{ h\in\mathbb{R}^{d}\times\mathcal{G}\times\mathcal{G}\rightarrow\mathbb{R}\hspace{1mm}:\hspace{1mm}\sum_{g,g^{\prime}}\int h_{gg^{\prime}}^{2}e^{-W}K(g,g^{\prime})dx<\infty \bigg\}.
\end{equation}
\item For every $g\in\G$ \begin{equation}\label{discrete_continuity}\sigma_{g}=\pi_{2\#}\gamma_{g}-\tau\sum_{g'}h_{gg^{\prime}}(x)K(g,g^{\prime})e^{-W(x)}.
\end{equation}
\end{enumerate}
The last term on the right hand side of the identity  \eqref{discrete_continuity} must be interpreted as the positive measure on $\R^{d}$  whose density (with respect to the Lebesgue measure) is given by
\[\tau\sum_{g'}h_{gg^{\prime}}(x)K(g,g^{\prime})e^{-W(x)}.\ 
\]
In the remainder we refer to the measures $\gamma_g$ as \textit{transport plans} and to the functions $h$ as \textit{mass exchange maps}. 

A \textit{static transportation cost} between $\mu,\sigma$ is defined by
\begin{equation}\label{sta_cost_def}\mathcal{A}^{\mathcal{G},W,\tau}(\mu,\sigma):=\inf_{(\gamma,h)\in ADM(\mu,\sigma)}C^{W,K}_{\tau}(\gamma,h),
\end{equation}
where 
\begin{equation}\label{pre_def}C_{\tau}^{W,K}(\gamma,h):=\sum_{g,g^{\prime}\in\,\mathcal{G}} \left(  \frac{1}{2\tau} \int_{\R^d}\int_{\R^d}|x-x^{\prime}|^{2}d\gamma_{g}+\frac{\tau}{4}\int_{\mathbb{R}^{d}}h_{gg^{\prime}}^{2}K(g,g^{\prime})e^{-W}\hspace{1mm}dx\right).
\end{equation}
Since the set $ADM(\mu,\sigma)$ may very well be the empty set, we follow the convention that the infimum of a quantity over an empty set is equal to $+\infty$. We use $Opt(\mu,\sigma)$  to denote the set of minimizers of \eqref{sta_cost_def} when $\mathcal{A}^{\mathcal{G},W,\tau}(\mu,\sigma)$  is finite.

The static semi-discrete optimal transport problem introduced above can be interpreted as an optimal two stage mass transport process from one distribution over $\R^d \times \G$ to another. In the first stage, mass is transported along each fiber of $\R^d$ (i.e. a set of the form $\R^d \times \{ g\}$). In the second stage, mass gets exchanged along every fiber of $\G$ (i.e. a set of the form $\{ x\} \times \G$). The optimal transport plans and optimal exchange maps (and implicitly the optimal intermediate mass distribution after stage 1), are chosen so as to minimize the sum of two terms: one that corresponds to aggregate quadratic cost in stage one, and the other that corresponds to an average of discrete $H^{-1}$ norms of the mass exchanged during stage two. In section \ref{sec:optimal_maps} we study the above semi-discrete (static) transport problem mathematically. In particular, we study properties of the set $ADM(\mu, \sigma)$ and characterize $Opt(\mu, \sigma)$ in a way that resembles Brenier's theorem for optimal transport in Euclidean space. Part of the motivation for the definition of this static problem comes from the theoretical desire of recovering the system \eqref{eqn:GradFlowE} as limit of a JKO scheme relative to some meaningful cost function. While this transport problem is not the same as the dynamic one from Definition \ref{def:SemiDiscreteOTdist}, we believe that they are actually closely related. This is a topic that we may explore in future work. \nc
\\

Let us now return to our aim of defining the gradient descent of the relative entropy energy $\mathcal{E}$ using the minimizing movement scheme. We use the cost function $2\tau \mathcal{A}^{\G,W,\tau}$ introduced above to produce the series of iterates in \eqref{eqn:JKO} for $\M=\mathcal{P}_{2}(\R^{d}\times\G)$ and $E=\mathcal{E}$. We will assume that the initial datum $\mu_{0}\in\mathcal{P}_{2}(\mathbb{\mathbb{R}}^{d}\times\mathcal{G})$ 
satisfies $\mathcal{E}(\mu_{0})<\infty.$   Moreover, we will impose a further technical condition and assume that $\mu_{0}$ has a probability density $f_{0}$ such that
\begin{equation}\label{Z_03}\lambda e^{-V}\leq f_{0}\leq\Lambda e^{-V},
\end{equation}
for some positive constants $\lambda$ and $\Lambda.$ Setting $\mu_0^\tau := \mu_0$, we will then let $\mu_{n+1}^{\tau}$ be  a \nc minimizer of 
\begin{equation}\label{def:OurJKO}\sigma \in \mathcal{P}_2(\R^d \times \G) \longmapsto\mathcal{E}(\sigma)+\mathcal{A}^{\G,W,\tau}(\mu,\sigma),
\end{equation}
where we set $\mu=\mu_{n}^{\tau}$ . In section \ref{sec:Prelim} we study properties of the minimization problem \eqref{def:OurJKO}, and in particular provide conditions under which minimizers exist (see Proposition \ref{A_step_EL}). It will then be straightforward to see that the resulting iterates must be absolutely continuous with respect to the measure $dxdg$, and thus can be written as $d\mu_{n}^\tau(x,g)= f_{n}^\tau(x,g) dxdg$. A continuous-time extension of the above iterates is defined via piecewise constant interpolation in time. Namely, 
\[f^\tau(t):=f_{n+1}^{\tau}, \quad t \in (n\tau,(n+1)\tau].\]

Comparing the minimization problems \eqref{def:JKO2:lagrangian} and \eqref{def:OurJKO}, we see our semi-discrete transportation cost plays the role of the kinetic energy in the Lagrangian formulation of the JKO scheme.\\

\color{black}

Our main theoretical result is the following:
\begin{theorem} \label{main-result} Suppose that $f_0$ satisfies \eqref{Z_03}, $W$ satisfies the conditions from Example \eqref{exam:MobilityMassIndep} and in addition for some constants $\lambda', \Lambda'$
\begin{equation} 
\lambda' e^{-W(x)} \leq e^{-V(x,g)} \leq  \Lambda' e^{-W(x)}. 
\label{eqn:WandV}
\end{equation}
where $V:\R^d \times \G \rightarrow \R$ is a differentiable function in $x$ that also satisfies
\[\sum_g \int_{\R^d} |\nabla_x V(x,g)|^2 e^{-V(x,g)}dx <\infty.\]

Then, for any sequence $\tau_{k}\downarrow0$  there exists a subsequence, not relabeled, for which $f^{\tau_{k}}$ converges to $f$   in $L^{2}(0;t_{F},L_{loc}^{2}(\mathbb{\mathbb{R}}^{d}\times\mathcal{G})))$ for any $t_F>0,$ where the map $t\in [0,\infty) \rightarrow f(t)$ belongs to $L_{loc}^{2}([0,\infty),W^{1,2}(\mathbb{\mathbb{R}}^{d}\times\mathcal{G})))$  and is a weak solution of \eqref{reacd_{E}} (see Definition \ref{def:WeakSol}).

Moreover, for every $t>0$
\begin{equation}\label{Z3}\lambda e^{-V(x,g)}\leq f(t,x,g)\leq\Lambda e^{-V(x,g)},
\end{equation}
for almost every  $(x,g)$ in $\mathbb{\mathbb{R}}^{d}\times\mathcal{G},$ where $\lambda, \Lambda$ are the constants in \eqref{Z_03}. 
\end{theorem}
We prove Theorem \ref{main-result} in section \ref{JKO_proof_sec}.

\begin{remark}
The function
\[
f_{\infty}(x,g)=ce^{-V(x,g)},
\]
with $c$ chosen so that
\[
\sum_{g\in\mathcal{G}}\int cf_{\infty}(x,g)\hspace{1mm} dx = 1,
\]
is an equilibrium point and solves equation \eqref{reacd_{E}}.
Consequently, the property described in \eqref{Z3} coincides with a well known maximum principle for the Fokker-Planck equation.
\color{black}
\end{remark}

\color{black}

\section{Metric and geometric properties of $W_2$}\label{ap:heuristic}

\subsection{Proof of Theorem \ref{thm:Dist}}
\label{ap:TheoremDistance}

\medskip

\textbf{1.} Let $\mu_0, \mu_1$ be two elements in $\mathcal{P}_2(\R^d \times \G)$. First we prove that the infimum in the definition of $W^2_2(\mu_0, \mu_1)$ is finite by exhibiting one solution to the continuity equation connecting $\mu_0$ and $\mu_1$ with finite kinetic energy. One such solution is described as follows.

Let us first assume that $\mu_0$ and $\mu_1$ are supported on the set $B(0,R)\times \G$ for some $R>0$. For each $g \in \G$ let $m_g :=\mu_0( \R^d \times \{ g\}) $ be the total mass assigned to the fiber $\R^d \times \{ g \}$ by $\mu_0$ and let $\mu_{0g}$, $\mu_{1g}$ be the positive measures over $\R^d$ defined by
\[ \mu_{0g}(A) := \mu_0(A \times \{ g\}) , \quad  \mu_{1g}(A) := \mu_1(A \times \{ g\}) \quad \forall A \subseteq \R^d , \text{ Borel}. \]
Also, let $\tilde{\mu}_1$ be the first marginal of the measure $\mu_1$, i.e.
\[ \tilde{\mu}_1(A) = \mu_1(A \times \G), \quad \forall A \subseteq \R^d , \text{ Borel}.  \]
Since the measures $\mu_{0g}$ and $m_g \tilde{\mu}_1 $ have the same amount of total mass, we can find a solution  $t\in[0,1] \mapsto ( \nu_{t,g}, \nabla_x \phi_t( \cdot,g))$ to the continuity equation on $\R^d$
\[ \dot{ \nu}_{t,g} + \divv_x(  \nabla_x \phi_t( \cdot,g) \nu_{t,g} )=0, \]
satisfying $ \nu_{0,g}= \mu_{0g} $, $ \nu_{1,g}= m_g\tilde{\mu}_{1}$, and
\[ \int_{0}^1 \int_{\R^d} |\nabla_x \phi_t(x,g)|^2 d\nu_{t,g}(x) dt < \infty.    \]

On the other hand, notice that for every $g \in \G$ the measure $\mu_{1g}$ is absolutely continuous with respect to $\tilde{\mu}_1$, and for $\tilde{\mu}_1$-a.e. $x$ we have
\[ \sum_{g} \frac{d\mu_{1g}}{d\tilde{\mu}_1}(x)=1.   \]
For each such $x$ we can find a solution to the discrete continuity equation $t \in [0,1] \mapsto (\gamma_{t,x}, \nabla_g \psi_t(x, \cdot))$ 
\[ \dot{\gamma}_{t,x} + \divv_g(  \nabla_g \psi_t(x,\cdot) \cdot \hat{\gamma}_{t,x} )=0 \]
satisfying $\gamma_{0,x}(g)= m_g$ and $\gamma_{1,x}(g)= \frac{d\mu_{1g}}{d\tilde{\mu}_1}(x)$ for all $g \in \G$, and satisfying 
\[  \int_{0}^1 \sum_{g,g'} |\nabla_g \psi_t(x,g,g')|^2 K(g,g')d \hat{\gamma}_{t,x}(g,g')dt \leq C,  \]
for some constant $C$ that only depends on $R$. Such solution exists due to assumptions (A0) or (A5) on $\theta$ and the fact that discrete optimal transport is well defined in that case (see \cite{Maas,Erbar2012}).

We define 
\[  \mu_t := \begin{cases} \sum_{g \in \G}  m_g  d {\nu}_{2t,g}(x) \otimes \delta_g   , \quad t \in [0,1/2] \\ \sum_{g \in \G}  \gamma_{(2t-1),x}(g)  d\tilde{\mu}_1(x) \otimes \delta_g  , \quad t \in [1/2,1] \end{cases}\]
and 
\[ \phi_t(x,g):=  \begin{cases} \phi_{2t}(x,g) , \quad t \in [0,1/2] \\ 0 , \quad t \in [1/2,1] \end{cases}  \quad \quad \psi_t(x,g):=  \begin{cases} 0 , \quad t \in [0,1/2] \\ \psi_{2t-1}(g), \quad t \in [1/2,1] \end{cases}  \]
It is straightforward to verify that $t\in[0,1] \mapsto (\mu_t,\nabla_x \phi_t, \nabla_g \psi_t )$ solves the semi-discrete continuity equation, connects $\mu_0$ and $\mu_1$, and has finite kinetic energy.

If $\mu_0, \mu_1$ are not compactly supported as assumed above, then pick \textit{any} $\tilde{\mu}_0, \tilde \mu_1$ compactly supported satisfying 
\[ \mu_0(\R^d \times \{g \}) = \tilde \mu_0(\R^d \times \{g \}) , \quad \mu_1(\R^d \times \{g \})= \tilde \mu_1(\R^d \times \{g \}), \quad \forall g \in \G.\]
One can then dynamically transport mass from $\mu_0$ to $\tilde{\mu}_0$ restricting the transport to each fiber $\R^d \times \{ g\}$ using a continuity equation with finite kinetic energy on each fiber (this is simply OT in $\R^d$). Then, one can transport dynamically from $\tilde{\mu}_0$ to $\tilde{\mu}_1$ (as done above) and finally transport dynamically from $\tilde{\mu}_1$ to ${\mu}_1$ restricting the transport to each fiber $\R^d \times \{ g\}$ (again doing OT just on $\R^d$).

\textbf{2.} Let us now show that $W_2(\mu_0, \mu_1)=0 $ if and only if $\mu_0 = \mu_1$.
First notice that if $\mu_0= \mu_1$ we may take $\phi_t \equiv 0$, $\psi_t\equiv 0$ and $\mu_t = \mu_0$ for all $t \in [0,1]$. Then, it is clear that $t \in [0,1] \rightarrow (\mu_t, \phi_t)$ solves the continuity equation, has zero kinetic energy, and connects $\mu_0$ and $\mu_1$, from where it follows that $W_2^2(\mu_0, \mu_1)=0$. 

Now let us suppose that $W_2^2(\mu_0, \mu_1) =0 $. We want to show that $\mu_0=\mu_1$. Fix an arbitrary test function $\zeta: \R^d \times \G \rightarrow \R$ where $\zeta( \cdot,g)$ is smooth and compactly supported for all $g \in \G$. From the condition $W_2(\mu_0, \mu_1)=0$ we see that for every $\veps>0$ there is a solution to the continuity equation $t \mapsto (\mu_t, \nabla_x \phi_t, \nabla_g \psi_t)$ connecting $\mu_0$ and $\mu_1$ with kinetic energy less than $\veps$, i.e., 
\begin{align*}
\mathcal{K}:=\int_{0}^{1} \left( \int_{\R^d}\sum_{g} |\nabla_x \phi_t(x,g)|^2d\mu_{t}(g,x) +\int_{\R^{d}}\sum_{g,g^{\prime}}\nabla_g\psi _{t}(x,g,g^{\prime}){}^{2}K(g,g^{\prime})d\hat{\mu}_{t}(x,g,g^{\prime}) \right) dt \leq \veps.
\end{align*}
Using \eqref{def:SemiDiscreteContEq_{0}} (after integration over $t \in [0,1]$) for the above test function $\zeta$, we conclude that 
\[ \left | \sum_{g \in \G} \int_{\R^d}\zeta(x,g)d\mu_1(x,g) -  \sum_{g \in \G} \int_{\R^d}\zeta(x,g)d\mu_0(x,g)  \right | \leq C_\zeta \sqrt{\mathcal{K}} \leq C_\zeta \sqrt{\veps} \]
where $C_\zeta$ is a constant that only depends on the test function $\zeta$. Given that $\veps$ was arbitrary we can conclude that
\[  \sum_{g \in \G} \int_{\R^d}\zeta(x,g)d\mu_1(x,g) =  \sum_{g \in \G} \int_{\R^d}\zeta(x,g)d\mu_0(x,g). \]
Finally, since $\zeta$ was an arbitrary smooth compactly supported test function we deduce that $\mu_0 = \mu_1$.

\textbf{3.} Next, we show that $W_2(\mu_0, \mu_1) = W_2(\mu_1, \mu_0)$. To see this, simply notice that any solution $t \in [0,1] \mapsto (\mu_t, \nabla_x\phi_t,\nabla_g \psi_t) $ to the continuity equation starting at $\mu_0$ and ending at $\mu_1$, can be reverted in time $t \in [0,1] \rightarrow (\mu_{1-t}, -\nabla_x\phi_{1-t}, -\nabla_g\psi_{1-t})$ producing in this way a solution to the continuity equation that starts at $\mu_1$ and ends at $\mu_0$, and has the exact same kinetic energy as the original curve.

\textbf{4.} Lastly we prove the triangle inequality. First we observe that after a standard reparametrization (of time) by arc-length it follows that for every $\mu, \tilde \mu \in \mathcal{P}_2(\R^d \times \G)$ and every $T>0$,
\begin{align}
\begin{split}
 W_2(\mu,\tilde \mu) & =\inf_{t \in[0,T] \mapsto (\mu_{t},\nabla_x \phi_t,\nabla_g \psi_t)}  \int_{0}^{T}  \biggr(  \int_{\R^d} \sum_{g\in\mathcal{G}}  | \nabla_x \phi_t (x,g)|^{2}d\mu_{t}(x,g)
\\ & +\int_{\R^{d}}\sum_{g,g^{\prime}} (\nabla_g \psi_t(x,g,g^{\prime})){}^{2}K(g,g^{\prime})d\hat{\mu}_{t}(x,g,g^{\prime}) \biggr)^{1/2} dt, 
\end{split}
\label{eqn:Wdef}
\end{align}
where the inf ranges over all solutions $t\in [0,T] \mapsto (\mu_t, \nabla_x \phi_t, \nabla_g \psi_t )$ to the semi-discrete continuity equation with $\mu_0=\mu$ and $\mu_T = \tilde \mu$

Let now $\mu_0, \mu_1, \mu_2$ be arbitrary elements in $\mathcal{P}_2(\R^d \times \G)$. From \eqref{eqn:Wdef}, for any $\veps>0$ we may consider $t \in [0,1] \mapsto (\mu_t,\nabla_x \phi_t, \nabla_g \psi_t)$ and $t \in [0,1] \mapsto (\widetilde \mu_t, \nabla_x \widetilde \phi_t, \nabla_g \widetilde \psi_t)$ solutions to the semi-discrete continuity equation satisfying $\mu_0= \mu_0$, $\mu_1= \mu_1 = \widetilde \mu_0$, $\widetilde \mu_1= \mu_2$ and
\begin{align*}
\int_{0}^{1} \left( \frac{1}{2}\sum_{g}\int_{\R^d} |\nabla_x \phi_t(x,g)|^2d\mu_{t}(x,g) +\int_{\R^{d}}\sum_{g,g^{\prime}}\nabla_g\psi _{t}(x,g,g^{\prime})^{2}K(g,g^{\prime})d\hat{\mu}_{t}(x,g,g^{\prime}) \right)^{1/2} dt  \\\leq W_2(\mu_0, \mu_1) + \veps,
\end{align*}
\begin{align*}
\int_{0}^{1} \left( \frac{1}{2}\sum_{g}\int_{\R^d} |\nabla_x \widetilde \phi_t(x,g)|^2d\widetilde \mu_{t}(x,g) +\int_{\R^{d}}\sum_{g,g^{\prime}}\nabla_g \widetilde \psi _{t}(x,g,g^{\prime})^{2}K(g,g^{\prime})d\hat{\widetilde \mu}_{t}(x,g,g^{\prime}) \right)^{1/2} dt  \\\leq W_2(\mu_1, \mu_2) + \veps.
\end{align*}
We then consider
\[  \gamma_t := \begin{cases} \mu_{t}, \quad t \in [0,1] \\ \widetilde \mu_{t-1}  , \quad t \in [1,2] \end{cases}\]
and the potentials
\[ \alpha_t(x,g):= \begin{cases} \phi_{t}(x,g) , \quad t \in [0,1] \\ \widetilde \phi_{t -1}(x,g) , \quad t \in [1,2] \end{cases}  \quad \quad \beta_t(x,g):=  \begin{cases} \psi_{t}(x,g) , \quad t \in [0,1] \\ \widetilde \psi_{t-1}(x,g) , \quad t \in [1,2]. \end{cases}  \]
It follows that $t \in [0,2] \mapsto (\gamma_t,\nabla_x \alpha_t, \nabla_g \beta_t )$ solves the semi-discrete continuity equation, connects $\mu_0$ and $\mu_2$, and satisfies
\begin{align*}
\int_{0}^{2} \left( \frac{1}{2}\sum_{g}\int_{\R^d} |\nabla_x \widetilde \phi_t(x,g)|^2d\widetilde \mu_{t}(x,g) +\int_{\R^{d}}\sum_{g,g^{\prime}}\nabla_g \widetilde \psi _{t}(x,g,g^{\prime})^{2}K(g,g^{\prime})d\hat{\widetilde \mu}_{t}(x,g,g^{\prime}) \right)^{1/2} dt  
\\\leq W_2(\mu_0, \mu_1) + W_1(\mu_1,\mu_2)  + 2\veps.
\end{align*}
From \eqref{eqn:Wdef} it follows that $W_2(\mu_0,\mu_2) \leq W_2(\mu_0, \mu_1)+ W_2(\mu_1, \mu_2) + 2\veps$. Since $\veps>0$ was arbitrary the result now follows.



\nc

\subsection{Tangent plane characterization}
\label{sec:tangentplanes}

In this section we provide concrete conditions under which the statement of Theorem \ref{TangentPlanes} can be made rigorous. The bottom line is that the arguments presented in this section motivate the formal characterization for the tangent plane $\T_\mu \mathcal{P}_2(\R^d \times \G)$, i.e. infinitesimal curves on $\mathcal{P}_2(\G \times \R^d)$ passing through $\mu$. The main result of this section can be interpreted as a minimal selection principle for the potentials $(\phi, \psi)$ driving a given solution to the continuity equation. Some of the results proved below will be used again later on when we get to analyze the static semi-discrete transport problem from section \ref{sec:main-results}.

Throughout this section we work with measures of the form $d\mu(x,g) = f(x,g) dx dg$ for a density function $f$ satisfying basic boundedness conditions. We also use the following spaces of potentials:
\begin{equation}\label{kernel} \Phi:=  \Bigg\{ \veps \in L^2_c(\R^d \times \G) \text{ s.t. } \int_{\R^d} \veps(x,g) dx=0 \quad \forall g, \quad \sum_g \veps(x,g)=0 \text{ a.e. } x \in \R^d  \Bigg\}, \end{equation}
where $L^2_c(\R^d \times \G)$ stands for the space of $L^2(\R^d \times \G)$ functions with compact support (i.e. almost everywhere equal to zero outside a set of the form $B(0,R) \times \G$), and also
\begin{equation*}\label{kernel2} \Phi^\perp:= \Bigg\{ \varphi \in L^2_{loc}(\R^d \times \G) \text{ s.t. } \int_{\R^d} \sum_g \varphi(x,g) \veps(x,g) dx =0, \quad \forall \veps \in \Phi    \Bigg\}. \end{equation*}


\begin{lemma}\label{prop_tangent} Let $f:\R^d \times \G \rightarrow \R $ be a probability density such that in every compact subset of $\R^d \times \G$ is bounded and bounded away from zero. Let $\phi, \psi$ be two potentials belonging to $L^{2}_{loc}(\R^d \times \G)$ for which 
\[\int_{\R^d}\sum_{g}|\nabla_x  \phi(x,g)|^2f(x,g) dx  +  \int_{\R^d}\sum_{g,g'}|\nabla_g  \psi(x,g,g')|^2 K(g,g') \hat{f}(x,g,g')dx< \infty.\] 
Consider the minimization problem:
\begin{equation}
 \inf_{\tilde \phi , \tilde \psi \in L^{2}_{loc}(\R^d \times \G)} \int_{\R^d}\sum_{g}|\nabla_x \tilde \phi(x,g)|^2f(x,g) dx  + \int_{\R^d}\sum_{g,g'}|\nabla_g \tilde \psi(x,g,g')|^2 K(g,g') \hat{f}(x,g,g') dx
 \label{eqn:LemSelecPrinciple}
\end{equation}
subject to
\[ \divv_x(f \nabla_x \tilde \phi) +  \divv_g(\hat f \nabla_g \tilde \psi) = \divv_x( f \nabla_x \phi) + \divv_g( \hat f \nabla_g \psi), \]
where the equality must be interpreted in the sense of distributions. 

Then, there exists a minimizing pair $\tilde \phi, \tilde \psi$ for the above problem. In addition, any minimizing pair must satisfy $\tilde \phi - \tilde \psi \in \Phi^\perp$.

\label{prop:tangentplanes}	
\end{lemma}
	\begin{proof}
 
\textbf{1.} Let us start by proving the existence of minimizers. First we consider the slightly modified problem
\begin{equation}
 \inf_{\tilde \phi , h } \int_{\R^d}\sum_{g}|\nabla_x \tilde \phi(x,g)|^2f(x,g) dx  +  \int_{\R^d}\sum_{g,g'}|h_{gg'}(x)|^2 K(g,g') \hat{f}(x,g,g') dx
 \label{eqn:LemSelecPrincipleAux1}
\end{equation}
subject to
\[ \divv_x(f \nabla_x \tilde \phi) +  \divv_g(\hat f \cdot h) = \divv_x( f \nabla_x \phi) + \divv_g( \hat f \nabla_g \psi), \]
where the minimization is now over pairs $(\tilde \phi, h)$ for $\tilde \phi$ as in problem \eqref{eqn:LemSelecPrinciple} and $h \in L^2_{loc}(\R^d \times \G \times \G) $ an antisymmetric function on $\G \times \G$ (i.e. $h_{gg'}(x)= -h_{g'g}(x)$ for every $x,g,g'$). Existence of solutions to \eqref{eqn:LemSelecPrincipleAux1} follows immediately from the direct method of the calculus of variations. 
From a solution $(\tilde \phi, h)$ to problem \eqref{eqn:LemSelecPrincipleAux1} we now construct a solution to \eqref{eqn:LemSelecPrinciple}. Fix $x \in \R^d$. Thanks to Proposition \ref{lemm:GraphPDE} there exists a solution $\tilde \psi(x, \cdot)=\tilde \psi_x$ to the graph PDE 
\[ \divv_g( \nabla_g \tilde \psi_x \hat{f}_x)= \divv_g(h_x \hat{f}_x ),  \]
which satisfies $\sum_{g} \tilde \psi_x(g)=0$. Following the proof of Proposition \ref{lemm:GraphPDE} and using the fact that $\sum_g \tilde \psi_x(g)=0$ we can conclude that there exists a constant $C_x>0$ for which
\begin{equation}
   \sum_{g} |\tilde \psi(x,g)|^2\leq C_x \sum_{g,g'} |\nabla_g \tilde \psi(x,g,g')|^2 K(g,g')\hat{f}(x,g,g').  
   \label{eqn:LemSelecPrincipleAux2}
\end{equation}
The constant $C_x$ can be assumed to be uniform on compact subsets of $\R^d$ thanks to the assumptions on $\theta $ and the fact that in each compact subset of $\R^d \times \G$ the function $f$ is assumed to be bounded and bounded away from zero. Using \eqref{eqn:DiscreteIntegrationParts} we obtain
\begin{align*}
 \sum_{gg'} |\nabla_g \tilde \psi_x|^2 K(g,g')\hat{f}_x(g,g') &= -\sum_{g} \divv_g(\nabla_g \tilde \psi_x \hat{f}_x) \tilde \psi_x  = -\sum_{g} \divv_g(h_x \hat{f}_x) \tilde \psi_x 
 \\&= \sum_{gg'} \nabla_g \tilde \psi_x \cdot h_x  K(g,g')\hat{f}_x(g,g')    ,
\end{align*}
and thus, from Cauchy-Schwartz inequality 
\[  \sum_{gg'} |\nabla_g \tilde \psi_x(g,g')|^2 K(g,g')\hat{f}_x(g,g')  \leq   \sum_{gg'} |h_x(g,g')|^2 K(g,g')\hat{f}_x(g,g').  \]
The above implies that $(\tilde{\phi}, \nabla_g \tilde\psi)$ is also a solution to \eqref{eqn:LemSelecPrincipleAux1}. Given that $\tilde \psi$ is in $L^2_{loc}$ thanks to \eqref{eqn:LemSelecPrincipleAux2}, we deduce that $(\tilde \phi, \tilde \psi)$ is a minimizing pair for \eqref{eqn:LemSelecPrinciple}.

\textbf{2.} Let $(\tilde \phi , \tilde \psi)$ be an arbitrary minimizing pair. Let $\veps$ be an arbitrary element in $\Phi$ and pick a specific measurable representative for it (which we also denote by $\veps$). For each fixed $g$ consider the PDE (in $x$) 
\begin{equation}
  \veps(\cdot,g)  = \divv_x( f(\cdot, g) \nabla_x  \eta(\cdot, g) ).
\label{eqn:LemPDE}  
\end{equation}
Existence of a solution $\eta(\cdot, g)$ in $L^2_{loc}(\R^d)$ follows from standard arguments in the theory of elliptic PDEs, given that $\veps$ has compact support and that in each compact subset of $\R^d \times \G$ $f$ is bounded and bounded away from zero. Also, let $x$ be a Lebesgue point for all the functions $\veps(\cdot, g)$, and consider the graph PDE 
\begin{equation}
  -\veps(x,\cdot) = \divv_g( \hat{f}_x \cdot \nabla_g \beta(x,\cdot)    ) .   
\label{eqn:LemGraphPDE}  
\end{equation}
This equation has a unique solution (that we denote by $\beta(x, \cdot)$) that averages to zero according to Lemma \ref{lemm:GraphPDE} (given that $\veps(x,\cdot)$ has average zero). Moreover, the function $\beta$ can be seen to be in $L^2_{loc}$ using the inequalities from Proposition \ref{lemm:GraphPDE}.

Now, for each $s\in \R$ consider the perturbed potentials:
\[ \phi_s(x,g):= \tilde \phi(x,g) + s \eta(x,g), \]
\[ \psi_s(x,g):= \tilde \psi(x,g) + s \beta(x,g),\]
and notice that
        \[ \divv_x(  f \nabla_x \phi_s ) = \divv_x(  f  \nabla _x \tilde \phi ) + s \divv_x(  f \nabla_x \eta ) = \divv_x(  f  \nabla _x\tilde \phi ) +  s\veps     \]
		\[ \divv_g(  \hat{f} \cdot \nabla _g \psi_s ) = \divv_g(  \hat{f} \cdot \nabla _g \tilde \psi ) + s \divv_g(  \hat{f} \cdot \nabla_g \beta ) = \divv_g(  \hat{f} \cdot \nabla _g \tilde \psi) -s\veps,   \]
	    so that in particular, for every $s \in \R$, 
	   the pair $(\phi_s, \psi_s)$
	    is admissible in the minimization of \eqref{eqn:LemSelecPrinciple}. Let $\mathcal{K}: \R \rightarrow \R$ be the function
		\[  \mathcal{K}(s) := \sum_g \int_{\R^d}  | \nabla _x \phi_s |^2 f(x,g) dx  + \sum_{g,g'} \int_{\R^d} ( \psi_s(x,g) - \psi_s(x, g') )^2 K(g,g')\hat{f}(x,g,g')dx,   \]
		which is minimized at $s=0$ by definition of $(\tilde \phi, \tilde \psi)$. Computing $\frac{d}{ds} \mathcal{K} (s)$ and evaluating at $s=0$, we deduce that
        \begin{align*}
        \begin{split}
         0 &= \sum_g \int_{\R^d} \nabla_x \eta(x,g) \cdot \nabla_x \tilde \phi(x,g) f(x,g) dx 
         \\&+   \int_{\R^d} \sum_{g,g'} (\beta(x,g)- \beta(x, g') )( \tilde \psi(x,g) - \tilde \psi(x,g')   ) K(g,g')\hat{f}(x,g,g') dx
         \\ & = \sum_g \int_{\R^d} \veps(x,g)\tilde \phi(x,g)dx - \int_{\R^d} \sum_g \veps(x,g) \tilde \psi(x,g)  dx
         \\&= \sum_g \int_{\R^d} \veps(x,g)(\tilde \phi(x,g)- \tilde \psi(x,g))dx
         \end{split}
         \end{align*}
		where the second equality follows from the fact that $\eta(\cdot, g)$ solves \eqref{eqn:LemPDE} and $\beta(x, \cdot)$ solves \eqref{eqn:LemGraphPDE}. Since $\veps\in \Phi$ was arbitrary, it follows that $\tilde \phi - \tilde \psi$ belongs to $\Phi^\perp$ as we wanted to show.

	\end{proof}


	\nc


\begin{lemma}\label{kernel_perp}
For any $\varphi$ in $\Phi^{\perp}$ there exists $\varphi_{1}:\mathbb{R}^{d}\rightarrow\mathbb{R}$ in $L^2_{loc}(\R^d)$ and $\varphi_{2}:\G\rightarrow\mathbb{R}$ such that
\[
\varphi(x,g)=\varphi_{1}(x)+\varphi_{2}(g), \quad \forall g \in \G, \quad \text{ a.e. }x  \in \R^d.
\]
Conversely, if $\varphi$ admits the above decomposition then $\varphi \in \Phi^\perp$.
\end{lemma}
\begin{proof}
Let $\varphi\in \Phi^{\perp}$ and fix a Lebesgue point $x_0$ for all the functions $\varphi(\cdot, \tilde g)$. Let
\[
\varphi_{2}(\tilde g):=\varphi(x_{0},\tilde g), \quad \tilde g \in \G.
\]
Observe that from Fubini's theorem any function that is independent of $x$ belongs to $\Phi^{\perp},$ and thus, $\varphi_{2}$ must be contained in $\Phi^{\perp}.$ Define now the function
\[
\varphi_{1}(\tilde x,\tilde g ):=\varphi(\tilde x,\tilde g)-\varphi_{2}(\tilde g).
\]
To complete our proof we must show that $\varphi_{1}$ does not depend on $\tilde g.$ For this purpose, let $x$ be an arbitrary Lebesgue point for all the functions $\varphi(\cdot, \tilde g)$. Fix $g,g' \in\mathcal{G}.$ Let $r>0$ and consider the test function 
\begin{equation}
\veps_r:=\xi^r_{x,g}-\xi^r_{x,g^{\prime}}-\xi^r_{x_{0},g}+\xi^r_{x_{0},g^{\prime}},
\label{eqn:AuxTestFunctionsEE}
\end{equation}
where $\xi^r_{x,g}: \R^d \times \G \rightarrow \R $ is given by
\[ \xi^r_{x,g}(\tilde{x}, \tilde{g}):= \frac{1}{|B(x,r)|} \mathds{1}_{B(x,r)}(\tilde x)\mathds{1}_{\{\tilde g=g\}}. \]
Notice that by construction $\veps_r$ is contained in $\Phi.$ Also, since $\varphi$ and $\varphi_{2}$ are contained in $\Phi^{\perp},$ $\varphi_{1}$ is contained in $\Phi^{\perp}$ too. Hence,
\begin{align*}
0&=\sum_{g}\int_{\mathbb{R}^d}\veps_r(\tilde x,\tilde g)\varphi_{1}(\tilde x,\tilde g) d\tilde x
\\&= \frac{1}{|B(x,r)|}\int_{B(x,r)}\varphi_1(\tilde x,g)d \tilde x - \frac{1}{|B(x,r)|}\int_{B(x,r)}\varphi_1(\tilde x,g')d \tilde x 
\\&- \frac{1}{|B(x_0,r)|}\int_{B(x_0,r)}\varphi_1(\tilde x,g)d \tilde x + \frac{1}{|B(x_0,r)|}\int_{B(x_0,r)}\varphi_1(\tilde x,g')d \tilde x.
\end{align*}
We may now take $r \rightarrow 0$ and use the fact that $x_0$ and $x$ were assumed to be Lebesgue points for the functions $\varphi(\cdot, g)$ and $\varphi(\cdot, g')$ (thus also for $\varphi_1$) to conclude that
\[0 = \varphi_1(x,g)- \varphi_1(x,g')-\varphi_1(x_0,g)+\varphi_1(x_0,g'). \]
By construction $\varphi_{1}(x_{0},g)=\varphi_1(x_0,g')=0$. Consequently, we deduce that
\[
\varphi_{1}(x,g)=\varphi_{1}(x,g^{\prime}).
\]
Since $x,$ $g$ and $g^{\prime}$ were arbitrary, we conclude that $\varphi$ can be written as the sum of a function of $x$ only and a function of $g$ only.

The converse statement is a direct consequence of Fubini's theorem.

\end{proof}

\begin{remark}
\label{rem:EEPerp}
Notice that from the proof of Lemma \ref{kernel_perp} it actually follows that if $\varphi \in L^2_{loc}(\R^d \times \G) $ is such that $\sum_g \int_{\R^d} \varphi(x,g) \veps(x,g) dx $ for all $\veps$ of the form \eqref{eqn:AuxTestFunctionsEE} then $\varphi$ can be written as $\varphi(x,g) = \varphi_1(x) + \varphi_2(g) $  (and in particular it follows that $\varphi \in \Phi^\perp$). We will use this observation in Proposition \ref{opt_couplings}. 
\end{remark}

We may now combine the previous two lemmas to deduce the following minimum selection principle providing concrete support to Theorem \ref{TangentPlanes}.

\begin{proposition}
Under the same assumptions on $f$ from Lemma \ref{prop_tangent}, there exists a minimizing pair for problem \eqref{eqn:LemSelecPrinciple} of the form $(\varphi, \varphi)$.
\end{proposition}

\begin{proof}
Consider an arbitrary minimizing pair for problem \eqref{eqn:LemSelecPrinciple}.
By Lemma \ref{prop_tangent} we know that this pair must satisfy $\tilde \phi - \tilde \psi \in \Phi^\perp$, and by Lemma \ref{kernel_perp} we can conclude that 
\[ \tilde \phi - \tilde \psi = \varphi_1 + \varphi_2,    \]
for some $\varphi_1:\R^d \rightarrow \R  $ in $L^2_{loc}(\R^d \times \G)$ and $\varphi_2: \G \rightarrow \R$. Consider now the function
\[ \varphi(x,g):= \tilde \phi(x,g) -\varphi_2(g)  \]
and notice that we can also write it as
\[ \varphi(x,g)= \tilde \psi(x,g)+ \varphi_1(x).  \]
It follows that
\[  \nabla_x \varphi = \nabla_x \tilde \phi, \quad \nabla_g \varphi = \nabla_g \tilde \psi.  \]
Due to the above relationship it follows that $(\varphi, \varphi)$ is admissible for the optimization problem \eqref{eqn:LemSelecPrinciple} and that it achieves the same value as that of the minimizing pair $(\tilde \phi, \tilde \psi)$. Therefore, $(\varphi, \varphi)$ solves \eqref{eqn:LemSelecPrinciple}.

\end{proof}






\nc



\subsection{A formal computation of the acceleration of a curve in $\mathcal{P}_2(\R^d \times \G)$: geodesic equations and accelerated methods for optimization}
\label{sec:geodesic}
In this section, we present a heuristic argument that motivates the discussion in section \ref{sec:Hamiltonian}. The heuristics are based on the formal computation of the acceleration of a given curve in $\mathcal{P}_2(\R^d \times \G)$.

Let us recall that the covariant derivative $\nabla_{\dot{q}(t)}$ along a smooth curve $t \mapsto q(t)$ on a smooth Rimennian manifold $\M$ is a mapping taking vector fields  into vector fields along the curve $q$. This mapping makes sense of the idea of differentiation of a vector field $t \mapsto p(t)$ along the curve in a way that is compatible with the Riemannian structure of $\M$. We will now recall a formula from Riemannian geometry that characterizes $\nabla_{\dot q} \dot q$ (the covariant derivative of the velocity of the curve, i.e. the acceleration of the curve) in terms of variations of the kinetic energy. For that purpose we let $t \in [0,T] \mapsto q(t)$ be a fixed smooth curve in $\M$. We recall that a (smooth) \textit{proper variation} of the curve $q$ is a smooth function $\alpha: (s,t) \in (-\veps, \veps) \times [0,T] \rightarrow \M$  satisfying $\alpha(0,t) = q(t)$ for all $t \in [0,T]$ and $\alpha(s,0)=q(0)$, $\alpha(s,T) = q(T)$ for all $s \in (-\veps, \veps)$. In particular, the maps $t \in [0,T] \mapsto \alpha(s,t) $ can be understood as describing nearby curves to the original curve $q$, and in that light, the vector field $v(t) =\frac{\partial }{\partial s} \alpha(0,t)$  known as the \textit{variational field} of $\alpha$ (which is a vector field along the curve $q$) describes an infinitesimal deformation of the curve maintaining its endpoints anchored. A well known result in Riemannian geometry (e.g.  Proposition 2.4 in Chapter 9 in \cite{docarmo1992riemannian}) states that:
\begin{equation}
\frac{d}{ds} \biggr \lvert_{s=0} \left( \frac{1}{2} \int_{0}^T  \left  \lvert  \frac{\partial}{\partial t}\alpha(s,t)    \right \rvert_{q(t)}^2dt \right)  = - \int_{0}^T \left \langle v(t) , \nabla_{\dot{q}} \dot{q} \right \rangle_{q(t)} dt.
\label{eqn:FirstVariationM}
\end{equation}
Since in the above one can take arbitrary variations of $q$, the previous expression indeed characterizes $\nabla_{\dot{q}} \dot {q}$ completely: regardless of the smooth proper variation taken,  the first variation of the kinetic energy (the left hand side) must match the right hand side which is expressed in terms of the corresponding variational field and the acceleration of the curve $\nabla_{\dot{q}} \dot{q}$.

Using the above discussion as motivation, let us now consider a curve $t \in [0,T] \mapsto f_{t}\in \mathcal{P}_2(\R^d \times \G)$ and let us provide a formal definition for its acceleration; here and in what follows we identify a measure $d\mu(x,g)= f(x,g) dxdg$ with its density,  and let $(\nabla_x \varphi_t, \nabla_g \varphi_t) $ be the velocity of the curve at time $t$. \nc Let $(s,t)\in (-\veps, \veps ) \times [0,T] \mapsto (f_{s,t}, \nabla_x \varphi_{s,t} , \nabla_g \varphi_{s,t})$ be a proper variation of $t\mapsto f_t$. Namely, we assume $(f_{0,t}, \varphi_{0,t}) = (f_t, \varphi_t)$ for all $t$, and $f_{s,0}= f_0$, $f_{s,T}= f_T$ for all $s\in(-\veps, \veps)$. We use $\psi_{s,t}$ to denote a potential associated to the curve $s \in (-\veps, \veps) \mapsto f_{s,t}$. The map $t \in [0,T] \mapsto \psi_t := \psi_{0,t}$ can then be interpreted as the corresponding variational field of the varition $(s,t) \mapsto f_{s,t}$. We assume all functions are smooth, and smooth in $s$ and $t$ so that we can take derivatives in $x,s,t$ at will.

Relative to the proper variation introduced above we define 
\[ F(s) :=  \frac{1}{2} \int_{0}^T \left(  \sum_{g} \int_{\R^d} |\nabla_x \varphi_{s,t}|^2 f_{s,t}(x,g)dx + \int_{0}^T  \sum_{g,g'} \int_{\R^d} |\nabla_g \varphi_{s,t}|^2 \hat{f}_{s,t}(x,g,g')dx \right) dt , \]
for  $s \in(-\veps , \veps)$, which according to \eqref{def:InnerProduct} can also be written as 
\[ \frac{1}{2} \int_{0}^T \langle  \varphi_{s,t}  , \varphi_{s,t} \rangle_{f_{s,t}} dt.    \]
We show that
\begin{equation}
    \frac{d}{ds} F(s)\biggr \rvert_{s=0} =   - \int_{0}^T  \left \langle \psi_t , \partial_t \phi_t + \frac{1}{2}|\nabla_x \varphi_t|^2 +  \sum_{g'} |\nabla_g\varphi_t(\cdot, \cdot, g')|^2 \partial_1\theta( f_{t}(\cdot, \cdot), f_{t}(\cdot, g'))  \right \rangle_{f_t} dt,
    \label{eqn:FirstVariationMeasures}
\end{equation}
which when compared to \eqref{eqn:FirstVariationM} motivates the definition of the acceleration of the curve $t \in [0,T] \mapsto (f_t, \nabla_x \varphi_t, \nabla_g \varphi_t)$ at time $t$ as the potential:
\[(x,g) \in \R^d \times \G \mapsto  \partial _t \varphi_t(x,g) + \frac{1}{2}|\nabla_x \varphi(x,g)|^2 +  \sum_{g'} |\nabla_g\varphi_t(x,g , g')|^2 \partial_1\theta( f_{t}(x, g), f_{t}(x, g')) .\]
Notice that in turn, the above definition motivates the geodesic equations given in \eqref{geo_equation}, as well as the (continuous time) accelerated scheme in \eqref{second_order} for the optimization of the relative entropy defined in \eqref{def:RelativeEntropy} (using the expression for its gradient that we found in section \eqref{sec:FormalGradient}) in light of the discussion in section \ref{sec:Hamiltonian}.

We now formally obtain \eqref{eqn:FirstVariationMeasures}. First,
\begin{align}
\label{eqn:AuxAcc0}
\begin{aligned}
\frac{d}{ds} F(s) &= \int_{0}^T \sum_{g}\int_{\R^d} (\nabla_x \partial_s\varphi_{s,t} \cdot \nabla_x \varphi_{s,t})  f_{s,t}(x,g)dt
\\&+ \int_{0}^T \sum_{g,g'}\int_{\R^d} (\nabla_g \partial_s\varphi_{s,t} \cdot \nabla_g \varphi_{s,t})K(g,g')\hat{f}_{s,t}(x,g,g')dt 
\\&+ \frac{1}{2}\int_{0}^T \sum_{g}\int_{\R^d} |\nabla_x \varphi_{s,t}|^2  \partial_s f_{s,t}(x,g)dt  + \frac{1}{2}\int_{0}^T \sum_{g,g'}\int_{\R^d} |\nabla_g \varphi_{s,t}|^2 K(g,g')\partial_s \hat{f}_{s,t}(x,g,g')dt 
\\&= \int_{0}^T \sum_{g}\int_{\R^d} (\nabla_x \partial_s\varphi_{s,t} \cdot \nabla_x \varphi_{s,t})  f_{s,t}(x,g)dt  
\\&+ \int_{0}^T \sum_{g,g'}\int_{\R^d} (\nabla_g \partial_s\varphi_{s,t} \cdot \nabla_g \varphi_{s,t})K(g,g')\hat{f}_{s,t}(x,g,g')dt 
\\&+ \frac{1}{2}\int_{0}^T \sum_{g}\int_{\R^d} |\nabla_x \varphi_{s,t}|^2  \partial_s f_{s,t}(x,g)dt  
\\&+ \int_{0}^T \sum_{g,g'}\int_{\R^d} |\nabla_g \varphi_{s,t}|^2  K(g,g') \partial_1\theta( f_{s,t}(x,g), f_{s,t}(x,g') ) \partial_sf_{s,t}(x,g) dt.
\end{aligned}
\end{align}
On the other hand, integration by parts and the fact that $\partial_s \varphi(0,s)=0$ and $\partial_s \varphi(s,T)=0$ for all $s$ (because the variation is proper) lead to
\begin{align*}
\int_{0}^T  \partial_t \varphi_{s,t}(x,g) \partial_s f_{s,t}(x,g) dt &= -\int_{0}^T   \varphi_{s,t}(x,g) \partial_s \partial_t f_{s,t}(x,g) dt 
\\& = - \frac{d}{ds}\left( \int_{0}^T \varphi_{s,t} \partial_t f_{s,t} dt \right) + \int_{0}^T \partial _s \varphi_{s,t} \partial_t f_{s,t} dt.
\end{align*}
After integration over $x,g$ and using the continuity equation, the above implies
\begin{align}
\label{eqn:AuxAcc1}
\begin{split}
&\int_{0}^T \sum_g \int_{\R^d}  \partial_t \varphi_{s,t}(x,g) \partial_s f_{s,t}(x,g) dx  dt 
\\&=- \frac{d}{ds} \left( \int_{0}^T \left( \sum_g\int_{\R^d} |\nabla_x \varphi_{s,t}|^2 f_{s,t}dx  + \sum_{g,g'} \int_{\R^d } |\nabla_g \varphi_{s,t}|^2 \hat{f}_{s,t}dx   \right) dt   \right)
\\&+ \int_{0}^T  \sum_g \int_{\R^d} \partial_s \varphi_{s,t} \partial_t f_{s,t} dx dt
\\& =  - 2\frac{d}{ds}F(s) 
\\&+ \int_{0}^T \sum_{g}\int_{\R^d} (\nabla_x \partial_s\varphi_{s,t} \cdot \nabla_x \varphi_{s,t})  f_{s,t}(x,g)dt  + \int_{0}^T \sum_{g,g'}\int_{\R^d} (\nabla_g \partial_s\varphi_{s,t} \cdot \nabla_g \varphi_{s,t})\hat{f}_{s,t}(x,g,g')dt.
\end{split}
\end{align}
Combining \eqref{eqn:AuxAcc0} and \eqref{eqn:AuxAcc1} we deduce that $\frac{d}{ds}F(s) $ can be written as: 
\[ - \int_{0}^T \int_{\R^d}\sum_{g} \left(  \partial_{t} \varphi_{s,t} + \frac{1}{2}|\nabla_x \varphi_{s,t}|^2+ \sum_{g'} |\nabla_g\varphi_{s,t}|^2 K(g,g')\partial_1\theta(f_{s,t}(x,g),f_{s,t}(x,g')   \right) \partial_s  f_{s,t}(x,g)   dx dt.  \]
Finally, at $s=0$ we have $\partial_s f_{s,t}= - \divv_x(\nabla_x \psi_t  f_{t} ) - \divv_g(\nabla_g \psi_t  \hat{f}_{t} ), $ and thus \eqref{eqn:FirstVariationMeasures} follows combining the above with the semi-discrete continuity equation.  
\nc



\section{Properties of minimizing pairs of the static semi-discrete optimal transport problem}\label{sec:optimal_maps}

 In this section we study the minimizers of the static semi-discrete transportation problem that we introduced in section \ref{sec:main-results}. Some of the results presented in this section will be used in the sequel while others are of interest on their own. We seek to reproduce the result of Brenier 
\cite[Theorem 1.26]{A-G} that characterizes optimal transport maps in the Euclidean setting in terms of convex functions. Our characterization is presented in Proposition \ref{opt_couplings}. We begin by studying the existence of optimal pairs. 
%



\begin{lemma} \textbf{ (Existence of optimal pairs)}  \label{coup_exists} Let $\mu,\sigma \in \mathcal{P}_2(\R^d \times \G)$ and suppose that $W_2^{\G,W,\tau}(\mu, \sigma)<\infty$.  Then, the set  $Opt(\mu, \sigma)$  (i.e. the set of solutions to \eqref{sta_cost_def}) \nc is non-empty.  
\end{lemma}
\begin{proof}


Let us consider a minimizing sequence of admissible pairs $\ensuremath{\{(\gamma_{n},h_{n})\}_{n=1}^{\infty}}$ and note that since $\mathcal{A}^{\mathcal{G},W,\tau}(\mu,\sigma)<\infty$  we have that, passing to a subsequence if necessary, we can assume that the second moments of $\{\gamma_{n}\}_{n=1}^{\infty},$  and the norm of $\{h_{n}\}_{n=1}^{\infty}$  in the weighted space $L_{W}^{2}(\mathbb{R}^{d}\times\mathcal{G\times\mathcal{G}})$ are equibounded (see \eqref{L2Wgg}). Consequently, since $L_{W}^{2}(\mathbb{R}^{d}\times\mathcal{G\times\mathcal{G}})$  is a Hilbert space, the existence of a minimizer follows by a standard lower compactness/lower semicontinuity and weak convergence argument (see \cite[Theorem 1.2]{A-G}). Indeed, since the constraint \eqref{discrete_continuity} is linear, we can pass it to the limit by weak convergence of $\gamma_{n}$  and $h_{n}$  in duality with smooth functions with compact support.
\end{proof}

Notice that if $\mu=\sigma$ then $W_{2}^{\G, W,\tau}=0 <\infty$.
The following lemma will not be used in the sequel, but provides other examples of $\mu$ and $\sigma$ for which one can prove that $W_2^{\G, W, \tau}(\mu, \sigma)<\infty$.

\begin{lemma}
Let $\mu,\sigma \in \mathcal{P}_2(\R^d \times \G)$ be absolutely continuous w.r.t. $dxdg$ and assume that $\sigma$'s density belongs to the space:
\begin{equation}\label{L2W}L_{W}^{2}(\mathbb{R}^{d}\times\mathcal{G}):=\bigg\{ f: \mathbb{R}^{d}\times\mathcal{G}\rightarrow\mathbb{R}\hspace{1mm}\text{s.t.}\hspace{1mm}\sum_{g\in\mathcal{G}}\int|f_{g}|^{2}e^{W}dx<\infty\bigg\}.
\end{equation}
Then, $W_2^{\G,W,\tau}(\mu,\sigma)<\infty$.
\end{lemma}
\begin{proof}
We begin by showing that the cost $\mathcal{A}^{\mathcal{G},W,\tau}(\mu,\sigma)$ is finite. Let $f$ and $\tilde f $ be the densities for $\mu $ and $\sigma$ respectively, and define
\[ m_g := \int_{\R^d} f_g(x) dx , \quad g \in \G,  \]
\[ \tilde{f}(x) :=   \sum_{g}\tilde{f}_g(x), \quad x \in \R^d.\]
Notice that for every $g \in \G$ the positive measures $f_g$ and $m_g \tilde f$ have the same total mass, and thus there exists a coupling $\gamma_g$ between them. In particular, $\pi_{1 \sharp} \gamma= f_g $, $\pi_{2 \sharp} \gamma= m_g \tilde f$ and also
\[ \int_{\R^d \times \R^d} |x-\tilde x|^2 d\gamma_g< \infty. \]
Now, notice that for every $x \in \R^d$ we have
\[ \sum_{g} (   m_g \tilde f(x) - \tilde f_g(x)  ) =0. \]
Therefore, we may use Proposition \ref{lemm:GraphPDE} in order to find $\eta(x,\cdot)$ satisfying
\begin{equation}
   m_g \tilde{f}(x) - \tilde f_g (x)    =  \sum_{ g' } ( \eta(x,g) - \eta(x,g')   ) K(g,g'), \quad \forall g \in \G,
   \label{eqn:AuxExistenceStatic0}
\end{equation}
as well as
\begin{equation}
\sum_{g,g'} |\eta(x,g) - \eta(x,g')|^2 K(g,g') \leq C   \sum_{g} | m_g \tilde{f}(x) -  \tilde{f}_g(x) |^2 e^{2W(x)},
\label{eqn:AuxExistenceStatic1}
\end{equation}
for some constant $C$ that only depends on the weighted graph $(\G,K)$. We let
\[  h_{gg'}(x) := \frac{e^{W(x)}}{\tau} ( \eta(x,g) - \eta(x,g')), \quad x \in \R^d ,  \quad g,g' \in \G       \]
and notice that from \eqref{eqn:AuxExistenceStatic0} it follows that
\[  \sigma_g = \pi_{2 \sharp} \gamma_g - \tau \sum_{g'}h_{gg'}(x)K(g,g') e^{-W(x)}.   \]
We observe that $h$ is clearly antisymmetric in $\G \times \G$, and thanks to \eqref{eqn:AuxExistenceStatic1} and the fact that $\tilde f \in L^2_W(\R^d \times \G)$ also satisfies
\[  \sum_{gg'} \int_{\R^d} h_{gg'}^2 e^{-W}K(g,g') dx < \infty.   \]
The bottom line is that $(\gamma, h) \in ADM(\mu, \sigma)$ and  $C_\tau^{W,K}(\gamma, h) < \infty$. It follows that $W_{2}^{\G ,W,\tau}(\mu, \sigma) < \infty$.
\end{proof}

\begin{remark} 

To provide an example where the cost is infinite suppose that $\G$ consists of two elements $g_1, g_2$ and $K(g_1, g_2)>0$. Let $\mu$ be the measure with representation $\mu_{g_1}= \delta_{x_1}$ for some $x_1\in \R^d$ and $\mu_{g_2}=0$ (i.e. all mass is in $g_1$), and let $\sigma$ be the measure with $\sigma_{g_1}=0$ and $\sigma_{g_2}=\delta_{x_2}$ for some $x_2 \in \R^d$.  We show that $ADM(\mu, \sigma)= \emptyset$. Indeed, if there existed an admissible pair, from \eqref{discrete_continuity} we would have that
\[ \delta_{x_2}=\sigma_{g_2}= \pi_{2\sharp} \gamma_{g_2} - \tau h_{g_2g_1}(x) K(g_1,g_2) e^{-W(x)}dx = - \tau h_{g_2g_1}(x) K(g_1,g_2) e^{-W(x)}dx.  \]
In other words, we would conclude that $\delta_{x_2}$ admits a density w.r.t. Lebesgue measure.

%
\end{remark}

\color{black}

The main ingredient necessary to prove the main result of this section, i.e. Proposition \ref{opt_couplings}, is a set of variational inequalities satisfied by optimal pairs. We obtain such inequalities by computing the first variation of minimizing pairs under suitable perturbations. We do this in the next lemma.
Before stating this result let us first introduce some notation that will be used in the remainder of the section. We let $\mu, \sigma $ be as in Lemma \ref{coup_exists} and assume that $\sigma$ has a density. To a given minimizing pair $(\gamma,h)$ we associate the density
\begin{equation}
\label{trasported_density}
\bar{f}_g(x):= \sigma_g(x)+\tau\sum_{g'} h_{gg^{\prime}}K(g,g') e^{-W},
\end{equation}
which corresponds to the density of the measure $\pi_{2\sharp} \gamma_g$. An immediate observation is that each $\gamma_{g}$  is an optimal plan for the OT problem between $\mu_g$ and $\pi_{2\sharp} \gamma_g$ for the cost $c(x,y)=\frac{|x-y|^{2}}{2\tau}$. 
\nc
Given that $\pi_{2\sharp} \gamma_g$ has a density, we know that there exists a unique map $S_{g}:\R^d \rightarrow \R^d$ such that $(S_{g},\text{Id})_{\#}\bar{f}_{g}=\gamma_{g}$  (see \cite{A-G}[Theorem 6.2.4 and Remark 6.2.11], for example). We will use the maps $ \{S_{g}\}_{g\in\mathcal{G}}$ to state the variational inequalities satisfied by minimizers of the static semi-discrete transportation problem. This set of inequalities serves as analogue to the notion of cyclical monotonicity that appears in the classical (Euclidean) optimal transport setting.
\begin{lemma}\textbf{(Variational inequalities)} \label{perturba}
Let $\mu$ and $\text{\ensuremath{\sigma}}$  satisfy the hypothesis of Lemma \ref{coup_exists} and suppose that in addition $\sigma$ has a density w.r.t. $dxdg$. Let $(\gamma,h)$  be an element in $Opt(\mu,\sigma)$. Then, the following properties hold:
\begin{itemize}
\item For any $g$ in $\mathcal{G}$ and any $y$ in $\mathbb{R}^{d}$, suppose we have two sequences $\ensuremath{\{g_{l}\}_{l=0}^{M}}$ and $\ensuremath{\{g_{l}^{\prime}\}_{l^{\prime}=0}^{M^{\prime}}}$ in $\mathcal{G},$ that satisfy both:
\begin{itemize}
\item[a)] The two sequences describe paths in the graph with the same initial and final endpoints, i.e, we have that $g_{0}=g_{0}^{\prime},$ $g_{M}=g_{M^{\prime}}^{\prime}, K(g_{l},g_{l+1})>0,$ and $K(g'_l,g'_{l+1})>0.$
 \item[b)] The point $y$ is a Lebesgue point for all the functions $h_{g_{l-1}g_l}$ and $h_{g'_{l-1}, g'_l}$.
\end{itemize}
Then, 
\begin{equation}\label{g-cycle}\sum_{l=1}^{M}h_{g_{l-1}g_{l}}(y)= \sum_{l^{\prime}=1}^{M^{\prime}}h_{g'_{l-1}g'_{l}}(y).\ 
\end{equation}
\end{itemize}
\begin{itemize}
\item Fix $g$ and $g^{\prime}$ satisfying $K(g,g')>0$ and assume that $y$ is a Lebesgue point for $S_{g}$ which also belongs to the support of $\pi_{2\sharp}\gamma_g$, and that $y'$ is a Lebesgue point for $S_{g^{\prime}}$ which also belongs to the support of $\pi_{2\sharp}\gamma_{g'}$. Then,
\begin{align}\begin{aligned}\label{trans_{v}s_{e}xchange}(h_{gg^{\prime}}(y')-h_{gg^{\prime}}(y)) & +\bigg[\frac{|y^{\prime}-S_{g}(y)|^{2}}{2\tau}-\frac{|y-S_{g}(y)|^{2}}{2\tau}\bigg]\\
 & +\bigg[\frac{|y-S_{g^{\prime}}(y')|^{2}}{2\tau}-\frac{|y^{\prime}-S_{g^{\prime}}(y^{\prime})|^{2}}{2\tau}\bigg]\geq0.
\end{aligned}
\end{align}
\end{itemize}
\end{lemma}
\begin{proof}
Let us start with a small outline describing the main ideas behind the proof.

\textbf{Heuristic Proof:} We begin analyzing \eqref{trans_{v}s_{e}xchange}. The idea is to perturb $\gamma_{g}$ by transporting a small amount of mass from $(S_{g}(y),g)$ into $\ensuremath{(y^{\prime},g)}$ instead of transporting it to $(y,g)$. On the other hand,  $\gamma_{g^{\prime}}$ is perturbed by transporting a small amount of mass from  $\ensuremath{\ensuremath{\ensuremath{(S(y^{\prime}),g^{\prime})}}}$ into $\ensuremath{\ensuremath{\ensuremath{(y,g^{\prime})}}}$ instead of transporting it to $\ensuremath{\ensuremath{(y^{\prime},g^{\prime})}}$. By modifying the plans $\gamma_g$ and $\gamma_{g'}$, we create a transport cost differential
\begin{equation}\label{trans_dif1}\bigg[\frac{|y^{\prime}-S_{g}(y)|^{2}}{2\tau}-\frac{|y-S_{g}(y)|^{2}}{2\tau}\bigg]+\bigg[\frac{|y-S_{g^{\prime}}(y^{\prime})|^{2}}{2\tau}-\frac{|y^{\prime}-S_{g^{\prime}}(y^{\prime})|^{2}}{2\tau}\bigg],
\end{equation}
per unit of mass transported. To balance the above perturbation in the transportation and remain with an admissible pair we must also perturb $h_{y^{\prime}}(gg^{\prime})$ and $h_{y^{\prime}}(g^{\prime}g)$ so that the extra amount of mass created by the transportation perturbation gets removed from $(y^{\prime},g)$ and put into $(y^{\prime},g^{\prime}).$ We must also perturb $h_{y}(g^{\prime}g)$ and $h_{y}(gg^{\prime})$ so that the extra amount of mass created by the transportation perturbation gets removed from $(y,g^{\prime})$ and put into $(y,g).$ Modifying the mass exchange function $h$ in this way creates a mass exchange cost differential of
\[h_{gg^{\prime}} (y^{\prime})-h_{gg^{\prime}}(y),
\]
per unit of mass transported. The resulting modified pair is still admissible, and by optimality of the original pair $(\gamma,h)$, it must be the case that
\begin{align*}\begin{aligned} (h_{gg^{\prime}} (y^{\prime})-h_{gg^{\prime}}(y)) & +\bigg[\frac{|y^{\prime}-S_{g}(y)|^{2}}{2\tau}-\frac{|y-S_{g}(y)|^{2}}{2\tau}\bigg]\\
 & +\bigg[\frac{|y-S_{g^{\prime}}(y^{\prime})|^{2}}{2\tau}-\frac{|y^{\prime}-S_{g^{\prime}}(y^{\prime})|^{2}}{2\tau}\bigg]\geq0,
\end{aligned}
\end{align*}
which is precisely \eqref{trans_{v}s_{e}xchange}.

To deduce \eqref{g-cycle} we consider two sequences $\{g_{l}\}_{l=1}^{M}$ and $\{g_{l}^{\prime}\}_{l=1}^{M^{\prime}}$ satisfying the given conditions $a)$ and $b)$ for some $y $ in $\mathbb{R}^{d}.$
We send some extra mass from the point $(y,g_{0})$ to the point $(y,g_{1})$ by increasing $h_{y}(g_{0}g_{1}).$ Then we take the extra mass at $ (y,g_{1})$ and send it to $(y,g_{2})$ by increasing $h_{y}(g_{1}g_{2}).$ We can continue in this fashion until we reach the point $(y,g_{M})=(y,g'_{M^{\prime}}).$ At this stage we will have a deficit of mass at the point $(y,g_{1})$ and an excess of mass at the point $(y,g'_{M^{\prime}})$ and we will pay an excess exchange cost given by:
\[
\sum_{l=1}^{M}h_{g_{l-1}g_{l}}(y),
\]
per unit of mass transported. We can balance the previous perturbation by reversing the mass exchange along the sequence $\{g_{l}^{\prime}\}_{l=1}^{M^{\prime}}$. Namely,  for each pair $g'_{l},g'_{l+1}$ we reduce the mass sent from $(y,g'_{l})$ to $(y,g'_{l+1})$ by decreasing $h_{g'_{l}g'_{l+1}}(y).$ Doing this we save
\[
\sum_{l=1}^{M'}h_{g'_{l-1}g'_{l}}(y)
\]
in terms of the cost. By optimality we must have 
\[
\sum_{l=1}^{M}h_{g_{l-1}g_{l}}(y) \geq \sum_{l=1}^{M'}h_{g'_{l-1}g'_{l}}(y).
\]
We can then switch the roles of the sequences and obtain the opposite inequality and from this deduce \eqref{g-cycle}.

Let us now make the previous ideas rigorous.

\textbf{Rigorous proof: 1.}  We begin with the proof of  \eqref{g-cycle}.
Let us fix two positive real numbers $r,\veps>0$ . We perturb our minimizer $(\gamma, h)$ by considering a new mass exchange function:
\[h_{g_{l-1}g_{l}}^{r,\varepsilon}(\hat y):=\begin{cases}
h_{g_{l-1}g_{l}}(\hat y) & \text{if\hspace{1mm}}\hat y\in B_{r}^{c}(y)\\
h_{g_{l-1}g_{l}}(\hat y)+\frac{\veps}{\tau K(g_{l-1}, g_l)e^{-W(\hat y)}} & \text{if\hspace{1mm}} \hat y\in B_{r}(y),
\end{cases}
\]
\[  h_{g'_{l-1}g'_{l}}^{r,\varepsilon}(\hat y):=\begin{cases}
h_{g'_{l-1}g'_{l}}(\hat y), & \text{if\hspace{1mm}}\hat y\in B_{r}^{c}(y)\\
h_{g'_{l-1}g'_{l}}(\hat y)- \frac{\veps}{\tau K(g'_{l-1}, g'_l)e^{-W(\hat y)}} & \text{if\hspace{1mm}}\hat y\in B_{r}(y),
\end{cases}\]
$h^{r,\veps}_{g_{l}g_{l-1}} = -h_{g_{l}g_{l-1}} $ and $h^{r,\veps}_{g'_{l}g'_{l-1}} = -h_{g'_{l}g'_{l-1}} $ to maintain the asymmetry, and finally $h_{gg^{\prime}}^{r,\varepsilon}=h_{gg^{\prime}}$ whenever $(g,g^{\prime})$ is not one of the consecutive pairs in the sequences. In the above we use $B_{r}(y)$ to denote the Euclidean ball of radius $r$ centered at $y.$ 

It is straightforward to see that the pair $(\gamma, h^{r,\veps})$ is admissible, and thus by the optimality of $(\gamma, h)$ we have  $C_{\tau}(\gamma,h)\leq C_{\tau}(\gamma, h^{r,\varepsilon})$, which simplifies to
\begin{align*}\begin{aligned}0 & \leq \frac{\tau}{2}\sum_{l=1}^M\int_{B_r(y)} \left( \left( h_{g_{l-1} g_{l}} +  \frac{\veps}{\tau K(g_{l-1}, g_l)e^{-W(\hat y)}} \right)^2  - ( h_{g_{l-1} g_{l}} )^2\right)K(g_{l-1},g_{l}) e^{-W(\hat y)} d\hat y      
\\& + \frac{\tau}{2}\sum_{l=1}^{M'}\int_{B_r(y)} \left( \left( h_{g'_{l-1} g'_{l}} - \frac{\veps}{\tau K(g'_{l-1}, g'_l)e^{-W(\hat y)}} \right)^2  - ( h_{g'_{l-1} g'_{l}} )^2\right)K(g'_{l-1},g'_{l}) e^{-W(\hat y)} d\hat y.
\end{aligned}
\end{align*}
Dividing by $\varepsilon$ and letting $\varepsilon\rightarrow0$ yields
\begin{align*}\begin{aligned}0 & \leq \int_{B_{r}(y)} \left( \sum_{l=1}^M h_{g_{l-1}g_l}(\hat{y})- \sum_{l=1}^{M'} h_{g'_{l-1}g'_l}(\hat{y})      \right)   d\hat y.
\end{aligned}
\end{align*}
Dividing by the volume of $B_{r}(y)$, letting $r\rightarrow0,$ and recalling that $y$ was assumed to be a Lebesgue point for all the functions $h_{g_{l-1}g_l}$ and $h_{g'_{l-1}, g'_l}$ we conclude that
\[ \sum_{l=1}^M h_{g_{l-1}g_l}({y})\geq  \sum_{l=1}^{M'} h_{g'_{l-1}g'_l}({y}).  \]
Switching the roles of the sequences we obtain the reverse inequality. \eqref{g-cycle} follows.

\textbf{2.} Let us now consider $(y_1,g_1)$ and $(y_2,g_2)$ such that $K(g_1, g_2)>0$, $y_1$ is a Lebesgue point of $S_{g_1}$ and belongs to the support of $\pi_{2\sharp} \gamma_{g_1} $, $y_2$ is a Lebesgue point of $S_{g_2}$ and belongs to the support of $\pi_{2 \sharp } \gamma_{g_2}$, and $y_1 \not = y_2$. Fix $\veps>0$, and let $r$ be a small enough positive number so that $B_r(y_1) \cap B_r(y_2)= \emptyset$. We now construct measures $\gamma_{g_1}^{r,\veps}$, $\gamma_{g_2}^{r,\veps}$ and a function $h^{r,\veps}_{g_1g_2}$ which we use to formalize the perturbation argument provided in the heuristic proof. To define these measures and function we first need to introduce some objects.

Let us start by defining
\[ m_1:= \gamma_{g_1}(\R^d\times B_r(y_1)), \quad m_2:= \gamma_{g_2}(\R^d\times B_r(y_2)).\]
Notice that both numbers are nonzero given that $y_1$ belongs to the support of $\pi_{2 \sharp} \gamma_{g_1}$ and $y_2$ belongs to the support of $\pi_{2 \sharp} \gamma_{g_2}$.
To ease the notation we use $\bar{\mu}_{g_1}$ and $\bar{\mu}_{g_2}$ to denote the positive measures
\[ \bar{\mu}_{g_1} := \pi_{2\sharp} \gamma_{g_1} = \bar{f}_{g_1} dx, \quad  \bar{\mu}_{g_2} := \pi_{2\sharp} \gamma_{g_2} = \bar{f}_{g_2}dx,\] 
and consider also the positive measures $\bar{\mu}_{g_1} |_{B_r(y_1)}$ and $\bar{\mu}_{g_2} |_{B_r(y_2)}$ defined by
\[ \bar{\mu}_{g_1} |_{B_r(y_1)}(A) := \bar{\mu}_{g_1}(A \cap B_r(y_1)), \quad  \bar{\mu}_{g_2} |_{B_r(y_2)}(A) := \bar{\mu}_{g_2}(A \cap B_r(y_2)),  \]
for all Borel subsets $A$ of $\R^d$. 

Let us consider the maps $\mathcal{T}_{y_{1}}^{y_{2}}(y):=(y-y_{1}+y_{2})$ and $\mathcal{T}_{y_{2}}^{y_{1}}(y):=(y-y_{2}+y_{1})$. Also, let $T_{1}: B_r(y_1) \rightarrow B_r(y_1)$ be an optimal transport map (for the quadratic cost) between the measures $ \mathcal{T}_{y_{2}}^{y_{1}}{\sharp }( \frac{m_1}{m_2} \bar{\mu}_{g_2}|_{B_r(y_2)})$ and  the measure $\bar{\mu}_{g_1} |_{B_r(y_1)}$ (measures that can be checked to have the same total mass), and let $T_{2}: B_r(y_2) \rightarrow B_r(y_2)$ be an optimal transport map between the measures $ \mathcal{T}_{y_{1}}^{y_{2}}{\sharp }( \bar{\mu}_{g_1}|_{B_r(y_1)})$ and the measure $\frac{m_1}{m_2}\bar{\mu}_{g_2} |_{B_r(y_2)}$.

We can now define the measures $\gamma_{g_1}^{r,\veps}$ and $\gamma_{g_1}^{r,\veps}$ by
\begin{align*}
 \gamma_{g_1}^{r,\veps}(A \times C) &:= \gamma_{g_1}(A \times C)  -\veps \gamma_{g_1}(A \times (C \cap B_r(y_1))) 
 \\&+ \veps (S_{g_1}, T_{2} \circ \mathcal{T}_{y_1}^{y_2})_{\sharp} \bar{\mu}_{g_1}|_{B_r(y_1)} (A \times C), 
\end{align*}
and
\begin{align*}
 \gamma_{g_2}^{r,\veps}(A \times C) &:= \gamma_{g_2}(A \times C) -\veps \frac{m_1}{m_2} \gamma_{g_2}(A \times (C \cap B_r(y_2))) 
 \\&+ \veps  (S_{g_2}, T_{1} \circ \mathcal{T}_{y_2}^{y_1})_{\sharp}( \frac{m_1}{m_2}\bar{\mu}_{g_2}|_{B_r(y_2)} )(A \times C), 
\end{align*}
for all $A,C$ Borel subsets of $\R^d$. For $g$ that is neither $g_1$ nor $g_2$ we set $\gamma_{g}^{r,\veps}= \gamma_g$. Notice that $\pi_{1\sharp } \gamma_{g_1}^{r,\veps} = \mu_{g_1}$ and $\pi_{2\sharp } \gamma_{g_2}^{r,\veps} = \mu_{g_2}$.

Finally, we define 
\[ h_{g_1g_2}^{r,\veps}(y):= h_{g_1g_2}(y) + \frac{\veps}{\tau K(g_1,g_2)e^{-W(y)} } \left( \frac{m_1}{m_2} \overline{f}_{g_2}(y) \mathds{1}_{B_r(y_2)}(y) -  \overline{f}_{g_1}(y) \mathds{1}_{B_r(y_1)}(y)  \right)  \]
and set $h_{g_2g_1}^{r,\veps}=-h_{g_1g_2}^{r,\veps}$, and $h_{gg'}^{r,\veps}= h_{gg'}$ for pairs $g,g'$ different from $g_1g_2$. 
It is straighforward to check that $h^{r,\veps} \in L_2^{W,K}(\R^d \times \G \times \G )$ and that for every $g \in \G$ 
\[  \sigma_g = \pi_{2\sharp  }\gamma^{r,\veps}_g- \tau \sum_{g'}h_{gg'}K(g,g')e^{-W}.  \]
That is, $(\gamma^{r,\veps}, h^{r,\veps}) \in ADM(\mu, \sigma)$ and thus by optimality of $(\gamma, h)$ we deduce that $C_\tau(\gamma,h )\leq C_\tau(\gamma^{r,\veps}, h^{r,\veps})$. This inequality simplifies to

\nc

\begin{align*}\varepsilon\int_{B_{r}(y_{1})} & \bigg[\frac{|Id-S_{g_1}|^{2}}{2\tau}-\frac{|T_{2}\mathcal{\circ T}_{y_{1}}^{y_{2}}-S_{g_1}|^{2}}{2\tau}\bigg]\bar{f}_{g_1}\hspace{1mm}dy\\
 & \leq\varepsilon\frac{m_{1}}{m_{2}}\int_{B_{r}(y_{2})}\bigg[\frac{|T_{1}\mathcal{\circ T}_{y_{2}}^{y_{1}}-S_{g_2}|^{2}}{2\tau}-\frac{|Id-S_{g_2}|^{2}}{2\tau}\bigg]\bar{f}_{g_2}\hspace{1mm}dy\\
 & +\frac{\tau}{2}\int_{B_{r}(y_{1})}\bigg[(h_{g_1g_2}-\frac{\varepsilon}{\tau K(g_1,g_2) e^{-W}}\bar{f}_{g_1})^{2}-h_{g_1g_2}^{2}\bigg]K(g_1,g_2)e^{-W}\hspace{1mm}dy\hspace{1mm}\\
 & +\frac{\tau}{2}\int_{B_{r}(y_{2})}\bigg[(h_{g_1g_2}+\frac{\varepsilon}{\tau K(g_1,g_2)e^{-W}}\frac{m_1}{m_2}\bar{f}_{g_2})^{2}-h_{g_1g_2}^{2}\bigg]K(g_1,g_2) e^{-W}\hspace{1mm}dy.
\end{align*}
If we divide by $\varepsilon$ and let $\varepsilon\rightarrow 0$
we obtain
\begin{align*}\int_{B_{r}(y_{1})} & \bigg[\frac{|Id-S_{g_1}|^{2}}{2\tau}-\frac{|T_{2}\mathcal{\circ T}_{y_{1}}^{y_{2}}-S_{g_1}|^{2}}{2\tau}\bigg]\bar{f}_{g_1}\hspace{1mm}dy\\
 & \leq\frac{m_{1}}{m_{2}}\int_{B_{r}(y_{2})}\bigg[\frac{|T_{1}\mathcal{\circ T}_{y_{2}}^{y_{1}}-S_{g_2}|^{2}}{2\tau}-\frac{|Id-S_{g_2}|^{2}}{2\tau}\bigg]\bar{f}_{g_2}\hspace{1mm}dy\\
 & \hspace{4em}-\int_{B_{r}(y_{1})}h_{g_1g_2}\bar{f}_{g_1}\hspace{1mm}dy+ \frac{m_{1}}{m_{2}}\int_{B_{r}(y_{2})}h_{g_1g_2}\bar{f}_{g_2}\hspace{1mm}dy.
\end{align*}
Consequently, dividing by $m_{1}$ and expanding we obtain
\begin{align}\begin{aligned}\label{pre_exchange}\frac{1}{m_{1}}\int_{B_{r}(y_{1})} & \bigg[\frac{|Id-S_{g_1}|^{2}}{2\tau}-\frac{|\mathcal{T}_{y_{1}}^{y_{2}}-S_{g_1}|^{2}}{2\tau}+R_{2}(y)\bigg]\bar{f}_{g_1}\hspace{1mm}dy\\
 & \leq\frac{1}{m_{2}}\int_{B_{r}(y_{2})}\bigg[\frac{|\mathcal{T}_{y_{2}}^{y_{1}}-S_{g_2}|^{2}}{2\tau}-\frac{|Id-S_{g_2}|^{2}}{2\tau}+R_{1}(y)\bigg]\bar{f}_{g_2}\hspace{1mm}dy\\
 &-\frac{1}{m_{1}}\int_{B_{r}(y_{1})}h_{g_1g_2}\bar{f}_{g_1}\hspace{1mm}dy+  \frac{1}{m_{2}}\int_{B_{r}(y_{2})}h_{g_1g_2}\bar{f}_{g_2}\hspace{1mm}dy,
\end{aligned}
\end{align}
where
\begin{align*}\frac{1}{m_{1}}\int_{B_{r}(y_{1})}|R_{2}(y)|\bar{f}_{g_1}\hspace{1mm}dy & =\frac{1}{m_{1}}\int_{B_{r}(y_{1})}\bigg|\frac{|\mathcal{T}_{y_{1}}^{y_{2}}-S_{g_1}|^{2}}{2\tau}-\frac{|T_{2}\mathcal{\circ T}_{y_{1}}^{y_{2}}-S_{g_1}|^{2}}{2\tau}\bigg|\bar{f}_{g_1}\hspace{1mm}dy\\
 & =\frac{1}{2\tau m_{1}}\int_{B_{r}(y_{1})}|\langle\mathcal{T}_{y_{1}}^{y_{2}}-T_{2}\mathcal{\circ T}_{y_{1}}^{y_{2}},\mathcal{T}_{y_{1}}^{y_{2}}+T_{2}\mathcal{\circ T}_{y_{1}}^{y_{2}}-2S_{g_1}\rangle|\bar{f}_{g_1}\hspace{1mm}dy\\
 &\hspace{10em} \leq\frac{r}{\tau}\frac{1}{m_{1}}\int_{B_{r}(y_{1})}|\mathcal{T}_{y_{1}}^{y_{2}}+T_{2}\mathcal{\circ T}_{y_{1}}^{y_{2}}-2S_{g_1}|\bar{f}_{g_1}dy,
\end{align*}
and by a similar computation
\[
\frac{1}{m_{2}}\int_{B_{r}(y_{2})}|R_{1}(y)|\bar{f}_{g_2}\hspace{1mm}dy\leq\frac{r}{\tau}\frac{1}{m_{2}}\int_{B_{r}(y_{2})}|\mathcal{T}_{y_{2}}^{y_{1}}+T_{1}\mathcal{\circ T}_{y_{2}}^{y_{1}}-2S_{g_2}|\bar{f}_{g_2}dy.
\]
We now use the above estimates and let $r\downarrow0$ in \eqref{pre_exchange} to deduce \eqref{trans_{v}s_{e}xchange} (with $(y,g)=(y_1,g_1)$ and $(y',g')=(y_2,g_2)$).
\end{proof}

Before proceeding with our characterization of optimal pairs let us first recall some useful definitions from the classical optimal transport theory. 
First, given a symmetric $c\hspace{1mm}:\hspace{1mm}\mathbb{R}^{d}\times\mathbb{R}^{d}\rightarrow\mathbb{R},$ we say that a function $\varphi\hspace{1mm}:\hspace{1mm}\mathbb{R}^{d}\rightarrow\mathbb{R}$ is $c$\textit{-concave}, if it can be written as
\begin{equation*}
\varphi(y)=\inf_{x\in\mathbb{R}^{d}}c(x,y)-\psi(x),\quad \forall y \in \R^d,
\end{equation*}
for some $\psi\hspace{1mm}:\hspace{1mm}\mathbb{R}^{d}\rightarrow\mathbb{R}.$ The $c$-\textit{transform} of a given $\varphi$ is the function $\varphi^{c}$ defined by 
\begin{equation}\label{ctrans_def}
\varphi^{c}(x):=\inf_{y\in\mathbb{R}^{d}}c(x,y)-\varphi(y),
\end{equation}
and its $c$-\textit{superdifferential} is the set
\begin{equation}\label{sub_def}
\partial_{+}^{c}\varphi:=\bigg\{(x,y)\in\mathbb{R}^{d}\times\mathbb{R}^{d}\hspace{1mm}:\hspace{1mm}\varphi^c(x)+ \varphi(y)=c(x,y)\bigg\}.
\end{equation}
To characterize minimizers of Problem 2.4, in the proposition below we will use the quadratic cost
\[
c(x,y):=\frac{1}{2\tau}|x-y|^{2}.
\]
We will also use the spaces $\Phi$ and $\Phi^\perp$ defined in \eqref{kernel}.
\begin{proposition}
{\textbf{(Characterization of optimal pairs)}}
\label{opt_couplings}
Let $\mu, \sigma$ be absolutely continuous with respect to $dxdg$ and assumme that $W_2^{\G,W,\tau}(\mu, \sigma)<\infty$. 
Also, let $(\gamma,h)$ be in $ADM(\mu,\sigma)$ and assume that $\mu_g$'s density and $\bar{f}_g$ as defined in \eqref{trasported_density} are strictly positive for every $g$ in $\mathcal{G}$. Then, the following are equivalent
\begin{itemize}
\item[$i.$] $C_{\tau}(\gamma,h)$  is minimal among all pairs in $ADM(\mu,\sigma).$ 
\item[$ii.$]There exist functions $\phi,\psi\hspace{1mm}:\hspace{1mm}\mathbb{R}^{d}\times\mathcal{G}\rightarrow\mathbb{R}$ satisfying the following properties:
\begin{itemize}
\item[$a)$] For every $g$ in $\mathcal{G},$ the plan $\gamma_{g}$  is supported on $\partial^{c}_{+}\phi_{g},$  for some $c$-concave function $\phi_{g}(\cdot)=\phi(\cdot,g).$
\item[$b)$] For Lebesgue almost every point $y \in \R^d$ the function $\psi_{y}(\cdot)=\psi(y,\cdot)$ satisfies 
\begin{equation}\label{mass_A1}\psi_{y}(g^{\prime})-\psi_{y}(g)=h_{gg^{\prime}}(y), \quad \forall g,g' \text{ with } K(g,g')>0.
\end{equation}
\item[$c)$] The difference $\phi-\psi$ belongs to $\Phi^\perp$ as defined in \eqref{kernel}. 
\end{itemize}
\item[$iii.$] We can find a single potential $\varphi\hspace{1mm}:\hspace{1mm}\mathbb{R}^{d}\times\mathcal{G}\rightarrow\mathbb{R}$ satisfying properties $a),$ $b),$ and $c)$ from item $ii.$
\end{itemize}
\end{proposition}
\begin{proof} \textbf{1.} Optimality of $(\gamma,h)$ implies that $\gamma_g$ is an optimal coupling between $\mu_g$ and $\pi_{2\sharp}\gamma_{g}$ for every $g$, and thus the proof that $i.\implies ii.a)$ follows directly from the classical (Euclidean) optimal transport theory (see \cite[Theorem 1.13]{A-G}). To prove that  $i.\implies ii.b),$  let us fix $y_{0}$ in $\mathbb{R}^{d}$ and $g_{0}$ in $\mathcal{G}$  and define
\[\psi(y_{0},g):=\sum_{l=1}^{M}h_{g_{l-1}g_{l}}(y_{0}),\ 
\]
for some sequence $\{g_{l}\}_{l=0}^{M}$ starting at $g_0$, with $K(g_l,g_{l+1})>0$, and for which $g_{M}=g$. Such sequence exists given that $(\G, K)$ was assumed to be connected. On the other hand, observe that by \eqref{g-cycle} the potential $\psi$ is well defined (i.e. does not depend on the actual sequence connecting $g_0$ and $g$). In particular, we also have
\[ \psi(y_0, g') = \sum_{l=1}^M h_{g_{l-1}g_{l}}(y_{0}) + h_{gg'}(y_0).\]
$ii.b)$ now follows.

We proceed to show that  $i.\implies ii.c).$  According to Remark \ref{rem:EEPerp} it suffices to show that the difference  $\psi-\phi$  is orthogonal to any $\veps$ of the form \eqref{eqn:AuxTestFunctionsEE} 
\[\varepsilon=\xi^r_{y^{\prime},g}-\xi^r_{y^{\prime},g^{\prime}}-\xi^r_{y,g}+\xi^r_{y,g^{\prime}},\ 
\]
for arbitrary $y,y',g,g'$ and $r>0$. To show this we proceed as follows.

Fix $g,g'$ with $K(g,g')>0$. We first claim that the function
\[ u_{gg'}(y) := \psi_{g'}(y) - \psi_g(y) +\phi_g(y) - \phi_{g'}(y).   \]
is a.e. constant, where $\psi$ is as in item $ii.b)$. To see this, notice that from  Brenier's theorem for the classical optimal transport problem with the (rescaled) quadratic cost the following holds: the functions $\phi_g, \phi_{g'}$ can be written as
\[  \phi_g(y)=-\beta_g(y) + \frac{|y|^2}{2\tau},  \quad \phi_{g'}(y)=-\beta_{g'}(y) + \frac{|y|^2}{2\tau}, \]
for convex functions $\beta_g$ and $\beta_{g'}$, and the maps $S_g$ and $S_{g'}$ are a.e. equal to $\tau \nabla_y \beta_g$ and $\tau \nabla_y \beta_{g'}$ respectively.  In particular, we can write
\[ u_{gg'}(y) = \psi_{g'}(y) - \psi_g(y) -\beta_g(y) + \beta_{g'}(y), \quad y \in \R^d.   \]
Now, for a given pair $y,y' \in \R^d$, we have $u_{gg'}(y) \geq u_{gg'}(y')$ or $ u_{gg'}(y') \geq u_{gg'}(y)$. Suppose for the moment that the first inequality holds. In that case, 
\begin{equation}
 \beta_g(y)- \beta_{g'}(y) - \beta_{g}(y')+ \beta_{g'}(y')  \leq \psi_{g'}(y) - \psi_g(y) +  \psi_g(y')- \psi_{g'}(y')  .
\label{eqn:AuxCharacOptimal1}
\end{equation}
After simplification, item $ii.a)$ and \eqref{trans_{v}s_{e}xchange} imply
\begin{equation}
\psi_{g'}(y) - \psi_g(y) +  \psi_g(y')- \psi_{g'}(y') \leq - \langle y'-y , \nabla_y \beta_g(y) \rangle   -  \langle y-y' , \nabla_y \beta_{g'}(y') \rangle,
\label{eqn:AuxCharacOptimal2}
\end{equation}
for a.e. $y,y'$. Combining \eqref{eqn:AuxCharacOptimal1} and \eqref{eqn:AuxCharacOptimal2}, and recalling the definition of $u_{gg'}$ we obtain
 \begin{align}
 \begin{split}
|u_{gg'}(y) - u_{gg'}(y')| & \leq \beta_g(y') - ( \beta_{g}(y) + \langle y'-y , \nabla_{y}\beta_g (y) \rangle ) 
\\& + \beta_{g'}(y) - ( \beta_{g'}(y') + \langle y-y' , \nabla_{y} \beta_{g'}(y') \rangle ).
\end{split} 
\label{eqn:AuxCharacOptimal3}
 \end{align}
Notice that if instead $u_{gg'}(y') \geq u_{gg'}(y)$ we would have obtained the same inequality as the one above changing the roles of $g$ and $g'$ on the right hand side, so we do not lose generality in assuming the former inequality. Given that along every straight line $\ell$ the functions $\beta_g, \beta_{g'}$ are convex, their distributional second derivatives (along $\ell$) are characterized in terms of Radon positive measures, implying that along almost every line $\ell$ in $\R^d$ the right hand side in \eqref{eqn:AuxCharacOptimal3} is $O(|y-y'|)$, and in particular $u_{gg'}$ is a locally Lipschitz function along $\ell$. Furthermore, along almost every line in $\ell$ and for almost every $y,y'$ on that line, the right hand side of \eqref{eqn:AuxCharacOptimal3} is $o(|y-y'|)$ (given that Radon measures can only have at most a countable number of point masses). This implies that the locally Lipschitz function $u_{gg'}$ (restricted to $\ell$) has derivative a.e. equal to zero, thus implying that the function is constant along almost every line $\ell$. From this it follows that $u_{gg'}$ is almost everywhere constant in $\R^d$.  The bottom line is that for almost every $y,y' \in \R^d$ we have
\[ \left(\psi_{g'}(y') - \psi_{g}(y') - \psi_{g'}(y) + \psi_{g}(y)   \right) -  \left(\phi_{g'}(y') - \phi_{g}(y') - \phi_{g'}(y) + \phi_{g}(y)   \right) =0.   \]
From the above it now follows that 
\[  \int_{\R^d} \sum_{\tilde g}  (\psi(y,\tilde g) - \phi(y, \tilde g)) \veps(y,\tilde g) \hspace{1mm} dy=0 , \]
for $\veps$ as in \eqref{eqn:AuxTestFunctionsEE}. This concludes the proof.

\nc

\textbf{2.} We now show that \textit{ii. implies iii.} By Lemma \eqref{kernel_perp} we can find $\varphi_{1}\hspace{1mm}:\hspace{1mm}\mathbb{R}^{d}\rightarrow\mathbb{R}$ in $L^2_{loc}(\R^d)$ and $\varphi_{2}\hspace{1mm}:\hspace{1mm}\mathcal{G}\rightarrow\mathbb{R}$ such that 
\[ 
\phi_{g}(y)-\text{\ensuremath{\psi_{y}(g)}}=\varphi_{1}(y)+\varphi_{2}(g).
\]
Let us define 
\[ 
\varphi(y,g):=\phi_{g}(y)-\varphi_{2}(g)=\ensuremath{\psi_{y}(g)}+\varphi_{1}(y).
\]
Clearly, we have that
\[ 
\varphi(y,g^{\prime})-\varphi(y,g)=\ensuremath{\psi_{y}(g^{\prime})-\psi_{y}(g)}.
\]
Thus $ii.b),$ follows. On the other hand, since $\phi_{g}$ is $c$-concave, $\phi_{g}(\cdot)-\varphi_{2}(g)$ is $c$-concave too. Also, it is straightforward to verify that the superdifferential of $\phi_{g}(y)$ and $\phi_{g}(y)-\varphi_{2}(g)$ agree. In particular, $ii.a)$ holds for the potential $\varphi$. 


\textbf{3.} To prove  that $iii. \implies i.$,  let $(\tilde{\gamma},\tilde{h})$  be any element of $ADM(\mu,\sigma).$ Then, using item $ii.a)$,  $\eqref{discrete_continuity},$ $\eqref{ctrans_def},$ and $\eqref{sub_def},$ we have that
\begin{align*}\begin{aligned}C_{\tau} & (\gamma,h)=\sum_{g\in\mathcal{G}}\bigg[\int_{\mathbb{\mathbb{R}}^{d}\times\mathbb{\mathbb{R}}^{d}}c(x,y)d\gamma_{g}+\frac{\tau}{4}\sum_{g^{\prime}\in\mathcal{G}}\bigg(\int h_{gg^{\prime}}^{2}(y)K(g^{\prime},g)e^{-W}dy\bigg)\bigg]\\
 & =\sum_{g\in\mathcal{G}}\bigg[\int_{\mathbb{\mathbb{R}}^{d}\times\mathbb{\mathbb{R}}^{d}}(\varphi_{g}^c(x)+\varphi_{g}(y))d\gamma_{g}+\frac{\tau}{4}\sum_{g^{\prime}\in\mathcal{G}}\bigg(\int h_{gg^{\prime}}^{2}(y)K(g^{\prime},g)e^{-W}dy\bigg)\bigg]\\
 & =\sum_{g\in\mathcal{G}}\bigg[\int_{\mathbb{\mathbb{R}}^{d}}\varphi_{g}^cd\mu_{g}+\int_{\R^d}\varphi_gd\sigma_{g}+\tau\sum_{g^{\prime}\in\mathcal{G}}\int\bigg(\varphi_{g}(y)(h_{gg^{\prime}}(y))K(g^{\prime},g)e^{-W}\bigg)dy\\
 & \hspace{17em}+\frac{\tau}{4}\sum_{g^{\prime}\in\mathcal{G}}\bigg(\int h_{gg^{\prime}}^{2}(y)K(g^{\prime},g)e^{-W}dy\bigg)\bigg]\\
 & =\sum_{g\in\mathcal{G}}\bigg[\int_{\mathbb{\mathbb{R}}^{d}\times\mathbb{\mathbb{R}}^{d}}(\varphi_{g}^c(x) +\varphi_g(y))d\tilde{\gamma}_{g} +\tau\sum_{g^{\prime}\in\mathcal{G}}\bigg(\int\varphi_{g}(y)\big(h_{gg^{\prime}}(y)-\tilde{h}_{gg^{\prime}}(y)\big)K(g',g)e^{-W}dy\bigg)
 \\ & \hspace{17em}+\frac{\tau}{4}\sum_{g^{\prime}\in\mathcal{G}}\bigg(\int h_{gg^{\prime}}^{2}(y)K(g^{\prime},g)e^{-W}dy\bigg)\bigg]\\
 & \leq\sum_{g\in\mathcal{G}}\bigg[\int_{\mathbb{\mathbb{R}}^{d}\times\mathbb{\mathbb{R}}^{d}}c(x,y)d\tilde{\gamma}_{g} +\sum_{g^{\prime}\in\mathcal{G}}\frac{\tau}{2}\bigg(\int\big(\varphi_{g}(y)-\varphi_{g^{\prime}}(y)\big)\big([{h}_{gg^{\prime}}(y)-\tilde{h}_{gg^{\prime}}(y)]K(g^{\prime},g)e^{-W}dy\bigg)\\
 & \hspace{24em}+\frac{\tau}{4}\sum_{g^{\prime}\in\mathcal{G}}\bigg(\int h_{gg^{\prime}}^{2}(y)K(g^{\prime},g)e^{-W}dy\bigg)\bigg],
\end{aligned}
\end{align*}
where in the last line we have used the antisymmetry of $h$ and $\tilde h$. Now, from item $ii.b)$ and the above inequality we obtain 
\begin{align*}\begin{aligned}C_{\tau}(\gamma,h) & \leq\sum_{g}\int_{\mathbb{\mathbb{R}}^{d}\times\mathbb{\mathbb{R}}^{d}}c(x,y)d\tilde{\gamma}_{g}+\frac{\tau}{4}\sum_{g,g'}\bigg(\int \tilde h_{gg^{\prime}}^{2}(y)K(g^{\prime},g)e^{-W}dy\bigg)\\
& +\frac{\tau}{4}\sum_{g,g'}\bigg(\int ( h_{gg^{\prime}}^{2}(y) - \tilde h_{gg^{\prime}}^{2}(y))K(g^{\prime},g)e^{-W}dy\bigg)
\\& +\sum_{g,g'}\frac{\tau}{2}\bigg(\int\big(h_{gg^{\prime}}(y)\big)\big(\tilde{h}_{gg^{\prime}}(y)-h_{gg^{\prime}}(y)\big)K(g^{\prime},g)e^{-W}dy\bigg)\\
 & \leq C_{\tau}(\tilde{\gamma},\tilde{h}).
\end{aligned}
\end{align*}

\end{proof}

\nc

\section{Properties of JKO minimizers and maximum principle}
\label{sec:Prelim}

In this section we prove a series of preliminary results characterizing solutions to the optimization problem \eqref{def:OurJKO}. In Proposition \ref{Barriers} we show that the iterates of the minimizing movement scheme satisfy a maximum principle that is characteristic of the Fokker Plank equation. In Proposition \ref{A_step_EL} we show that the corresponding potential $\varphi$ generating the associated optimal transport map and optimal exchange function from Proposition \ref{opt_couplings} agrees with \eqref{formal_gradient}, i.e. with the formula for the gradient of $\EE$ suggested by the formal computation from section \ref{sec:Riemman}.\\

{\color{black}

{\color{black}
We begin by showing that minimizers of \eqref{def:OurJKO} exist.

\begin{lemma}\label{A_step_exixts}\textbf{(Existence of minimizers to \eqref{def:OurJKO}).} Let $\mu$ be a measure in $\mathcal{P}_{2}(\mathbb{R}^{d}\times\mathcal{G})$ with the property that $\mathcal{E}(\mu)<\infty$. Then, there exists a minimizer $\mu_{\tau}\in\mathcal{P}_{2}(\mathbb{R}^{d}\times\mathcal{G})$ of 
\begin{equation}\label{minmov}\sigma\rightarrow \EE(\sigma)+\mathcal{A}^{W,\mathcal{G},\tau}(\mu,\sigma).
\end{equation}
Moreover, such a minimizer is absolutely continuous with respect to the measure $dxdg$.
\end{lemma}
\begin{proof} 
Since the entropy of $ \mu $ is finite, by considering the competitor $\sigma=\mu$ we deduce that the infimum in \eqref{minmov} if finite as well. Now, consider a minimizing sequence of measures $\{\sigma^{n}\}_{n=1}^{\infty},$  with corresponding optimal pairs $\{(\gamma^{n},h^{n})\}_{n=1}^{\infty}$ in $ADM(\mu,\sigma^{n}).$ Then, by construction, the second moments of $\ensuremath{\{\gamma^{n}\}_{n=1}^{\infty}}$ and the norm of $\{h^{n}\}_{n=1}^{\infty}$ in the weighted space $ L_{W}^{2}(\mathbb{R}^{d}\times\mathcal{G\times\mathcal{G}})$ are equibounded. Thus, following the argument of Lemma \ref{coup_exists} we can guarantee the existence of a pair $(\gamma,h)$ such that up to subsequence not relabeled, $\gamma^n$ converges weakly to $\gamma$, $h^n$ converges weakly (in $L_W^2$) to $h$ and
\[
\liminf_{n\rightarrow\infty}C_{\tau}(\gamma^{n},h^{n})\geq C_{\tau}(\gamma,h).
\]
From  
\[
\sigma_{g}^{n}=\pi_{2\#}\gamma_{g}^n-\tau\sum_{g'}h^n_{gg^{\prime}}(x)K(g,g^{\prime})e^{-W},
\]
and the weak convergence of the sequences $\{(\gamma^{n},h^{n})\}_{n=1}^{\infty}$, we deduce that
\begin{align*}\begin{aligned}\mu_{\tau} & :=\lim_{n\rightarrow \infty}\pi_{2\#}\gamma_{g}^{n}-\tau\sum_{g'}h_{gg^{\prime}}^{n}(x)K(g,g^{\prime})e^{-W}\\
 & \hspace{2em}=\pi_{2\#}\gamma_{g}-\tau\sum_{K(g^{\prime},g)>0}h_{gg^{\prime}}(x)K(g,g^{\prime})e^{-W}.\\
\end{aligned}
\end{align*}
Consequently, the pair $(\gamma,h)$ belongs to $\text{AMD}(\mu,\mu_\tau)$. Finally,
the inequality
\[\liminf_{n\rightarrow\infty}\mathcal{E}(\sigma^{n})\geq \mathcal{E}(\mu_\tau),\ 
\]
is a consequence of the weak convergence of $\sigma^{n}$ towards $\mu_\tau$ and the weak lower semi continuity of the relative entropy. The desired result follows.
\end{proof}
}



{\color{black}
 In the next lemma we prove a set of variational inequalities satisfied by minimizers of \eqref{def:OurJKO}. These inequalities are the main ingredient necessary to attain the main results of this section, i.e. Propositions \ref{Barriers} and \ref{A_step_EL}. We obtain these inequalities by computing the first variation of minimizing pairs under suitable perturbations. 
 

\begin{proposition}\textbf{(Variational inequalities of JKO minimizers)}.\label{min_prop} Let $\mu$ and $\mu_\tau$ be as in Lemma \ref{A_step_exixts}, and let $f_\tau$ be $\mu_\tau$'s density. Let $\{ \gamma_g\}_g$ and $h$ be the optimal transport plans and optimal exchange functions for the static semi-discrete optimal transport between $\mu$ and $\sigma=\mu_\tau$. The following inequalities hold: 
\begin{itemize}
\item Let $y\in\mathbb{R}^d$ be a Lebesgue point for the function $h_{g_1g_2}$ where $K(g_1,g_2)>0$, and suppose that $(y,g_{2})$ is an element in the support of $f_{\tau}$. Then,
\begin{equation}\label{E_stop_exchange}
\log f_{\tau}(y,g_{1})+V(y,g_{1})-[\log f_{\tau}(y,g_{2})+V(y,g_{2})]\geq h_{g_{1}g_{2}}(y).
\end{equation}
\item Let $y_{1}$ be a Lebesgue point for $S_{g}$ and suppose that $(y_{1},g),(y_{2},g)$ belong to the support of $f_\tau$. Then, 
\begin{equation}\label{E_stop_transport}
\log f_{\tau}(y_{2},g)+V(y_{2},g)-[\log f_{\tau}(y_{1},g)+V(y_{1},g)]+\frac{|S_{g}(y_{1})-y_{2}|^{2}}{2\tau}\geq\frac{|S_{g}(y_{1})-y_{1}|^{2}}{2\tau}.
\end{equation}
\item
Let $(x,y)$ be an element in the support of $\gamma_{g}$ for some $g$ in $\mathcal{G},$ and suppose that $x$ and $y$ belong to the support of $f_{\tau,g}$. Then, 
\begin{equation}\label{trans_eq}
\log f_{\tau}(x,g) + V(x,g)-[\log f_{\tau}(y,g) + V(y,g)]\geq\frac{|x-y|^{2}}{2\tau}.
\end{equation}
\nc
\end{itemize}
\end{proposition}
\begin{proof}
Let us start with a small outline describing the main ideas behind the proof.
	
\textbf{Heuristic Proof:} We begin by proving \eqref{E_stop_exchange}. For this purpose we consider the following perturbation of the optimal pair $(\gamma,h)$. The idea is to stop exchanging a small amount of mass between $\ensuremath{(y,g_{1})}$ and $\ensuremath{(y,g_{2})}.$ By doing this we save
\[
h_{g_{1}g_{2}}(y)+\log f_{\tau}(y,g_{2}) +1 +V(y,g_{2}),
\]
in terms of the mass exchange cost and the entropy, and we pay an extra
\[
\log f_{\tau}(y,g_{1})+1+ V(y,g_{1}),
\]
in terms of the entropy of the excess mass we now have in $(y,g_{1}).$ Thus, \eqref{E_stop_exchange} follows by optimality.\\

We proceed to the proof of \eqref{E_stop_transport}. We perturb $\gamma_{g}$ as follows. Instead of transporting a small amount of the mass from $(S_{g}(y_{1}),g)$ into $(y_{1},g)$ we transport it to $\ensuremath{(y_{2},g)}.$ By doing this, we create a transport cost differential
\[\frac{|y_{2}-S_{g}(y_{1})|^{2}}{2\tau}-\frac{|y_{1}-S_{g}(y_{1})|^{2}}{2\tau}.
\]
The resulting excess mass in  $(y_{2},g)$ and deficit of mass in $(y_{1},g)$ create an entropy differential of 
\[\log f_{\tau}(y_2,g)+V(y_2,g)-[\log f_{\tau}(y_1,g)+V(y_1,g)].
\]
Thus, \eqref{E_stop_transport} follows by optimality.\\

Finally, to prove \eqref{trans_eq} we take a pair $(x,y) $ in the support of $\gamma_{g}$ where both $x,y$ are assumed to belong to the support of $f_{\tau,g}$. Now, by setting  $y=y_{1}$ and $x=S(y_{1})=y_{2}$ in inequality \eqref{E_stop_transport} we have
\[ \log f_{\tau}(x,g)+V(x,g)-[\log f_{\tau}(y,g)+V(y,g)]\geq\frac{|y-x|^{2}}{2\tau}. \] 

\textbf{Rigorous proof:} We only prove \eqref{E_stop_exchange}. The proof of \eqref{E_stop_transport} follows exactly as in Proposition 3.7 from \cite{F-G} and the proof of \eqref{trans_eq} follows the same lines as Lemma \ref{perturba}.

Let $y\in\mathbb{R}^d$ be a Lebesgue point for the function $h_{g_1g_2}$ and suppose that $(y,g_{2})$ is an element in the support of $f_{\tau}$. Let $r$ and $\varepsilon$ be positive numbers.  We perturb the minimizing pair $(\gamma,h)$ by considering the new mass exchange function $\text{\ensuremath{h_{g_1g_2}^{r,\varepsilon}}\ensuremath{\hspace{1mm}:\hspace{1mm}\mathbb{R}^{d}\rightarrow\mathbb{R}} }$ defined by 
\[
h_{g_{1}g_{2}}^{r,\varepsilon}(\hat y)=\begin{cases}
h_{g_{1}g_{2}}(\hat y), & \text{if\hspace{1mm}}\hat y\in B_{r}^{c}(y)\\
h_{g_{1}g_{2}}(\hat y)-\frac{\varepsilon f_{\tau,g_{2}}(\hat y)}{\tau K(g_{1},g_{2})e^{-W(\hat y)}} & \text{if\hspace{1mm}} \hat y\in B_{r}(y),\end{cases}\]
$h_{g_2g_1}^{r,\veps}:= - h_{g_2g_1}^{r,\veps}$ and $ h_{gg^{\prime}}^{r,\varepsilon}=h_{gg^{\prime}}$ whenever $(g,g^{\prime})$ is not $(g_{1},g_{2})$ or $(g_{2},g_{1}).$
Observe that this produces a competitor $\mu_{\tau}^{r,\varepsilon}$ whose densities are given by
\[
f_{\tau,g_{1}}^{r,\varepsilon}(\hat y)=\begin{cases}
f_{\tau,g_{1}}(\hat y), & \text{if\hspace{1mm}} \hat y\in B_{r}(y)^{c}\\
f_{\tau,g_{1}}(\hat y)+\varepsilon f_{\tau,g_{2}}(\hat y) & \text{if\hspace{1mm}} \hat y\in B_{r}(y)
\end{cases},\text{\ensuremath{\hspace{1em}f_{\tau,g_{2}}^{r,\varepsilon}(\hat y)=\begin{cases}
f_{\tau,g_{2}}(\hat y), & \text{if\hspace{1mm}}\hat y\in B_{r}(y)^{c}\\
(1-\varepsilon)f_{\tau,g_{2}}(\hat y) & \text{if\hspace{1mm}}\hat y\in B_{r}(y),
\end{cases}}}
\]
and $f_{\tau,g}=f_{\tau}^{r,\varepsilon}$ whenever $g$ is not $g_{1}$ or $g_{2}.$ From the minimality of $\mu_\tau$ we get that 
\[
\sum_g\int_{\R^d} \ee(f_{\tau},\hat y , g)d\hat{y}+\mathcal{C}_{\tau}(\gamma,h)\leq \sum_g\int \ee(f_{\tau}^{r,\varepsilon},\hat y, g)d\hat y +\mathcal{C}_{\tau}(\gamma,h^{r,\varepsilon}),
\]
which simplifies to
\begin{align*}\begin{split}&\int_{B_r(y)}\bigg[  \ee\big(f_{\tau,g_{1}}(\hat y),\hat y, g_1\big)+
\ee\big(f_{\tau,g_{2}}(\hat y),\hat y , g_2\big)+\frac{\tau}{2}h_{g_{1}g_{2}}^{2}(\hat y)K(g_{1},g_{2})e^{-W(\hat y)}\bigg]d \hat y\\
 & \leq\int_{B_r(y)} \bigg[\ee\big(f_{\tau,g_{1}}(\hat y)+\varepsilon f_{\tau,g_{2}}(\hat y),\hat y , g_1\big)+\ee\big((1-\varepsilon)f_{\tau,g_{2}}(\hat{y}),\hat y , g_2\big)
 \\&+\frac{\tau }{2}\bigg(h_{g_{1}g_{2}}(\hat y)
 -\frac{\varepsilon f_{\tau,g_{2}}(\hat y) }{\tau K(g_{1},g_{2})e^{-W(\hat y)}}\bigg)^{2} K(g_{1},g_{2})e^{-W(\hat y)}\bigg]d\hat y.
\end{split}
\end{align*}
Reordering terms, we obtain
\begin{align*}\begin{split}&\int_{B_r(y)}\bigg[  \ee\big(f_{\tau,g_{1}}(\hat y),\hat y , g_1\big)-\ee\big(f_{\tau,g_{1}}(\hat y)+\varepsilon f_{\tau,g_{2}}(\hat y),\hat y , g_1\big) \bigg]d \hat y  
\\&\leq\int_{B_r(y)} \bigg[\ee\big((1-\varepsilon)f_{g_{2}}(\hat y),\hat y, g_2 \big)-\ee\big(f_{\tau,g_{2}}(\hat y), \hat y, g_2\big)\\
 & +\frac{\tau }{2}\bigg[\bigg(h_{g_{1}g_{2}}(\hat y)
 -\frac{\varepsilon f_{\tau,g_{2}}(\hat y) }{\tau K(g_{1},g_{2})e^{-W(\hat y)}}\bigg)^{2}  - h^2_{g_1g_2}(\hat y)\bigg] K(g_{1},g_{2})e^{-W(\hat y)}\bigg]d\hat y.
\end{split}
\end{align*}
Dividing by $\varepsilon$ and letting $\varepsilon\rightarrow 0$ yields
\begin{align*}\begin{split} & \int_{B_{r}(y)}\bigg[- \log f_{\tau,g_1}(\hat y) - 1 - V(\hat y, g_1)  \bigg]f_{\tau,g_{2}}(y)dy\\
 & \leq\int_{B_{r}(y)}\bigg[ - \log f_{\tau,g_2}(\hat y) - 1 - V(\hat y, g_2) - h_{g_{1}g_{2}}(\hat y)\bigg]f_{\tau,g_{2}}(\hat y)d\hat y.\\
\end{split}
\end{align*}
Dividing by $\int_{B_r(y)} f_{\tau,g_{2}}(\hat y) d \hat y$, and letting $ r\rightarrow0$ we obtain the desired inequality.
\end{proof}
}

 In the next proposition we prove that minimizers of \eqref{def:OurJKO} satisfy a maximum principle that is characteristic of Fokker Planck equations.

{\color{black}

\begin{proposition}\label{Barriers}
{\textbf{(Consistent barriers)}}
Suppose that $\mu$ and $\mu_\tau$ are as in Lemma \ref{A_step_exixts}. Suppose in addition that $\mu$'s density satisfies:
\[
\lambda e^{-V(x,g)}\leq f(x,g)\leq\Lambda e^{-V(x,g)},\hspace{1em}
\]
for every $(x,g).$ Then, $f_\tau$ satisfies 
\begin{equation}\label{gaussians}
\lambda e^{-V(x,g)}\leq f_{\tau}(x,g)\leq\Lambda e^{-V(x,g)},
\end{equation}
as well.
\end{proposition}
\begin{proof} 
We only prove the lower bound in \eqref{gaussians} since the argument for the upper bound is completely analogous. Let us define the set 
\[A:=\{(x,g) \: : \:\lambda e^{-V(x,g)}>f_{\tau}(x,g)\},
\]
and consider the auxiliary positive measure
\[
d\mu_{\lambda}=\lambda e^{-V(x,g)}dxdg.
\]
Suppose for the sake of contradiction that 
\[\mu_{\lambda}(A)>0.\ 
\]
Then 
\[\mu(A)\geq\mu_{\lambda}(A)>\mu_{\tau}(A),
\]
and thus the set $A$  has to lose mass during the transportation. Consequently, at least one of the following facts should hold:\\
\begin{itemize}
\item[\textit{i.}] There exist $g \in \G$ and $y$ a Lebesgue point of $S_{g}$ such that $(S_g(y),g)\in A$ and $(y,g)\not\in A.$
\item[\textit{ii.}] There exist a pair of nodes $g,g^{\prime}$ with $K(g,g^{\prime})>0$ and $x$ a density point of $h_{gg^{\prime}}$ for which $(x,g)$ and $(x,g')$ belong to the support of $f_\tau$, $h_{gg'}(x)>0$, $(x,g)\in A$ and $(x,g^{\prime})\not\in A$.
\end{itemize}
Let us show that in both cases we reach a contradiction.

Case \textit{i:} In this case, we apply \eqref{E_stop_transport} with $y_{1}=y$ and $y_{2}=S_g(y)$ to obtain that 
\[\log f_{\tau}(y,g)+V(y,g)+\frac{1}{2\tau}|S_g(y)-y|^{2}\leq\log f_{\tau}( S_g(y) ,g)+V( S_g(y),g).
\]
Now, observe that the assumption that $(y,g)\not\in A$  implies that the left-hand side of the above inequality is bigger than $\log\lambda,$ whereas the assumption that $(S_g(y),g)\in A$ implies the right-hand side is strictly smaller than $\log\lambda. $ Thus, we reach a contradiction.\\

 Case \textit{ii:} In this case we apply \eqref{E_stop_exchange} with $g_{2}=g^{\prime},$ $g_{1}=g$ and $y=x,$ to obtain that
\[0<h_{gg^{\prime}}(x)\leq\log f_{\tau}(x,g)+V(x,g)-\log f_{\tau}(x,g^{\prime})-V(x,g^{\prime}).\ 
\]
Moreover, our assumption that $(x,g^{\prime})\not\in A$ and $(x,g)\in A$ implies that the right hand side is negative. Thus, we reach a contradiction. 
\end{proof}

}

 As a byproduct of the above proposition, we obtain a uniform control on the distance traveled by the transported mass.

\color{black}
\begin{lemma}\textbf{(Transportation bound)}\label{trans_bound} Let 
$\mu$, $\mu_\tau$, $\lambda,$ and $\Lambda$ be as in Proposition \ref{Barriers}.
Then, there exists $C>0$ such that for all $g \in \G$
\[
| y-x|\leq C\sqrt{\tau}\hspace{1em}\forall(x,y)\in supp(\gamma_g),
\]
where we recall $\gamma=\{\gamma_g \}_{g \in \G}$ is the set of optimal plans between $\mu$ and $\mu_\tau$. The constant $C$ can be taken to be $C= \sqrt{2}(\log(\Lambda)- \log(\lambda))$.
\end{lemma}}
\begin{proof} The estimate follows by combining  \eqref{trans_eq} with Proposition \ref{Barriers}.
\end{proof}

In the next lemma we show that the target density $f_{\tau}$ and the transported density 
\[
\bar{f}_{\tau}(x,g)=f_{\tau}(x,g)+\tau \sum_{g'}h_{gg^{\prime}}(x)K(g,g^{\prime})e^{-W(x)},
\]
are comparable. Recall that $\bar{f}_{\tau,g}$ is nothing but the density of the positive measure $\pi_{2\sharp}\gamma_g$.

{\color{black}
\begin{lemma}(\textbf{Positivity of the transported mass}) \label{pos_trans}
Let 
$\mu$, $\mu_\tau$, $\lambda,$ and $\Lambda$ be as in Proposition \ref{Barriers}, and let $\lambda',\Lambda'$ be as in \eqref{eqn:WandV}. Finally, let $\bar{f}_{\tau}$ be defined as above. Then, there exists a positive constant $\tau_{0}:=\tau_{0}(\lambda,\Lambda,\lambda', \Lambda') <1/2$ such that for any $\tau$ in $(0,\tau_{0})$ we have that $\bar{f}_\tau>0$, i.e. the support of $\pi_{2\#}\gamma_{g}$ is all of $\mathbb{R}^{d}$ for all $g \in \G$. Moreover, we have that 
\begin{equation}\label{comparable}
\frac{C}{1-\tau}<\frac{\bar{f}_{\tau,g}}{f_{\tau,g}}<C(1+\tau),
\end{equation}
for any $\tau$ in $(0,\tau_{0})$ for some constant $C$ that only depends on $\lambda, \Lambda, \lambda', \Lambda'$.
\begin{proof}
To prove \eqref{comparable}, we note that thanks to \eqref{E_stop_exchange} and \eqref{gaussians}, we have that the mass exchange function $h$ is uniformly bounded in terms of $\lambda$ and $\Lambda.$ Additionally, \eqref{gaussians} and the assumption \eqref{eqn:WandV} imply that the quotient of $e^{-W}$ and $f_{\tau}$ is uniformly bounded as well. Hence, the desired result follows.
\end{proof}
\end{lemma}}

 In the next  proposition we show that the potential $\varphi$ that generates the optimal transport map and exchange function between $\mu$ and $\mu_{\tau}$ for $\mu$ satisfying the conditions from Proposition \ref{Barriers} (see item iii. in Proposition \ref{opt_couplings}) agrees with the negative of \eqref{formal_gradient} which is the gradient of the relative entropy suggested by the formal Riemannian structure from section \ref{sec:Riemman}.

\begin{proposition}\textbf{(The gradient of the relative entropy and JKO minimizers)} \label{A_step_EL}
Let 
$\mu$, $\mu_\tau$, $\lambda,$ and $\Lambda$ be as in Proposition \ref{Barriers}, let $\lambda',\Lambda'$ be as in \eqref{eqn:WandV}, and let $\tau_{0}>0$ be as in Lemma \ref{pos_trans}. Then, for every  $\tau$ in $(0,\tau_{0})$ we have: 
\begin{itemize}
\item[i.] For each $g$ in $\mathcal{G}$  the optimal transport plan $\gamma_{\tau,g}$ is given by 
\begin{equation}\label{EL_plan}
\gamma_{\tau,g}=(S_{g},Id)_{\#}\big(\text{\ensuremath{f_{\tau,g}+\tau\sum h_{\tau,gg^{\prime}}K(g,g^{\prime})e^{-W}}} \big),
\end{equation}
where the corresponding optimal transport map $S_{g}$ satisfies
\begin{equation}\label{EL_transport}\frac{S_{g}(y)-y}{\tau}f_{\tau}(y,g)=\nabla_{x}f_{\tau}(y,g)+f_{\tau}(y,g)\nabla_{x}V(y,g),
\end{equation}
for almost every $y$ in $\mathbb{R}^{d}.$
\item[ii] For each pair $g,g'$ with $K(g,g')>0$ and for almost every $x$ in $\mathbb{R}^{d}$, the optimal exchange function $h_{\tau,gg'}$ satisfies
\begin{equation}\label{E-LA1}
h_{\tau,gg^{\prime}}(x)=\big[\log f_{\tau}(x,g)+V(x,g)-\log f_{\tau}(x,g^{\prime})-V(x,g^{\prime})\big].
\end{equation}

\end{itemize}
\end{proposition}

\nc

\begin{proof}
We begin by noting that thanks to Lemma \ref{pos_trans} and Proposition \ref{Barriers} we have that the support of $f_{\tau,g}$ and $\pi_{2\#}\gamma_{g}$ is $\mathbb{R}^{d}$ for any $g$ in $\mathcal{G},$ i.e. $f_\tau>0$ and $\bar{f}_\tau>0$.  We will use this fact together with the variational inequalities from Proposition \ref{min_prop}.

\textbf{1.} Let us begin by showing i.

Observe that due to \eqref{E_stop_transport} for any $(x,y)$ in the support of $\gamma_{g}$ we have  
\[
\log f_{\tau,g}(z)+V(z,g)-\log f_{\tau,g}(y)-V(y,g)+\frac{|x-z|^{2}}{2\tau}\geq\frac{|x-y|^{2}}{2\tau},
\]
for almost every $z$ in $\mathbb{R}^{d}.$ Expanding the squares and rearranging terms we obtain that
\[
\log f_{\tau,g}(z)+V(z,g)+\frac{|z|^{2}}{2}\geq\log f_{\tau,g}(y)+V(y,g)+\frac{|y|^{2}}{2}+\langle\frac{x}{\tau},z-y\rangle\text{\hspace{1em}for \ensuremath{\text{almost  every \ensuremath{z} in \ensuremath{\mathbb{R}^{d}.} }} }
\]
Such an inequality implies that, up to redefining $f_{\tau,g}$ in a set up measure zero, the function $\Phi_{g}(z)=\log f_{\tau,g}(z)+V(z,g)+\frac{|z|^{2}}{2}$ is convex and for almost every $y$ in $\mathbb{R^{d}}$ and every pair $(x,y)$ in the support of the optimal transport plan $\gamma_{\tau,g}$ we have that $\frac{x}{\tau}$ is contained in the subdiffrential of $\Phi_{g}$ at $y$. Following the notation from \cite[Section 3.1]{GigliBook} , we shall denote such a subdiffentrial by $\partial^{-}\Phi(y).$ Finally, since convex function are almost everywhere differentiable, we have that for almost every $y$ the set $\partial^{-}\Phi_{g}(y)$ is a singleton and
\[
\nabla_{z=y}\Phi_{g}=\frac{x}{\tau}.
\]
Moreover using the almost everywhere differentiability of $\Phi_{g}$ we get that $z\rightarrow\log f_{\tau,g}(z)+V(z,g)$ is almost everywhere differentiable and
\[
\nabla_{z=y}\bigg(\log f_{\tau,g}(z)+V(z,g)+\frac{|z|^{2}}{2}\bigg)=\frac{x}{\tau}
\]
which implies that 
\[
\tau\nabla_{y}(\log f_{\tau,g}+V_{g})=x-y.
\]
Notice that combining the above equation with Lemma \ref{trans_bound} we obtain that $\log f_{\tau,g} + V_{g}$ has a uniformly bounded gradient.
Consequently, $i$ follows.  \\
\color{black}
\textbf{2.} Let us now show ii.  Using \eqref{E_stop_exchange} we obtain  
\[
\log f_{\tau}(x,g^{\prime})+V(x,g^{\prime})-[\log f_{\tau}(x,g)+V(x,g)]\geq h_{g'g}(x),
\]
for almost every $x$ in $\mathbb{R}^{d}.$ Interchanging $g$ and $g^{\prime}$ we obtain the opposite inequality and thus the desired identity. Here, once more we have used the fact that Proposition \ref{Barriers} and Lemma \ref{trans_bound} imply that $f_{\tau}>0$ and
$\bar{f}_{\tau}>0.$
\end{proof}

 As a direct consequence of the above proposition, we obtain the following result:
{\color{black}
\begin{corollary}\label{Sob}\textbf{(Sobolev regularity)} 
Let 
$\mu$, $\mu_\tau$, $\lambda,$ and $\Lambda$ be as in Proposition \ref{Barriers}, let $\lambda',\Lambda'$ be as in \eqref{eqn:WandV}, and let $\tau_{0}>0$ be as in Lemma \ref{pos_trans}. Then, for every  $\tau$ in $(0,\tau_{0})$, $f_{\tau,g}$ is contained in the weighted Sobolev space $W^{1,2}(\mathbb{R}^{d},e^{W})$ for every $g$ in $\mathcal{G}.$ Moreover, 
\begin{equation}\label{W1,2:0}
\sum_{g\in\mathcal{G}}\int_{\R^d}| f_{\tau}(x,g)|^{2}e^{W}dx\leq C_1 \sum_{g\in\mathcal{G}}\int_{\R^d}e^{-W(x)}dx,
\end{equation}
\begin{equation}\label{W1,2}
\tau\sum_{g\in\mathcal{G}}\int_{\R^d}|\nabla_x f_{\tau}(x,g)|^{2}e^{W}dx\leq C_2\big[\mathcal{E}(\mu)-\mathcal{E}(\mu_{\tau})+\tau\big],
\end{equation}
for some constant $C_1$ that only depends on $\lambda,\Lambda,\lambda', \Lambda'$, and a constant $C_2$ that only depends on $\lambda, \Lambda, \lambda', \Lambda'$ and the quantity
\[ [ \nabla_x V ]_{e^{-V}}:= \sum_g \int_{\R^d} |\nabla_x V(y,g)|^2 e^{-V(y,g)}dy.\] 
\end{corollary}
\begin{proof}
The fact that $f_{\tau,g}$ belongs to $L^2(\R^d, e^{W})$ follows from \eqref{Barriers}, \eqref{eqn:WandV}, and the fact that $e^{-W}$ was assumed to be integrable. 
	
Now, note that by optimality
\[
\mathcal{E}(\mu_{\tau})+C_{\tau}(\mu,\mu_{\tau})\leq \mathcal{E}(\mu).
\]
Consequently, using \eqref{EL_transport} and the definition of the transportation cost, we deduce that 
\[
\frac{\tau}{2}\sum_{g\in\mathcal{G}}\int|\nabla_x \log f_{\tau}(y,g)+\nabla_xV(y,g)|^2\bar{f}_{\tau}(y,g)dy \leq \mathcal{E}(\mu)-\mathcal{E}(\mu_\tau).
\]
Hence, using \eqref{gaussians}, \eqref{comparable} and \eqref{eqn:WandV} we obtain 
\[\tau \sum_{g\in\mathcal{G}}\int|\nabla_x f_\tau(y,g)|^{2}e^{-W}dy\leq C\bigg(\mathcal{E}(\mu)-\mathcal{E}(\mu_{\tau})+\tau\bigg),
\]
for some constant $C$ that only depends on $\lambda,\Lambda,\lambda', \Lambda'$ and the quantity $[ \nabla_x V ]_{e^{-V}}$.
\end{proof}}



\section{Convergence of the JKO scheme: Proof of Theorem \ref{main-result}}\label{JKO_proof_sec}\hspace{1mm}

Let us start by defining precisely the notion of weak solution to equation \eqref{reacd_{E}}. 

\begin{definition}
	\label{def:WeakSol}
We say that a weakly continuous curve of measures $\{ \mu_{t} \}_{t \geq 0}$ in $\mathcal{P}_{2}(\mathbb{\mathbb{R}}^{d}\times\mathcal{G})$ with associated probability density functions $\{ f(t, \cdot, \cdot)\}_{t \geq 0}$ is a weak solution with initial condition $f_0$ \eqref{reacd_{E}} if 
 \[
 f(0,x,g)=f_{0}(x,g), \quad \forall (x,g) \in\mathbb{R}^{d}\times\mathcal{G}
 \]
 and 
 \begin{align*}\begin{aligned} \sum_{g}\bigg( &\int_{\mathbb{R}^{d}}\zeta_g f_{g}(s,x)dx  -\int_{\mathbb{R}^{d}}\zeta_g f_{g}(r,x)dx \bigg)=\int_{r}^{s}\bigg( \sum_{g}\int_{\R^d}\big[\Delta_x\zeta_{g}-\langle\nabla_x V_{g},\nabla_x\zeta_g\rangle\big]f_{g}(t,x)dx\\
  & \hspace{1em}+\frac{1}{2}\sum_{g,g^{\prime}}\int_{\R^d}[\zeta_{g^{\prime}}-\zeta_{g}][\log f_{g^{\prime}}(t,x) +V_{g'}-\log f_{g}(t,x)-V_{g}]K(g,g^{\prime})e^{-W(x)}dx\bigg)dt,
 \end{aligned}
 \end{align*}
 for every $r,s,$ in $[0,\infty),$ and every test function $\zeta$ in $C^{\infty}_{c}(\mathbb{R}^{d}\times{\mathcal{G}}).$ \\
\end{definition}

 With all the preliminary results from section \ref{sec:Prelim} we can now proceed to the proof of Theorem \ref{main-result}.
{\color{black}
\begin{proof}[Proof of Theorem \ref{main-result}] \text{ }

\medskip	

\textbf{1. JKO scheme produces an approximate solution.} Let $f_{0}$ be an initial datum with finite energy $\mathcal{E}(f_0)<\infty$ satisfying \eqref{Z_03}. Let $\tau_{0}$ be as in Lemma \ref{pos_trans}, Proposition \ref{A_step_EL}, and  Corollary \ref{Sob}. Let $\tau\in(0,\tau_{0})$, and for every $n \in \mathbb{N}$ let $(\gamma_{n}^{\tau},h_{n}^{\tau})$  be the minizing pair of transporting $f_{n}^{\tau}$ into $f_{n+1}^{\tau}$, where the $f_n^\tau$ are the densities iteratively constructed as in \eqref{def:OurJKO}. Let $S_{n,g}^{\tau}$ be the optimal transport map associated to $\gamma_{n,g}^{\tau}$ as in \eqref{EL_plan}, and let $\bar{f}_{n}^\tau$ be the density of the measure $\pi_{2\sharp} \gamma_{n,g}^\tau$, i.e. the transported density. We recall that $\bar{f}_{n,g}^\tau$ can be written as
\[\bar{f}_{n,g}^\tau =f_{n+1,g}^{\tau} + \tau\sum_{g'}h^\tau_{n,gg'}K(g,g^{\prime})e^{-W}.\]
Notice that by iterating Proposition \ref{Barriers} we have 
\[\lambda e^{-V_g}\leq f_{n,g}^{\tau}\leq\Lambda e^{-V_g}\hspace{1em}\forall n\in\mathbb{N},
\]
and by Lemma \eqref{pos_trans}
\[ \frac{C}{1-\tau}<\frac{\bar{f}^\tau_{n,g}}{f^\tau_{n+1,g}}<C(1+\tau). \]
Finally, recall that the discrete time sequence $f_{n}^\tau$ can be extended to continuous time by setting
\[f^{\tau}(t):=f_{n+1}^{\tau}\hspace{1em}\mbox{for}\hspace{1em}t\in\big(n\tau,(n+1)\tau],\ 
\]
We will now show that the curve $t \mapsto f^{\tau }(t)$ can be interpreted as an approximate solution to equation \eqref{reacd_{E}}.

Let $\zeta\in C_{c}^{\infty}(\mathbb{R}^{d}\times\mathcal{G})$ be an arbitrary test function. Then, 
\begin{align}\begin{aligned}\label{pre0_weak}\int_{\mathbb{R}^{d}}\zeta_g\hspace{1mm}f{}_{n+1,g}^{\tau}(y)dy-\int_{\mathbb{R}^{d}}\zeta_gf_{n,g}^{\tau}(x)dx & =\int\zeta_g(y)d\gamma_{n,g}^{\tau}(x,y)-\int\zeta_g(x)d\gamma_{n,g}^{\tau}(x,y)\\
 & \hspace{1em}+\tau\sum_{g^{\prime}\in\mathcal{G}}\int_{\R^d}\zeta_{g}h_{n,gg^{\prime}}^{\tau}K(g^{\prime},g) e^{-W}dy.
\end{aligned}
\end{align}
Using the fundamental theorem of calculus and \eqref{EL_transport}, we deduce  
\begin{align*}\begin{aligned}\int_{\R^d \times \R^d} & \zeta_g(y)d\gamma_{n,g}^{\tau}(x,y)-\int_{\mathbb{R}^{d}\times\mathbb{R}^{d}}\zeta_g(x) d\gamma_{n,g}^{\tau}(x,y)\\
 & =\int_{{\mathbb{R}^{d}}\times{\mathbb{R}^{d}}}\big(\zeta_g(y)-\zeta_g(x)\big)\hspace{1mm}d\gamma_{n,g}^{\tau}(x,y)\\
 & =\int_{{\mathbb{R}^{d}}\times{\mathbb{R}^{d}}}\big(\zeta_g(y)-\zeta_g(S_{n,g}^{\tau}(y))\big)\hspace{1mm}\bar{f}_{n,g}^\tau(y)dy\\
 & =\int_{{\mathbb{R}^{d}}\times{\mathbb{R}^{d}}}(\zeta_g(y)-\zeta_g( S_{n,g}^{\tau}(y))f_{n+1,g}^{\tau}(y)\hspace{1mm}dy+R_{1}(\tau,n,g)\\
 & =-\int_{\mathbb{R}^{d}}\langle\nabla_x\zeta_g,S_{n,g}^{\tau}-Id\rangle f_{n+1,g}^{\tau}(y)\hspace{1mm}dy+R_{2}(\tau,n,g)+R_{1}(\tau,n,g)\\
 & =-\tau\int_{\mathbb{R}^{d}}\langle\nabla_x\zeta_g,\nabla_x f_{n+1,g}^{\tau}+f_{n+1,g}^{\tau}\nabla_x V_g\rangle\hspace{1mm}dy+R(\tau,n,g),\\
\\
\end{aligned}
\end{align*}
where the error term is given by 
\begin{align*}\begin{aligned}R(\tau,n,g) & =R_{1}(\tau,n,g)+R_{2}(\tau,n,g)\\
& =\tau\int_{{\mathbb{R}^{d}}}(\zeta_g-\zeta_g\circ S_{n,g}^{\tau})\sum_{g'}h_{n,gg^{\prime}}^{\tau}K(g,g^{\prime})e^{-W}\hspace{1mm}dy\\
& +\int_{\mathbb{R}^{d}}\int_{0}^{1}\bigg(\langle\nabla_x\zeta_g\circ((1-s)S_{n,g}^{\tau}+sId),Id-S_{n,g}^{\tau}\rangle-\langle\nabla_x\zeta_g,Id-S_{n,g}^{\tau}\rangle\bigg)f_{n+1,g}^{\tau}(y)ds dy.
\end{aligned}
\end{align*}
Plugging back in \eqref{pre0_weak} and using \eqref{E-LA1}, we deduce that
\begin{align}\begin{aligned}\label{pre_weak} \sum_{g}\int_{\mathbb{R}^{d}}&\zeta_g\hspace{1mm}f{}_{n+1,g}^{\tau}(y)dy -\sum_g\int_{\mathbb{R}^{d}}\zeta_g\hspace{1mm}f_{n,g}^{\tau}(x) dx=-\tau\sum_g\int_{\mathbb{R}^{d}}\langle\nabla_x\zeta_g,\nabla_x f_{n+1,g}^{\tau}+f_{n+1,g}^{\tau}\nabla_x V_g\rangle\hspace{1mm}dy\\
 &+\frac{\tau}{2}\sum_{g,g'}\int_{\R^d}(\zeta_{g}-\zeta_{g'})\big[\log f_{n+1}^\tau(x,g)+V(x,g)-\log f_{n+1}^\tau(x,g')-V(x,g')\big]K(g,g^{\prime})e^{-W}dy\\
 & +\sum_g R(\tau,n,g).
\end{aligned}
\end{align}

Let us now estimate the error terms. First, using \eqref{E-LA1} and the bounds on $f_{n+1,g}^\tau$ we can bound the transfer functions $h_{n,gg'}^\tau$ by a constant that only depends on $\lambda$ and $\Lambda$, and then use Lemma \ref{trans_bound} to obtain
\begin{equation}\label{error}
|R(\tau,n,g)|\leq C_{1} \lVert \nabla_x \zeta_g \rVert_{L^\infty(\R^d)} \bigg(\tau^{\frac{3}{2}}+\int_{\R^d}|Id-S_{n,g}^{\tau}|^{2}f_{n+1,g}^{\tau}(y)dy\bigg),
\end{equation}
for some constant $C_{1}:=C_{1}(\lambda,\Lambda).$ Now, from the fact that $f_{n+1,g}^{\tau}$ and $\bar{f}_{n,g}^\tau$ are comparable, and from the definition of $f_{n+1,g}^\tau$ and the transport cost $W_2^{\G, W, \tau}$ it follows that
\[
\sum_{g}\int|Id-S_{g,n}^{\tau}|^{2}f_{n+1,g}^{\tau}dy\leq C_2 \sum_{g}\int|Id-S_{g,n}^{\tau}|^{2}\bar{f}_{n,g}^{\tau}dy \leq C_2 \tau \big(\mathcal{E}(f_{n}^{\tau})-\mathcal{E}(f_{n+1}^{\tau})\big).
\]
where $C_2:=C_2(\lambda, \Lambda, \lambda', \Lambda')$. Thus, combining the above inequalities with \eqref{error} we deduce that 
\begin{align}
\begin{split}
\sum_{n=M}^{N-1}\sum_{g} \big|R(\tau,n,g)\big|& \leq C_{3} \max_{g} \lVert\nabla_x \zeta_g \rVert_{L^\infty(\R^d)} \bigg(\tau^{3/2}(N-M)+\tau\bigg[\mathcal{E}(f_{M}^{\tau})-\mathcal{E}(f_{N}^{\tau})\bigg]\bigg)
\\ & \leq C_{3} \max_{g} \lVert\nabla_x \zeta_g \rVert_{L^\infty(\R^d)} \bigg(\tau^{3/2}(N-M)+\tau \mathcal{E}(f_0) \bigg) ,
\end{split}
\label{eqn:ErrorJKO}
\end{align}
for all $M \leq N-1$, where  $C_{3}:=C_{3}(\lambda,\Lambda, \lambda', \Lambda').$

Let us now fix $0\leq r<s$. We add up \eqref{pre_weak} from $M=\lceil r\backslash\tau \rceil$ to $N-1=\lceil s\backslash\tau\rceil-1$ (assuming that $\tau$ is small enough so that $M \leq N-1$) to get that
\begin{align}\begin{aligned} &\sum_g\int_{\R^d}\zeta_g f_{g}^{\tau}(s,x)\hspace{1mm}dx  -\sum_g\int_{\mathbb{R}^{d}}\zeta_g f_{g}^{\tau}(r,x)\hspace{1mm}dx\\
& =\int_{\tau \lceil r\backslash\tau\rceil}^{\tau\lceil s\backslash\tau\rceil}\bigg(-\sum_g\int_{\R^d}\langle\nabla_x\zeta_g,\nabla_x f_{g}^{\tau}(t,x)+ f_{g}^{\tau}(t,x) \nabla_x V_g\rangle dx\\
& \hspace{1em}+\frac{1}{2}\int_{\R^d}\sum_{g,g'}(\zeta_{g^{\prime}}-\zeta_{g})\big[\log f_{g^{\prime}}^{\tau}(t,x)+V_{g^{\prime}}-\log f_{g}^{\tau}(t,x)-V_{g}\big] K(g,g^{\prime})e^{-W}dx\bigg)dt\\
& \hspace{1em}+\sum_{n=M}^{N-1}\sum_g R(\tau,n,g)\\
& =\int_{\tau\lceil r\backslash\tau\rceil}^{\tau \lceil s\backslash\tau\rceil}\bigg(\sum_g\int_{\mathbb{R}^{d}}\big[\Delta_x\zeta_g-\langle\nabla_x\zeta_g,\nabla_x V_g\rangle\big]f_{g}^{\tau}(t,x)\hspace{1mm}dx\\
& \hspace{1em}+\frac{1}{2}\int_{\R^d}\sum_{g,g'}(\zeta_{g^{\prime}}-\zeta_{g})\big[\log f_{g^{\prime}}^{\tau}(t,x)+V_{g^{\prime}}-\log f_{g}^{\tau}(t,x)-V_{g}\big] K(g,g^{\prime})e^{-W}dx\bigg)dt\\
&\hspace{1em} +\sum_{n=M}^{N-1}\sum_gR(\tau,n,g).
\end{aligned}
\label{eqn:auxJKOApprox}
\end{align}
From \eqref{eqn:ErrorJKO} it is clear that as $\tau\rightarrow 0$ the error term in the above expression vanishes. Therefore, if we can show that as $\tau \rightarrow 0$ (along a sequence) the curve $t \mapsto f^\tau(t)  $ converges to a limiting curve $t \mapsto f(t)$ which is weakly continuous, and that this convergence is strong enough so that in particular we can pass to the limit in all the terms in the above expression, then we will have shown that the curve $t \mapsto f(t) $ is indeed a weak solution to \eqref{reacd_{E}}.

\textbf{2. Compactness.}  Let us consider a sequence $\{\tau_{k} \}_{k}$ of positive numbers converging to zero. Without the loss of generality we can assume that $\tau_{k} \leq \tau_0$ for all $k$.
Our goal is to show that we can pass to the limit in \eqref{eqn:auxJKOApprox}. For this purpose we use the Aubin-Lions Theorem (see Theorem 5 in \cite{Simon}). We introduce some notation first.

Let us fix $t_F>0$. For $h>0$ we define the \textit{translates} 
\[T_{h}f^{\tau_{k}}(t):=f^{\tau_{k}}(t+h).\ 
\]
Also, for $R>0$ we let $U_{R}:=B_{R}\times\mathcal{G}$, where $B_R$ is the the open ball in $\R^d$ with radius $R$ centered at the origin. Let $p$ be a positive number such that $p>d+1.$ Consider the Sobolev spaces $W^{1,2}(B_R)$ and $W^{2,p}(B_R)$, and denote by $W^{-2,p}(U_{R})$ the dual of $W^{2,p}(B_R)$. Notice that 
\[  W^{1,2}(B_R) \hookrightarrow L^2(B_R)  \hookrightarrow  W^{-2,p}(B_R),  \]
where the first embedding is compact and the second one is continuous; notice also that $W^{2,p}(B_R)$ embeds continuously into $C^1(B_R)$. 

We show the following:
\begin{itemize}
	\item[a)] For every $g \in \G$, $\{ f^{\tau_k}_g\}_{k}$ is bounded in $L^2(0,t_F; W^{1,2}(B_R))$. 
	\item[b)] For every $g \in \G$, $ \lVert T_h f^{\tau_k}_g - f^{\tau_k}_g \rVert_{L^2(0, t_{F}- h; W^{-2,p}(B_R) )} \rightarrow 0 $ as $h \rightarrow 0$, uniformly for all $k$. 
\end{itemize}
Theorem 5 in \cite{Simon} will then imply that for every $g \in \G$, $\{ f^{\tau_k}_g\}_{k}$ is precompact in $L^2(0,t_F;L^2(B_R))$.

\textbf{2a.} Observe that by iterating the bounds from Corollary \ref{Sob} along $f_{n}^{\tau_{k}}$ we deduce that 
\begin{equation}\label{eq:fisher0}
\int_{B_{R}}| f^{\tau_{k}}_g(t,x)|^{2}\hspace{1mm}dx \leq C_4, \quad \forall t \geq 0, \forall k  \in \mathbb{N}
\end{equation}
as well as
\begin{equation}\label{eq:fisher}\int_{0}^{t_{F}}\bigg(\int_{B_{R}}|\nabla_x f^{\tau_{k}}_g(t,x)|^{2}\hspace{1mm}dx\bigg)dt \leq C_4( \EE(f_0)+ t_F), \quad \forall k \in \mathbb{N},
\end{equation}
where the constant $C_4$ depends only on $\lambda, \Lambda, \lambda', \Lambda', R,W, |\mathcal{G}|.$ From the above inequalities it follows that for every $g \in \G$, the sequence $\{f^{\tau_{k}}_g\}_{k \in \mathbb{N}}$ is bounded in $L^{2}(0,t_{F};W^{1,2}(B_{R}))$ (and also in $L^{2}(0,t_{F};L^2(B_{R}))$). Moreover, for every $t\geq 0$ the sequence $\{ f^{\tau_k}_g(t)\}_{k \in \mathbb{N}}$ is bounded in $L^2(B_R)$.

\textbf{2b.}
Let $h$ be smaller than $t_F$. For $t\in[0,t_F-h)$ set $N_k=\lceil{\frac{t+h}{\tau_{k}}}\rceil-1$  and
$M_k=\lceil {\frac{t}{\tau_{k}}} \rceil$. Notice that if $N_k <M_k$ then $T_{h}f_{g}^{\tau_{k}}(t) =f_{g}^{\tau_{k}}(t)$, and so we may assume that $M_k \leq N_k$. For any given $\zeta_g\in W^{2,p}(B_R),$ we have
\begin{align*}\int_{B_R}\zeta_{g}(x) &(T_{h}f_{g}^{\tau_{k}}(t,x)  -f_{g}^{\tau_{k}}(t,x))\hspace{1mm}dx\\
 & =\sum_{n=M_k}^{N_k}\int_{B_R}\zeta_{g}\hspace{1mm}f_{n+1,g}^{\tau_{k}}(x)dx-\int_{B_{R}}\zeta_{g}\hspace{1mm}f_{n,g}^{\tau_{k}}(x)dx\\
 & =\sum_{n=M_k}^{N_k}\int_{B_R \times B_R}(\zeta_{g}(y)-\zeta_{g}(x))\hspace{1mm}d\gamma_{n,g}^{\tau_{k}}(x,y)- \tau_{k} \sum_{g^{\prime}}\int_{B_R}\zeta_{g}h_{n,gg^{\prime}}^{\tau_{k}}e^{-W}dx\\
 & =\sum_{n=M_k}^{N_k}\int_{B_R \times B_R}\int_{0}^{1}\langle\nabla\zeta_{g}(x+s(y-x)),y-x\rangle\hspace{1mm}ds\hspace{1mm}d\gamma_{n,g}^{\tau_{k}}- \tau_{k}\sum_{g^{\prime}}\int_{B_R}\zeta_{g}h_{n,gg^{\prime}}^{\tau_{k}}e^{-W}\hspace{1mm}dx\\
 & \leq C_6\sum_{n=M_k}^{N_k}\lVert\zeta_g \rVert_{C^1(B_R)} \bigg(\int_{\mathbb{R}^{d}\times\mathbb{R}^{d}}|y-x|^{2}\hspace{1mm}d\gamma_{n,g}^{\tau_{k}}\bigg)^{\frac{1}{2}} +C_{5}\tau_{k}||\zeta_g||_{W^{2,p}(B_{R})}\\
 & \leq C_{7}||\zeta||_{W^{2,p}(B_{R})}\sum_{n=M_k}^{N_k}\bigg[\bigg(\int_{\mathbb{R}^{d}\times\mathbb{R}^{d}}|y-x|^{2}\hspace{1mm}d\gamma_{n,g}^{\tau_{k}}\bigg)^{\frac{1}{2}}+\tau_{k}\bigg].
\end{align*}
In the above the constant $C_5 $ depends only on $\lambda, \Lambda$ and $W$, $C_6$ depends only on $R$, and $C_7:= C_5+C_6$. We have used the fact that $W^{2,p}(B_{R})$ embeds continuously into $C^{1}(U_{R})$, and we have also used the bounds on the exchange function $h_{n}^{\tau_k}$ from \eqref{E-LA1} together with the lower and upper bounds for the density $f^{\tau_k}_{n,g}$. Consequently,
\begin{align}\begin{aligned}\label{equicont}||T_{h}f_g^{\tau_{k}}(t) & -f_g^{\tau_{k}}(t)||_{W^{-2,p}(B_R)}\\
 & =\sup_{||\zeta_g||_{W^{2,p}(B_R)}=1}\int_{B_R}\zeta_g\big(T_{h}f_{g}^{\tau_{k}}(t,y)-f_{g}^{\tau_{k}}(t,y)\big)\hspace{1mm}dy\\
 & \leq C_{7}\bigg(\tau_{k}(N_k-M_k)+\big(\tau_{k}(N_k-M_k)\big)^{\frac{1}{2}}\bigg(\sum_{n=M_k}^{N_k}\bigg[\bigg(\int_{\mathbb{R}^{d}\times\mathbb{R}^{d}}\frac{|y-x|^{2}}{\tau_k}\hspace{1mm}d\gamma_{n,g}^{\tau_{k}}\bigg)\bigg)^{\frac{1}{2}}\\
 & \leq C_{7}\bigg(h+\sqrt{h}\bigg[\sum_{n=M_k}^{N_k}\mathcal{E}(f_{n}^{\tau_{k}})-\mathcal{E}(f_{n+1}^{\tau_{k}})\bigg]^{\frac{1}{2}}\bigg)\\
 & \leq C_{7}\bigg(h+\sqrt{h}\bigg[\mathcal{E}(f_{M_k}^{\tau_{k}})-\mathcal{E}(f_{N_k+1}^{\tau_{k}})\bigg]^{1/2}\bigg)
 \\& \leq C_{8}\bigg(h+\sqrt{h}\bigg[ \EE(f_0)\bigg]^{1/2}\bigg)
\end{aligned}
\end{align}
Here, we used Jensen's inequality, and the definition of $f_{n,g}^{\tau_k}$. This shows \[||T_{h}f^{\tau_{k}}_g-f_g^{\tau_{k}}||_{L^{^2}(0,t_F-h;W^{-2,p}(B_{R}))}\rightarrow0, \quad\text{ as } h\rightarrow0,\] 
 uniformly in $k$.

From \textbf{2a)} and \textbf{2b)} it now follows that for every $g \in \G$, the sequence $\{ f_{g}^{\tau_k}\}_{k \in \mathbb{N}}$ is precompact in $L^2(0,t_F;L^2(B_R))$ (Theorem 5 in \cite{Simon}). In particular, there exist a subsequence of $\{\tau_k \}_{k}$ (which we do not relabel) and an element $f_g \in L^2(0,t_F; L^2(B_R))$ such that $f_g^{\tau_k} \rightarrow f_g$ as $k \rightarrow \infty$ in $L^2(0,t_F; L^2(B_R))$. On the other hand, from \eqref{eq:fisher0} and \eqref{eq:fisher} it follows that for almost every $t \in [0,t_F]$ the sequence $\{ f^{\tau^k}_g(t) \}_{k}$ is bounded in $W^{1,2}(B_R)$ and thus precompact in $L^2(B_R)$ and in $W^{-2,p}(B_R)$. We can then use this fact and \eqref{equicont} to conclude from Arzela-Ascoli theorem that $\{ f^{\tau_k}_g \}_{k}$ converges in $C(0,t_F; W^{-2,p}(B_R))$ (in fact in $C^{1/2-\veps}$ for any $\veps$) to $f_g$. Moreover, a standard diagonal argument sending $R \rightarrow \infty$ along a sequence, allows us to assume without the loss of generality, that for every $g \in \G$, $f^{\tau_k}_g\rightarrow f_g$ in  $L^2(0,t_F; L^2_{loc}(\R^d))$, as well as $f^{\tau_k}_g\rightarrow f_g$ in $C(0,t_F; W^{-2,p}_{loc}(\R^d))$, as $k \rightarrow \infty$.

\textbf{3. Properties of $t \in [0,t_F) \mapsto f(t)$.}
We claim that for every $t\in [0,t_F)$ we have 
\[  \lambda e^{-V_g} \leq f_{g}(t) \leq \Lambda e^{-V_g}. \] 
Indeed, notice that from \eqref{eq:fisher0} it follows that for every $t\in[0,t_F)$, the sequence $\{ f^{\tau_k}_g(t) \}_{k \in \mathbb{N}}$ is bounded in $L^2(B_R)$ (for every $R$) and thus it must have a weakly converging subsequence in $L^2(B_R)$. Due to the fact that $f^{\tau_k}_g\rightarrow f_g$ in $C(0,t_F; W^{-2,p}_{loc}(\R^d))$, said subsequence must converge weakly to $f_g(t)$ in $L^2(B_R)$. Since each of the $ f^{\tau_k}_g(t)$ satisfies the desired lower and upper bounds in $B_R$, it follows that $f_g(t)$ satisfies the same bounds in $B_R$. Since $R$ was arbitrary we conclude that $f_g(t)$ satisfies the desired bounds in the whole $\R^d$. 

Now we claim that for every $t \in [0,t_F)$ 
\[ \sum_{g}\int_{\R^d} f_g(t,x) dx =1 .\]
Indeed, this is a direct consequence of the lower and upper bounds obtained above and the fact that for every $t \in [0,t_F)$ $f^{\tau_k}_{g}(t)$ converges in $W^{-2,p}_{loc}(\R^d)$ towards $f_g(t)$. In particular, we conclude that the curve $t\in [0,t_F) \mapsto f(t, \cdot, \cdot)$ is indeed a curve of probability measures on $\R^d \times \G$. Moreover, the fact that $f_g \in C(0,t_F; W^{-2,p}_{loc}(\R^d))$ and the upper and lower bounds on the densities $f_g(t)$ imply that the curve $t\in [0,t_F) \mapsto f(t)$ (seen as a curve of probability measures) is weakly continuous (here interpreted as weak convergence of probability measures).

It remains to show that the curve is a weak solution to \eqref{reacd_{E}}.

\textbf{4. Weak solution of \eqref{reacd_{E}}}.
Let $\zeta\in C_{c}^{\infty}(\mathbb{R}^{d}\times\mathcal{G}),$ and let $0\leq r<s <t_F$. 

From the convergence $f^{\tau_k}_g\rightarrow f_g$ in $C(0,t_F; W^{-2,p}_{loc}(\R^d))$ it follows 
\begin{equation}
\int_{\R^d}\zeta_{g}f_{g}^{\tau_{k}}(s,x)\hspace{1mm}dx-\int\zeta_{g}f_{g}^{\tau_{k}}(r,x)\hspace{1mm}dx\rightarrow\int\zeta_{g}f_{g}(s,x)\hspace{1mm}dx-\int\zeta_{g}f_{g}(r,x)\hspace{1mm}dx.\
\label{eqn:AuxFinal-1} 
\end{equation}

Now, using the fact that 
$f^{\tau_{k}}_g(t)\rightarrow f_g(t)$ in $L^{2}_{loc}(\R^d)$ for almost every $t \in [0, t_F)$, and using the upper and lower bounds for $f_{g}^{\tau_k}(t)$ and $f_g(t)$ we conclude that
\begin{align}\begin{aligned}\int_{\mathbb{R}^{d}}\sum_{g^{\prime}\in\mathcal{G}}(\zeta_{g}-\zeta_{g^{\prime}})(\log f_{g}^{\tau_{k}}(t,x) + V_g & -[\log f_{g^{\prime}}^{\tau_{k}}(t,x)+V_{g^{\prime}}]e^{-W}\hspace{1mm}dx\\
& \rightarrow\int_{\mathbb{R}^{d}}\sum_{g^{\prime}\in\mathcal{G}}(\zeta_{g}-\zeta_{g^{\prime}})(\log f_{g}(t,x)+V_{g}-[\log f_{g^{\prime}}(t,x)+V_{g^{\prime}}])e^{-W}\hspace{1mm}dx,
\end{aligned}
\label{eqn:AuxFinal0}
\end{align}
for almost every $t\in [0,t_F)$, and
\begin{equation}
\int_{\R^d}\big[\Delta_x\zeta_g-\langle\nabla_x\zeta_g,\nabla_x V_g\rangle\big]f^{\tau_{k}}_g(t,x)\hspace{1mm}dx\rightarrow\int_{\R^d}\big[\Delta_x\zeta_g-\langle\nabla_x\zeta_g,\nabla_x V_g\rangle\big]f_g(t,x)\hspace{1mm}dx,
\label{eqn:AuxFinal1}
\end{equation}
for almost every $t \in [0,t_F)$. 


Now, from the upper and lower bounds on $f_g^{\tau_k}$, it follows that for every $t \in [0, t_F)$
\[\int|\sum_{g^{\prime}\in\mathcal{G}}(\zeta_{g}-\zeta_{g'})(\log f_{g}^{\tau_{k}}(t,x)+V_{g}-[\log f_{g'}^{\tau_{k}}(t,x)+V_{g'}])|e^{-W}\hspace{1mm}dx\leq C_{10}||\zeta||_{L^{\infty}(\mathbb{R}^{d})},\ 
\]
for a constant $C_{10}$ that only depends on $\lambda, \Lambda$, $|\G|$, $W$, and also 
\begin{align*}\begin{aligned}\int_{\R^d}|\big[\Delta_x\zeta_{g}-\langle\nabla_x\zeta_{g},\nabla_x V_{g}\rangle\big]|f_{g}^{\tau_{k}}(t,x)\hspace{1mm}dx & \leq  ||\Delta_x\zeta_g||_{L^{\infty}(\mathbb{R}^{d})} +C_{11}||\nabla_x\zeta_g||_{L^{\infty}(\mathbb{R}^{d}\times\mathcal{G})} ([\nabla_x V ]_{e^{-V}})^{1/2},
\end{aligned}
\end{align*}
for a constant $C_{11}$ that depends only on $\lambda, \Lambda, \lambda', \Lambda'$. We recall that $[\nabla_x V ]_{e^{-V}}$ is the quantity defined in Corollary \ref{Sob}.

Using the above two inequalities and \eqref{eqn:AuxFinal0}, \eqref{eqn:AuxFinal1}, we can invoke the dominated convergence theorem twice, and then combine with \eqref{eqn:AuxFinal-1} in order to conclude that we can pass to the limit in \eqref{eqn:auxJKOApprox}. From this it follows that $t\mapsto f(t)$ is a weak solution to \eqref{reacd_{E}}.

\end{proof}}

\section{Summary and discussion on applications}
\label{sec:OpenProblems}

In this paper we introduce two types of optimal transport problems in the semi-discrete setting and then study gradient flows of relative entropy functionals with respect to these semi-discrete transport costs. The first problem uses a dynamic formulation a la Benamou-Brenier, and a formal Riemannian structure can be associated to it. The Riemannian formalism is used to motivate systems of equations representing a gradient descent scheme for the minimization of a relative entropy functional; the Riemannian formalism can also be used to motivate accelerated methods for optimization. With the second optimal transport problem (the static one) we seek to more rigorously introduce the notion of gradient flow of the relative entropy functional by considering a minimizing movement scheme of the relative entropy with respect to this cost. Theorem \ref{main-result} establishes an equivalence between the gradient flow equation formally derived through the Riemannian formalism of the first transport cost and the rigorous definition of gradient flow using the minimizing movement scheme with respect to the second transport cost.

There are several theoretical research directions that emanate from our work. First, we believe that it is worth establishing a closer relationship between the two semi-discrete optimal transport problems introduced in the paper (the static and dynamic formulations). Secondly, it is worth emphasizing that our main result on the convergence of the minimizing movement scheme from section \ref{sec:main-results} towards the gradient flow heuristically motivated using the Riemannian formalism was only proved for mobilities that are independent of the mass exchanged among nodes in the graph. We believe that it is worth obtaining a more general result that justifies the connection between these two gradient flows even further.

	In the remainder of the paper we discuss some thoughts on the main application motivating this work.

	\subsection{From semi-discrete optimal transport to neural architecture search}

	In machine learning, a neural network is a graph $g$ (the architecture) whose nodes are arranged into layers with edges connecting nodes at different layers. A collection of free parameters (or weights) $x$ is associated with the nodes and edges in the graph. The network architecture $g$, together with the numerical values of its associated parameters $x$, determine a series of transformations that, when composed, define a mapping of input vectors (input data) into output vectors (labels). Training a given neural network $g$ essentially means tuning the corresponding parameters $x$ so as to achieve a small mismatch between predicted and observed outputs associated with given training inputs. In other words, the training of a neural network $g$ is the optimization of an objective function (a loss function) over the free parameters $x$.

	In \textit{neural architecture search} the goal is to find an architecture $g$ that, once trained, gives the best performance possible when predicting data outputs. From a simplistic perspective, this problem can be stated as solving: \begin{equation} \label{eq:semidiscreteopti}
	\minn_{(x,g)\in\R^{d}\times\mathcal{G}}V(x,g).
	\end{equation}
	where $V$ is thought of as a loss function that typically depends on observed data as well as on additional regularization terms. The variable $x$ (the parameters of a network) can be interpreted as a $\R^d$-valued vector (for $d$ large enough but fixed for simplicity), whereas $g$ can be interpreted as an element in a finite family of architectures $\mathcal{G}$ (which in principle may be quite large). In short, in neural architecture search the optimization is over both the architecture space $\G$ and over the parameters. The tensorized representation of the problem is certainly an oversimplification because, in reality, the parameters $x$ associated to an architecture $g$ do not have an obvious correspondence with the parameters of a different architecture $g'$ (and in fact their dimensions do not even have to match). We will not elaborate much further on this simplification and here we just limit ourselves to saying that while unreasonable when $\G$ is interpreted as the whole space of architectures, the tensorized representation of problem \ref{eq:semidiscreteopti} is useful when one restricts to a local graph of architectures where one has access to morphisms or correspondences between the parameters of different architectures (just like restricting the optimization of a function defined on a curved manifold to a local chart).

There is an enormous literature on neural architecture search methodologies and some of its applications  (see
	\cite{search_lit} for a brief overview on the subject), but essentially most methods found in the literature fall into two main groups. The first group builds on ideas from reinforcement learning as in \cite{ReinfLearning} which uses optimization tools like those described in \cite{REINFORCE}. The second major group is based on evolutionary algorithms \cite{NeuroEvolution, RegularizedEvolution}, where one specifies rules for merging and mutation of different architectures in search of ``stronger" architectures. A third type of methodology is the morphism-based hill-climbing strategy from \cite{Hillclimbing}. There, the authors propose an iterative scheme that alternates between training for a \textit{fixed time} a group of architectures that are determined by a morphism family and then moving in the space of architectures according to the relative performance improvement in such training time. In all the methodologies listed above, the main objective is to avoid the full training of multiple neural networks (something that would be computationally forbidding), either by building surrogate objective functions that are easier to evaluate, by training networks for a fixed amount of time, or by learning to predict which architectures are more likely to give better results. Many techniques in the literature are based on the above strategies. To name a few: \cite{ParameterSharing,ProgressiveNeural,TransferableArchitectures,ProgressiveNeural,Hyperparameter1, search_lit}.

 In this sprawling landscape of methods and techniques for neural architecture search, mathematicians can bring to the table principled ideas and structures for the development of new algorithms or the improvement of existing ones. Take for example the hill-climbing algorithm from \cite{Hillclimbing} where it is key to tune the amount of time that neural networks have to be trained for. It is intuitively clear that setting a fixed time for training is not ideal as in that way one forces all models to be treated the same regardless of their sizes or architectures. In our paper \cite{practical} we elaborate on this issue and propose a method where the training time of architectures is dynamically chosen as dictated by an evolving particle system that is inspired by the gradient flow perspective developed in this paper. All along, our intention was to give meaning to the notion of gradient descent for the optimization of an objective in the space $\R^d \times \G$, i.e. how to propose a gradient based method for semi-discrete optimization (with neural architecture search as main application in mind). As discussed in section \ref{sec:OTEuclidean}, in the Euclidean setting there is a well known connection between gradient flows in the space of measures and dynamics in the base space. In the semi-discrete setting, this connection is sought through particle methods. Particle methods are one way to project to the space $\R^d \times \G$ the dynamics that were lifted to the space of probability measures $\mathcal{P}(\R^d \times \G)$ to make sense of a gradient based scheme. In \cite{practical} all the nuances that have to be resolved to make this conceptual idea feasible for neural architecture search are discussed.

We hope that the theoretical, methodological and implementation questions briefly described here are able to motivate further research in the mathematics and computer science communities.

\red

\bibliography{ref}
\bibliographystyle{abbrv}	

\end{document}